\documentclass[12pt, leqno]{amsart}
\usepackage{amsmath, amssymb, latexsym, amsthm, amsfonts, mathtools, xparse}
\usepackage{mathrsfs}
\usepackage{url}
\usepackage{hyperref}
\usepackage{xcolor}  
\usepackage[protrusion=false]{microtype}
\usepackage{geometry}
\usepackage{parskip}
\usepackage{paralist}
\usepackage[T1]{fontenc}
\usepackage[backend=biber, style=alphabetic, giveninits, backref=true]{biblatex}

\mathtoolsset{showonlyrefs}
\numberwithin{equation}{section}
\hypersetup{
    colorlinks=true, 
    linktoc=all,     
    linkcolor=blue,  
}
\makeatletter
\renewcommand\subsubsection{\@secnumfont}{\bfseries}%
\renewcommand\subsubsection{\@startsection{subsubsection}{3}
  \z@{.5\linespacing\@plus.7\linespacing}{-.5em}%
  {\normalfont\bfseries}}  
\makeatother
\theoremstyle{plain}
\newtheorem{thm}{Theorem}[section]
\newtheorem{lem}[thm]{Lemma}
\newtheorem{prop}[thm]{Proposition}
\newtheorem{cor}[thm]{Corollary}
\newcommand{\thmref}[1]{Theorem~\ref{#1}}
\newcommand{\lemref}[1]{Lemma~\ref{#1}}
\newcommand{\propref}[1]{Proposition~\ref{#1}}

\theoremstyle{definition}
\newtheorem{rmk}[thm]{Remark}

\newtheorem{conjecture}[thm]{Conjecture}

\newtheorem{ques}[thm]{Question}

\newcommand{\rmkref}[1]{Remark~\ref{#1}}

\newcommand{\psmb}{\left( \begin{smallmatrix}}
\newcommand{\psme}{ \end{smallmatrix} \right)}
\newcommand{\smat}[4]{\left( \begin{smallmatrix} #1 & #2 \\ #3 & #4 \\ \end{smallmatrix} \right)}

\renewcommand*{\mod}{\operatorname{mod}}

\newcommand{\z}{\mathbb{Z}}
\newcommand{\h}{\mbb{H}}

\newcommand{\lan}{\langle}
\newcommand{\ran}{\rangle}

\newcommand{\q}{\quad}

\newcommand{\prodd}{\prod \nolimits}
\newcommand{\summ}{\sum \nolimits}
\newcommand*{\Q}{\mathbb{Q}}

\newcommand*{\complex}{\mathbb{C}}

\newcommand{\mbb}{\mathbb}
\newcommand{\mc}{\mathcal}
\newcommand{\mrm}{\mathrm}

\newcommand{\n}{\nonumber}
\newcommand{\mf}{\mathbf}

\newcommand{\sumn}{\sum \nolimits}

\newcommand*{\GL}[2]{\operatorname{GL}_{#1}(#2)}
\newcommand*{\SL}[2]{\operatorname{SL}_{#1}(#2)}

\newcommand{\sltwo}{\mrm{SL}_2(\mbb Z)}

\newcommand{\sptwo}{\mrm{Sp}_2( \z)}

\newcommand{\skk}{\mrm{SK}_{k+1}}
\newcommand{\skkn}{\mrm{SK}_{k+1}(N)}

\newcommand{\jk}{J_{k+1,1}^{cusp}}

\newcommand{\jkn}{J^{cusp}_{k+1,1}(N)}

\newcommand{\seln}{S_{k+1}(N)}

\newcommand*{\QEDB}{\hfill\ensuremath{\square}}

\newcommand*{\norm}[1]{\left\lVert#1\right\rVert}


\author{Pramath Anamby }
\address{School of Arts and Sciences\\ 
Ahmedabad University\\ 
Ahmedabad -- 380052, India.}
\email{pramath.anamby@gmail.com, pramath.anamby@ahduni.edu.in}

\author{Soumya Das}
\address{Department of Mathematics\\ 
Indian Institute of Science\\ 
Bengaluru -- 560012, India.}
\email{soumya@iisc.ac.in}

\date{}

\title[Pullbacks of Saito-Kurokawa lifts ]{Pullbacks of Saito-Kurokawa lifts of square-free levels, their non-vanishing and the $L^2$-mass}

\begin{document}

\begin{abstract}
     We obtain the full spectral decomposition of the pullback of a Saito-Kurokawa (SK) newform $F$ of odd, square-free level; and show that the projections onto the elements $\mf g \otimes \mf g$ of an  arithmetically orthogonalized old-basis are either zero or whose squares are given by the certain $\mrm{GL}(3)\times \mrm{GL}(2)$ central $L$-values $L(f\otimes \mrm{sym}^2 g, \frac{1}{2})$, where $F$ is the lift of the $\mrm{GL}(2)$ newform $f$ and $g$ is the newform underlying $\mf g$.
     Based on this, we work out a conjectural formula for the $L^2$-mass of the pullback of $F$ via the CFKRS heuristics, which becomes a weighted average (over $g$) of the central $L$-values. We show that on average over $f$, the main term predicted by the above heuristics matches with the actual main term. We also provide several results and sufficient conditions that ensure the non-vanishing of the pullbacks.
\end{abstract}

\maketitle

\tableofcontents

\section{Introduction}

The theory of pullback formulas of automorphic forms, first shown by Garrett for the triple product $L$-functions, now a well-developed subject, has found various applications and continues to fascinate mathematicians. Let $F_1$ and $F_2$ be automorphic forms on some classical groups $G_1$ and $G_2$, respectively, such that $G_2 \subset G_1$. The aim is to understand the period integrals $\lan  F_1 |_{G_2}, F_2 \ran$ as $F_2$ varies over cusp forms on $G_2$. This clearly helps us to understand the spectral decomposition of the restriction $F_1 |_{G_2}$.

Central to these formulas is the philosophy of the GGP (Gan-Gross-Prasad) conjectures (\cite{ggp1, ggp2}), originally stated for tempered automorphic representation on orthogonal groups. These, when generalized to other settings, including non-tempered representations (\cite{gg}, \cite{qiu}), predict when the period integrals mentioned above do not vanish. One of the important criterion being the non-vanishing of an associated Rankin-Selberg $L$-function, which we will describe in our particular case soon.

This paper is concerned with the `pullbacks' of Saito-Kurokawa lifts (SK lifts in short), which are recalled below.
SK lifts of level $N$, weight $k+1$ ($k$ odd) form a special subspace -- which we denote by $\skkn$ -- of Siegel modular forms of degree $2$, weight $k+1$ and (Hecke) congruence level $N$. These can be characterized in a variety of ways (see \cite{maass1}, \cite{EZ}, \cite{PS}, \cite{schmidt-SKlift}, \cite{das-anamby2}, \cite{saha-pitale} for an overview). For example, it is useful to know that the Fourier coefficients of (most of the) $F \in \skkn$ can be obtained  from those of a (uniquely determined, up to constants) half-integral weight modular form $h$ in Kohnen's plus space of level $4N$ by the rule
\begin{align}
    a_F(T)= \sum_{d|c(T), (d,N)=1}d^{k}c_h(D/d^2), \label{sk-1}
\end{align}
where $c(T)= \mrm{gcd}(n,r,m)$, and $T=\psmb n & r/2 \\ r/2 & m \psme$.

SK lifts of higher levels can be defined as the span of those away-from-level Siegel modular Hecke eigenforms forms of degree $2$, level $N$, and weight $k+1$ (denoted above by $\skkn$) whose spinor $L$ function factorizes as $\displaystyle L^{(N)}(F,s)= \zeta^{(N)}(s+1/2)\zeta^{(N)}(s-1/2)L^{(N)}(f,s)$, where the superscript $L^{(N)}$ denotes away from $N$ factors. Here $f \in S_{2k}(N)$ corresponds uniquely to $h$ (upto constants) via the Kohnen-Shimura-Shintani isomorphisms (preserving the corresponding newforms), see \cite{kohnen1982newforms, Kr}.
Such objects were studied in \cite{schmidt-AL} via representation theory where they arise as CAP representations of $\mrm{GSp}(4)$, and in \cite{das-anamby2} classically -- especially the theory of new and oldforms. The classical (Hecke equivariant) lifting construction from Jacobi forms of higher levels was first given by Ibukiyama \cite{Ibu-SK} along the lines of Eichler-Zagier \cite{EZ}.

More precisely, these were called EZI (Eichler-Zagier-Ibukiyama) lifts in \cite{das-anamby2} -- as they, in general, form a subspace of the space of SK lifts of level $N$. It was shown in \cite{das-anamby2} that EZI and SK lifts coincide on the newspace in $\skkn$.
Thus here there is no ambiguity in calling them by either of the terminologies. As mentioned before, these can be described nicely via an explicit Hecke-equivariant lifting map from $S_{2k}(N)$, see \cite{Ibu-SK} or Section~\ref{prelim}, and satisfy the so-called `Maa\ss' relations as described in \eqref{sk-1}.

In 2005, Ichino \cite{Ich05} proved a remarkable pullback formula for SK lifts of full level, relating certain period integrals attached to the pullback to certain central $L$-values attached to elliptic cusp forms. This confirmed the expectation from generalized GGP conjecture for the SK lifts. The level aspect of this formula, which is naturally more involved (cf. \cite[p.~115]{garrett}), was proved by Chen \cite{chen}, Pal--Vera-Piquero \cite{PV} at around the same time. Whereas \cite{chen} uses the method of Ichino via see-saw identities associated to the theta-liftings between orthogonal and symplectic groups, \cite{PV} is based on the work of Qiu \cite{qiu} via the factorization of certain $\mrm{SL}(2)$-periods associated to the pairs $(\tilde{\pi}, \tau)$ -- where $\pi, \tau$ denote automorphic representations associated to holomorphic newforms $f,g$ of square-free levels, and $\tilde{\pi}$ denotes the Shimura-Shintani lift of $\pi$ to the metaplectic double cover of $\mrm{SL}(2)$.
We will briefly describe these results below in the current context in order to put things into perspective with the results of this paper.

Let $F \in \skkn$ be a arbitrary. Consider its pullback $F^\circ \in \seln \otimes \seln$ to the diagonal, i.e., $F^\circ:=F \mid_{\h \times \h}$. One can thus write,
\begin{equation} \label{pullback-spectral}
F^\circ(\tau, \tau')=\sum_{\mf g\in \mc{B}_{k+1}(N)} \left\lan F^\circ, \frac{\mf g}{\norm{\mf g}}\otimes\frac{\mf g}{\norm{\mf g}}\right\ran \frac{\mf g(\tau)}{\norm{\mf g}}\frac{\mf g(\tau')}{\norm{\mf g}} \q \q (\tau, \tau' \in \h),
\end{equation}
where $\mc{B}_{k+1}(N)$ denotes an orthogonal basis for $S_{k+1}(N)$ -- which we will specify soon. The relevant periods mentioned in the beginning are thus $\displaystyle \left\lan F^\circ, \frac{\mf g}{\norm{\mf g}}\otimes\frac{\mf g}{\norm{\mf g}}\right\ran$. 

Now let $F$ be the SK lift of a newform $f\in S_{2k}^{new}(N)$. Sometimes we denote $F$ as $F_f$ to make this dependence explicit. Notice that both new and oldforms may appear on the RHS of \eqref{pullback-spectral}.
Let $g\in S_{k+1}^{new}(N_g)$ be a newform with $N_g|N$ and write $N=N_gM_g$. If the Atkin-Lehner eigenvalue $w_f(p)$ of $f$ is $+1$ for all primes $p|M_g$, then one has the following central value formula.
\begin{equation}\label{old-intro}
    \Lambda(f\otimes \mrm{sym}^2 g, \frac{1}{2})= \frac{2^{k+1-\omega(M_g)}M_g^{7/4}}{N} \prod_{p|N_g}(p+1)^2\frac{\langle f,f \rangle}{\langle h,h\rangle} 
 \frac{|\langle F_{f}^\circ, g \otimes g|B_{M_g}\rangle|^2}{\langle g,g\rangle ^2}.
\end{equation}
When $N=N_g=1$, this is the original result due to Ichino \cite{Ich05}.
When $M_g=1$, i.e. when $g$ is new, the above formula is the main result of \cite{chen} and \cite{PV} (where non-trivial nebentypus was considered). When $M_g>1$, i.e. when $g$ is old, \eqref{old-intro} is the main result from \cite{PVold} (after normalizing the $L$-function to have central point at $s=1/2$). 
Note that the above pullback formula differs from e.g. \cite{PV} by a power of $M_g$. This accounts for different normalizations of the  $L$--functions in question. In our paper, the functional equation of the completed $L$-function relates $s \to 1-s$. See subsection~\ref{Lfn-defs}.

One of the main questions here is to provide sufficient conditions for the (non) vanishing of $F^\circ$, or equivalently of $\lan F^\circ , F^\circ\ran $, which when properly normalized (see section \ref{L2massDef} for definition), is referred to as the
 $L^2$-mass of pullback of $F$. Thus one is inevitably led to consider its full spectral decomposition (see \eqref{pullback-spectral}) -- this means that one has to understand all the terms $\lan F_{f}^\circ, g' \otimes g''\ran$, where $g'  \otimes g''$ runs over all the old or newforms in $S_{k+1}(N) \otimes S_{k+1}(N)$. 

Via multiplicity-one for $\mrm{GL}(2)$, one can show that the periods can only survive when both $g',g''$ come from the same underlying newform, see Lemma~\ref{diag-period}. This leads to the question:

\begin{ques} \label{qs1}
    Are all the terms $\lan F_{f}^\circ, g' \times g'\ran$, with a suitable arithmetically orthogonalized choice of the set of oldforms $\{g'\}$, either vanish, or can be described by a formula akin to \eqref{old-intro}?
\end{ques}
One can easily check that by simply trying to exploit the invariance of the pullback under simple Hecke operators like $\mc W(p)$ etc. in the inner product does not lead to any inspiring result. It is also not clear a priori whether the RHS of \eqref{old-intro} vanishes for some old-class of $g$. As far as we could gather, an answer to Question~\eqref{qs1} was not even expected in the affirmative. Nor does there seem to exist an obvious modification of the proofs in, say, \cite{PV} to answer the Question~\ref{qs1}.\footnote{A. Pal (2019) and C. de Vera-Piquero (2024), personal communication.}

Thus, a deeper analysis is necessary -- and one of the main aims of this paper is to answer this question completely.

\section{Outline of the strategy and the main results}

\subsection{Choice of suitable coset representatives}
One of the most crucial ingredients behind the main result of this paper is the choice of a suitable set of coset representatives $\Gamma_0^{(2)}(N)\backslash \Gamma_0^{(2)}(M)$, where $M|N$, in such a way that it behaves reasonably well with respect to the pullback operation. Other sets of such coset representatives can be found in the work of Waibel \cite{waibel2020arithmetic} -- however, these does not seem to behave well with the pullbacks. Klein's work \cite{klein2004} is quite helpful (but still not exactly as we want) for us, and we are inspired by this approach. We provide a self-contained and very explicit description of these representatives, which we believe would be useful elsewhere as well, see subsection~\ref{coset-section}.

\subsection{Pullback formulae} 
In a nutshell, the idea to obtain the pullback formulae is to use the Fourier-Jacobi expansion of $F$, to understand its interaction with the pullback operation, and to unfold the inner product on the RHS of \eqref{old-intro} sufficiently in terms of the Fourier-Jacobi coefficient $\phi=\phi_{1,F}$ of $F$. Even though for level $1$ this is straightforward, we believe it is a different ball game altogether for higher levels.

Before proceeding further, we recall the Hecke-equivariant correspondences (for odd, square-free $N$)
\begin{align} \label{ezi-corr}
    f \longleftrightarrow h \longleftrightarrow \phi \longleftrightarrow F_f,
\end{align}
taking newforms to newforms, which will be used throughout the paper. See, for instance \cite[Figure~1]{das-anamby2}.

\subsubsection{Step~(i)}
Let $F$ be the SK lift of $f\in S_{2k}^{new}(N_f)$. We simply write $N$ for $N_f$. For any divisor $N_g|N$, let $g\in S_{k+1}^{new}(N_g)$ and $M_g=N/N_g$. 
Even though ultimately our aim is to obtain the analogue of \eqref{old-intro} for the orthogonal basis elements $g_\sigma \otimes g_\sigma$ for all $g$ and $\sigma$ (which are defined in subsection~\ref{deg1-old-basis}), it is enough to first obtain the same for the simpler elements $g \otimes g|B_M$ for all $M|M_g$. 

To this end, for any such $M$ as above, we first show that 
\begin{align} \label{bridge1}
  \lan F^\circ, g \otimes g|W_N(M)\ran  = \sum_{d_g|N_g^\infty}\sum_{\ell_g|M_g^\infty} (d_g\ell_g)^{-(k+1)/2} \lan \phi^\circ, g | W_N U(d_g\ell_g) W_N\ran \cdot \lan g_N|B_{d_g\ell_g}, g|B_{M} \ran.
\end{align} 
Here $\phi^\circ(\tau)=\phi(\tau,0)$ is the pullback of $\phi$ to $\h$, and is not necessarily an eigenform, $W_N(M)$ is the Atkin-Lehner operator at $M$ in level $N$, see \eqref{AL-ops} and $B_M$ is the level raising operator defined in \eqref{BdOpr}. The relation \eqref{bridge1} is obtained by considering the equivariance of suitable Hecke operators with the pullback operation, see subsection~\ref{cv-old-sec}.

\subsubsection{Step~(ii)}
Next, in subsection \ref{th:IPequiv-proof} we will show that
\begin{equation} \label{main-bridge-demo}
\lan F^\circ, g\otimes g|B_{M}\ran =  \mc L_1(M) \cdot \mc L_2(M_g/M) \cdot \lan g,g\ran \lan \phi^\circ, g|W_N(M_g)\ran,   
\end{equation}
where $\displaystyle \mc L_1(M)=\prodd_{p|M} \mc L_1(p;M)$ and $ \mc L_2(M_g/M)=\prodd_{p|(M_g/M)} \mc L_2(p;M_g/M)$ are given by certain explicit Euler products depending on $f$.
This step starts right from \eqref{bridge1}. The result in \eqref{main-bridge-demo} is obtained by a careful analysis of the equation \eqref{bridge1} through computations in the GL(2) Hecke algebra and eigenvalues of Hecke eigenforms, and forms the technical heart of this entire section. The quantity $\displaystyle g_N :=\sum_{(n,N)=1}a_g(n)q^n$ complicates matters to quite some extent, and the reader can find the treatments of the above two inner products appearing in \eqref{bridge1} in the subsections \ref{gN-sec} and \ref{cdphi-sec}, \ref{cmnl-sec} respectively. The Euler products $\mc L_j(-)$ arise in combination from the treatment of these two inner products.

There is a rather delicate technical point here: a priori the inner product involving $\phi^\circ$ in \eqref{bridge1} evaluates as a finite linear combination of terms of the form $\displaystyle \sum_{\ell}\lan \phi^\circ, g|W_N(\ell)\ran$ -- it takes some work to show that only the terms corresponding to $\ell=M_g$ survive. This is implied by Lemma \ref{phi0-ortho}. Once this is done, this Step follows from the above discussion.

\subsubsection{Step~(iii)}
In this Step, we start from \eqref{main-bridge-demo}, compute and simplify the Euler products appearing therein, and put them together to arrive at one of the main results of the paper.

\begin{thm}\label{th:IPequiv}
    Let $N$ be any odd square-free integer. Let $g\in S_{k+1}^{new}(N_g)$ be a newform and $M_g=N/N_g$. Then for any $M|M_g$,
    \begin{equation} \label{IPequiv-eq}
     \lan F^\circ, g\otimes g|B_{M}\ran = \frac{M^{(k+1)/2}\lambda_g(M_g/M)}{M_g^{(k+1)/2}}\prod_{p|(M_g/M)}\left(1+\frac{1}{p}\right)^{-1} \lan F^\circ, g\otimes g|B_{M_g}\ran.   
    \end{equation}
\end{thm}
Namely, the identity appearing in \eqref{main-bridge-demo}, when applied to $M$ and $M_g$, allows us to express both $\lan F^\circ, g\otimes g|B_{M}\ran$ and $ \lan F^\circ, g\otimes g|B_{M_g}\ran$ in terms of $\lan \phi^\circ, g|W_N(M_g)\ran$, which leads to Theorem \ref{th:IPequiv}. Its proof can be found in subsection~\ref{th:IPequiv-proof}.

\begin{rmk}
    We also note here that Theorem \ref{th:IPequiv} does not depend on the sign of $w_f(p)$. That is, it holds for any $g\in S_{k+1}^{new}(N_g)$ and $f\in S_{2k}^{new}(N)$.
\end{rmk}

\subsubsection{Step~(iv)}
In this Step, we conclude the proofs of all of the main results, thereby answering Question \ref{qs1} completely. It turns out that the results depend on the Atkin-Lehner eigenvalues of the lifted form $f$.
The following set of divisors of $N$ will play a decisive role. Define
\begin{equation} \label{Lf-def}
  \mc L_f:=\{ L|N: w_f(p)=+1\, \text{ for all }\, p|L, L\neq 1\}.
\end{equation}

\begin{thm}\label{thm:CVOldclass}
Let $g\in S_{k+1}^{new}(N_g)$, $f\in S_{2k}^{new}(N)$ be a newform and let the set $\mc L_f$ be as defined in \eqref{Lf-def}. Put $M_g=N/N_g$ and $F:= F_f$.  For any character $\sigma$ of $(\mbb Z/2\mbb Z)^{\omega(M_g)}$, let $g_\sigma\in \mc B_{k+1}(N_g, g)$ be as defined in \eqref{blg}, subsection~\ref{deg1-old-basis}. Then,
\begin{align} 
 \lan F^\circ, g_\sigma\otimes g_{\sigma'} \ran &=0 &\text{
    whenever }   \sigma \neq \sigma';    \label{cv-00}\\[1.0ex] 
    \lan F^\circ, g_\sigma\otimes g_{\sigma} \ran &=0 &\text{     whenever }   M_g \notin \mc L_f;     \label{cv-0}\\[1.0ex] 
 \Lambda(f\otimes \mrm{sym}^2 g, \frac{1}{2})= \frac{2^{k+1-\omega(M_g)}M_g^{7/4}}{N} & \prod_{p|N_g}(p+1)^2\frac{\langle f,f \rangle}{\langle h,h\rangle} 
 \frac{|\langle F^\circ, g_\sigma \otimes g_\sigma\rangle|^2}{\langle g_\sigma,g_\sigma\rangle ^2} &\text{ whenever }   M_g \in \mc L_f. \label{cv-1}  
\end{align}
\end{thm}

Note that the collection of $g_\sigma$, as $g$ and $\sigma$ vary, form an orthogonal basis of $S_{k+1}(N)$, see subsection~\ref{deg1-old-basis}.
Therefore \eqref{old-intro} (with $M_g=1$), Theorem~\ref{thm:CVOldclass}, when applied to \eqref{pullback-spectral}, gives the complete spectral decomposition of the pullback $F^\circ$, up to the signs $\pm \sqrt{\Lambda(f\otimes \mrm{sym}^2 g, \frac{1}{2})}$. 
Determination of the signs is a rather difficult problem. It is not even known for the Waldspurger's theorem.

The first assertion (i.e. \eqref{cv-00}) is straight-forward, see Lemma \ref{cv-00-lem}.
To prove the second part (i.e. \eqref{cv-0}), we really have to take a route which is similar in spirit, but rather orthogonal to the setting of the previous Steps. Namely, the reader may have noticed that almost all of the calculations in the above Steps~(i)--(iii) featured the $\mrm{GL}(2)$ Hecke algebra -- they were concerned with the modular form $g$, and the calculations were aimed to arrive at a relationship between the quantities $\lan F^\circ, g\otimes g|B_{M}\ran$ and $\lan \phi^\circ, g|W_N(M_g)\ran$. 

To prove  \eqref{cv-0} however, we really have to work with $F^\circ$, i.e. with the Hecke algebra of $\mrm{GSp}(4)$. Here it is used that $F$ is a newform. Our strategy is to study and utilize the interaction of a suitable set of coset representatives $\gamma \in \Gamma_0^{(2)}(N)\backslash \Gamma_0^{(2)}(M)$ ($M|N, \, M<N$) appearing in the equation $\mrm{Tr}^{(2)}(N,M)  = \sum_\gamma F| \gamma=0$ with the pullback operation. It is rather remarkable that this approach (in which it is enough to assume that $M=N/p$ for some prime $p$) -- i.e. considering the equation 
\begin{equation}
   \left\lan \left( F|\mrm{Tr}^{(2)}(N,N/p) \right)^\circ, g\otimes g|V \right\ran=0;
\end{equation}
where $g$ is of level $N/p$ as above and $V \in \{Id., W_N(p) \}$ -- unfolds into an equation of the form
\begin{equation}
    \left\lan (F|U_S(p))^\circ, g\otimes g|V \ran = \lambda_F(p)\lan F^\circ, g\otimes g|V\right\ran;
\end{equation}
with $U_S(p)$ being the degree $2$ Siegel-$U(p)$ operator. From this, \eqref{cv-0}
of Theorem \ref{thm:CVOldclass} follows easily. The entire Section \ref{cv-0-proof} is devoted to the proof of this.

Finally the third assertion (i.e., \eqref{cv-1}), follows by noting that the inner product appearing on the RHS of \eqref{IPequiv-eq} is given by a central value (cf. \eqref{old-intro}). The proof of this result thus builds upon all of the Steps~(i)--(iii) and occupies most of the Section~\ref{dMg}. The proof of Theorem~\ref{thm:CVOldclass} can be found in subsection~\ref{thm2.3proof}.

\begin{rmk} \label{gg-vanish}
    An interesting observation is that in \eqref{pullback-spectral}, no newform $g$ whose level does not divide $N$, can appear. Thus we must have $\displaystyle \lan F^\circ, g \otimes g  \ran =0$ for such $g$. If in addition, this period is given by a central $L$-value (this falls outside the purview of the GGP conjectures), then it would furnish an example of the vanishing of a  central $L$-value without any obstruction from the functional equation. Probably such a central-value formula does not exist. See the next remark, however. 
\end{rmk}

\begin{rmk}
    Like in the previously known instances of pullback formulae, Theorem~\ref{thm:CVOldclass} also imply special value results in the sense of Deligne, for the central $L$-values $L(f\otimes \mrm{sym}^2 g, \frac{1}{2})$ with $f,g$ as in the Theorem. More generally, one can prove special value results for $L(f\otimes \mrm{sym}^2 g, \frac{1}{2})$ with $f,g$ new of arbitrary square-free levels by other integral representations, see e.g. \cite{chen2019deligne}, where base change to Hilbert modular forms gave one such. This, and \rmkref{gg-vanish}, do not contradict each other.
\end{rmk}

\subsection{$L^2$-norm of pullbacks}\label{L2massDef}
The ambiguity with the signs in \eqref{pullback-spectral} goes away when one works with its $L^2$ or Petersson norm. Let $N$ be square-free and $k$ be odd. For a newform $F\in \skk^{new}(N)$, as introduced in the level $1$ case (cf. \cite{liu-young}), let us define the $L^2$-mass of the pullback of $F$ by
\begin{equation} \label{nf-def}
   N(F):= \frac{v_2\lan F^\circ, F^\circ\ran}{v_1^2\lan F, F\ran},
\end{equation}
where where $v_1=\mrm{vol.}(\SL{2}{\mbb Z}\backslash \h)$, $v_2=\mrm{vol.}(\sptwo\backslash \h_2)$ and $\left<F^\circ,F^\circ\right>$ is the product of Petersson inner products on $\Gamma_0(N)\backslash \h\times \Gamma_0(N)\backslash \h$ and it is given by
\begin{equation}
\left<F^\circ,F^\circ\right>=\frac{1}{[\SL{2}{\mbb Z}:\Gamma_0(N)]^2}\int_{\Gamma_0(N)\backslash \h}\int_{\Gamma_0(N)\backslash \h}\mid F\big( \begin{psmallmatrix}\tau& \\ &\tau'\end{psmallmatrix}\big)\mid^2 \mrm{Im}(\tau)^{k+1}\mrm{Im}(\tau')^{k+1}d\mu(\tau)d\mu(\tau'),
\end{equation}
where $d\mu(z)=y^{-2}dxdy$, if $z=x+iy, y>0$. Moreover, $\lan F,F \ran$ is also normalized by the index. With the above notations, we show the following.

\begin{thm}[Norm of the pullback] \label{norm-thm}
Let $k$ be odd and $N$ odd, square-free. For an SK lift $F_f\in \skk^{new}(N)$ of $f\in S_{2k}^{new}(N)$, we have
\begin{align} \label{norm-pullback}
    N(F_f)= \frac{v_2}{v_1^2}\cdot \frac{24 \pi}{kN}\frac{\zeta_{(N)}(2)^3}{\zeta_{(N)}(4)}\sum_{L|N}\frac{\prod_{p|L}(p+1)^2(1+w_f(p))^2}{L^{2}}\sum_{g\in \mc B_{k+1}^{new}(N/L)}\frac{L(f\otimes \mrm{sym}^2 g, \frac{1}{2})}{L(\mrm{sym}^2 f, 1) L(f,\frac{3}{2})}.
\end{align}
\end{thm}

The fact that the quantity $N(F_f)$ can be expressed in terms of central values of certain $L$-functions lends itself to analysis via the theory of $L$-functions using methods from analytic number theory. In particular the two important questions are the non-vanishing of $N(F_f)$ -- which is equivalent to the non-vanishing of $F_f^\circ$ -- and is directly related to the GGP conjectures; and then the size of the quantity $N(F_f)$. The reader may note that \eqref{norm-pullback} matches with the corresponding expression in level $1$ as given in \cite{liu-young}.

By Lapid's result \cite{lapid2003nonnegativity},  or directly from \eqref{cv-1}, we know that
all of the central $L$-values appearing on the RHS of \eqref{norm-pullback} are non-negative. Since it is generally believed that central values of $L$-functions should not vanish unless forced by the sign of its functional equation, this leads one to conjecture that all of the pullbacks $F_f^\circ$ do not vanish. In the present case, one can check (see subsection \ref{Lfn-defs}) that the sign is indeed $+1$ for all the central $L$-values in question. See subsection \ref{nonvanish-intro} below for more on this.
In that case we propose a conjecture on its $L^2$ size (normalized as in \eqref{nf-def}).

\begin{conjecture} \label{norm-conj}
    As $N\longrightarrow \infty$ along odd square-free integers, there exists some $\delta>0$ such that,
    \begin{align}
        N(F_f)= \frac{2\phi(N)\zeta_{(N)}(2)}{N}\prod_{p|N}\left(1+\frac{w_f(p)}{p}\right)^{-1}\sum_{d|N}\frac{\zeta_{(d)}(2)}{\phi(d)} \prod_{p|d}\left(1+\frac{1}{p}\right)^3(1+w_f(p))^2 + O(N^{-\delta}).\label{NFConj}
    \end{align}
    When $N=p$, a prime, the above can be simplified and reads $N(F_f)= 2 + O(p^{-\delta})$.
\end{conjecture}
The asymptotic formula in \eqref{NFConj} may be understood as follows. By the heuristics on the moments of central values of $L$-functions from \cite{conrey2005integral}, each $\sum_g$ in \eqref{norm-pullback} has an asymptotic formula with error $O((N/L)^{-\delta'})$ for some $\delta'>0$ which are absolute, but may differ for each $L$. These, put together, then implies the global power saving in $N$. Indeed our calculation of the main term follows this approach.

\subsection{Non-vanishing of the pullbacks} \label{nonvanish-intro}

It is possible to show in Conjecture \ref{norm-conj}, that the main  term is $\gg  1$ for all $N$ odd and square-free, by simply considering the contribution of the $d=1$ terms in \eqref{NFConj}, on account of the non-negativity of the other terms. Unconditionally however, it seems difficult to say anything concrete except when $k=1$. One can approach the RHS of \eqref{norm-pullback} via the standard method of approximate functional equation along with a plethora of other analytic techniques -- but the asymptotic formula, even on average over $f \in S_{2k}^{new}(N)$ still seems rather difficult  (see section \ref{sec:avgnonvan} for more details) -- it is under consideration by the authors.

The first implication of \thmref{norm-thm}, combined with the non-negativity of $L(f\otimes \mrm{sym}^2 g, \frac{1}{2})$ alluded to the above, is the following.
\begin{cor}
\begin{align}
    F_f^\circ=0 \iff L(f\otimes \mrm{sym}^2 g, \frac{1}{2})=0 \,{ \textbf{ for all }  } \, g\in \mc B_{k+1}^{new}(N/L) \text{ such that } w_f(p)=+1 \text{ for all } p|L.
\end{align}
\end{cor}
The full spectral decomposition thus gives a much stronger implication towards (non) vanishing of $F_f^\circ$ than that is implied by any single pullback formula of $F$ calculated against a fixed $g \otimes g$.

Among other applications of our main result, we mention the following. When $k=1$, we can invoke a non-trivial result due to Arakawa-B\"ocherer \cite{arakawa2003vanishing} to show that $F_f^\circ \neq 0$ for all newforms $f \in S_{2k}^{new}(N)$. For all such $f$, this shows the existence of newform $ g \in S_{k+1}^{new}(N)$ such that $L(f\otimes \mrm{sym}^2 g, \frac{1}{2}) \neq 0$ under mild conditions, improving upon a corresponding result \cite[Theorem 5.3]{arakawa2003vanishing}. See Remark \ref{rmk-bp2}. 

We can also show that on average the $L^2$-norm is not too small, see Proposition \ref{avg-nonzero}. This is obtained by passing on to the sup-norms of spaces of elliptic cusp forms.

On a more positive note, we can show (cf. Theorem \ref{Positiveprop}) that for a positive proportion (which is $>1/7$) of newforms $f$ in $ S_{2k}^{new}(p)$, the pullbacks $F_f^\circ $ do not vanish. This is shown by identifying the kernel of the pullback map with certain spaces of modular forms -- utilizing the work in \cite{arakawa1999vanishing}.
Our methods here also are non-analytic per se, and we provide some sufficient conditions under which the non-vanishing of $F_f^\circ$ is plausible (see  subsection \ref{kernel-nonvanish}). We present a simple if and only if condition for the non-vanishing of $F_f^\circ$.
\begin{thm} \label{nonv-algo}
    Let $f \in S_{2k}^{new}(N)$ be a newform with $k$ odd, $N$ odd and square-free. Let $h \in S^+_{k+1/2}(4N)$ correspond to $f$ under the Kohnen-Shimura-Shintani correspondences. Put $H=h \cdot \vartheta_0$, where $\displaystyle \vartheta_0(\tau) = \sum_{n \in \z} q^{n^2}$, and write its Fourier expansion $\displaystyle H(\tau) = \sum_{n=1}^\infty a_H(n) q^n$. Then $H \in S_{k+1}(4N)$, and
    \begin{align} \label{algo-cond}
        F_f^\circ \neq 0 \iff N(F_f) \neq 0 \iff a_H(n) \neq 0 \text{ for some even } n.
    \end{align}
    Necessarily, $4|n$. Also \eqref{algo-cond} has to be checked only for finitely many values of $n=4m$, viz. $m \le \mc K:=k[\sltwo: \Gamma_0(N)] / 8$.
    \end{thm}
Of course, this Theorem is equivalent to the conditions that one gets by equating the Fourier coefficients of $F_f^\circ$ to zero, but we believe that the above formulation is much simpler to handle both theoretically and computationally. We refer the reader to subsection~\ref{sec:nonv-algo} for handy sufficient conditions 
which imply \eqref{algo-cond}. We also provide some equivalent conditions for non-vanishing akin to that in \eqref{algo-cond} in weight $k$ and level $16N$. See subsection~\ref{k16n}.

Of course, one might want to consider the unitary counterpart to this situation -- i.e., consider Hermitian modular forms and their pullbacks. These lead to central $L$-values of Rankin-Selberg $L$-functions of elliptic newforms, and will be forthcoming shortly \cite{das-singh1, das-singh2}.

In an Appendix (section \ref{appendix}) we calculate the actual main term of $\sum_f N(F_f)$ and show that it matches with the corresponding main term predicted by the Conjecture~\ref{norm-conj}.

\subsection*{Acknowledgments}
{\footnotesize
We sincerely thank A. Pal and C. de Vera-Piquero for kindly and patiently explaining to us some of the difficulties pertaining to the pullback formulae treated in this paper. Initial thoughts about the project date back to antiquity. They found direction through some stimulating discussions in 2019 between A. Pal and S.D. at Universit\"at Mannheim, where S.D. was supported by an Alexander von Humboldt Fellowship hosted by Prof. S. B\"ocherer. S.D. thanks the Alexander von Humboldt Foundation,
IISc. Bangalore, UGC Centre for Advanced Studies, DST India for financial support.

P.A. held an NBHM Postdoctoral Fellowship at IISER, Pune during the project's initial phase and thanks the National Board for Higher Mathematics (NBHM), DAE, India, for the financial support provided through this fellowship.
}

\section{Notation and set-up} \label{prelim}
In this paper, we will use standard notation, some of which are collected below, and the rest will be introduced as and when it is necessary. We note that $A \ll B$ and $A=O(B)$ are the Vinogradov and Landau notations, respectively. By $A\asymp B$ we mean that there exists a constant $c\ge 1$ such that $B/c\le A\le cB$. Any subscripts under them (e.g., $A\ll_n B$) indicate the dependence of any implicit constants on those parameters. Throughout, $\epsilon$ will denote a small positive number, which can vary from one line to another. We use $\z, \Q, \mbb R$, and $\complex$ to denote the integers,
rationals, reals, and complex numbers, respectively. For any integer $N$ and $s\in \mbb C$, we use throughout the paper, the notation 
\begin{align}
\zeta_{(N)}(s)=\prodd_{p|N}\left(1-p^{-s} \right)^{-1},  \text{ and } \zeta^{(N)}(s)=\prodd_{p\nmid N}\left(1-p^{-s} \right)^{-1}.
\end{align}
We use the standard diagonal embedding of $\sltwo\times \sltwo$ in $\sptwo$, given by $\smat{a_1}{b_1}{c_1}{d_1}\times \smat{a_2}{b_2}{c_2}{d_2}\mapsto \begin{psmallmatrix}
    a_1 & 0 & b_1 &0\\ 0 & a_2 & 0 & b_2\\ c_1 & 0 & d_1 & 0\\ 0 & c_2 & 0 & d_2
\end{psmallmatrix}.$
In the article, $N$ denotes an odd and square-free integer. For any positive integer $\ell$, $S_\ell(N)$ and $S_\ell^{(2)}(N)$ denote the spaces of cusp forms of degree $1$ and $2$ respectively.

\subsection{Some operators on modular forms}
In this subsection, we summarize, with notation, the various Hecke operators which would be used extensively later.

\subsubsection{Hecke operators} Let $(\ell, N)=1$. For any $m\in \mbb N$ and $\ell|m$, we define $T_N(\ell, m)$ to be the double coset operator $\Gamma_0(N)\smat{\ell}{0}{0}{m}\Gamma_0(N)$. We then define
\begin{equation}
    T_N(n)=\sum_{\ell m =n} T_N(\ell, m).
\end{equation}
For any eigenfunction $f$ of $T_N(n)$ and $T_N(\ell, m)$, let $\lambda_f(n)$ and $\lambda_f(\ell, m)$ denote the corresponding eigenvalues.

\subsubsection{Level raising operator} Let $d\in \mbb N$. The level raising operator $B_d: S_{\ell}(N)\longrightarrow S_\ell(Nd)$ is defined as
\begin{equation}\label{BdOpr}
    (f|B_d)(z):= d^{\ell/2}f(dz).
\end{equation}

\subsubsection{Atkin-Lehner Operators} \label{AL-ops}
For any $d|N$, the Atkin-Lehner operator is defined as below.
\begin{enumerate}
\item Atkin-Lehner operator acting on $S_{\ell}(N)$:
\begin{equation} \label{wnd-ell}
W_N(d)= \smat{da}{1}{Nb}{d}\q \text{ where } a,b\in \mbb Z \text{ such that } da-Nb/d=1.
\end{equation}

\item Atkin-Lehner operator acting on $S_\ell^{(2)}(N)$:
\begin{equation} \label{wnd-s}
\mc W_N(d)=\smat{daI_2}{I_2}{NbI_2}{dI_2} \q \text{ where } a,b\in \mbb Z \text{ such that } da-Nb/d=1.
\end{equation}
\end{enumerate} 

Note that $\mc W_N(d) = W_N(d) \times W_N(d)$. It is readily checked that $\mc W_N(d)$ normalizes $\Gamma_0^{(2)}(N)$ and as an operator on $S_{k}^{(2)}(N)$, acts by involution. We also note that $\mc W_N(N)=\smat{0}{-I_2}{NI_2}{0}$.

We have the following lemma relating $W_N(d)$ and $B_d$, which will be used frequently in the paper.
\begin{lem}\label{lem:WBCom}
 Let $M|N$ and $d|(N/M)$. Then there exists $\gamma\in \Gamma_0(M)$ such that $W_N(d)= \gamma B_d$.
\end{lem}
\begin{proof}
Let $W_N(d)$ be as in \eqref{wnd-ell}. Then we have
\begin{equation}
 \smat{da}{1}{Nb}{d}\cdot \smat{1/d}{0}{0}{1}= \smat{a}{1}{Nb/d}{d}=:\gamma.   
\end{equation}
Since $M|(N/d)$, we have that $\gamma\in \Gamma_0(M)$ and this completes the proof.
\end{proof}

\subsection{The Jacobi $V_m$ operator}\label{sec:VmOP}
The operators $V_m$ are defined on Jacobi forms of index $1$ and level $N$ as follows (even though they can be defined for any index). For $\phi \in \jkn$ define (cf. \cite[\S~4 (2)]{EZ}, \cite[Section 3]{Ibu-SK})
\begin{align} \label{vmdef}
    (\phi|V_m)(\tau,z)= m^{k-1} \sum_{\gamma } (c \tau+d)^{-k} e(-m c z^2/(c \tau+d)^2) \phi \big(\frac{a \tau +b}{c \tau+d}, \frac{m z}{c \tau+d} \big),
\end{align}
where $\gamma =\smat{a}{b}{c}{d} $ runs over a set of representatives $\Gamma_0(N) \backslash M_{2,m}(\z)$. Here $M_{2,m}(\z)$ denotes the set of $2\times 2$ integral matrices $\smat{a}{b}{c}{d}$ with determinant $m$ and $(a,N)=1$. The Fourier expansion of $\phi|V_m$ is as follows.
\begin{equation}
    (\phi|V_m)(\tau, z)= \sum_{n,r} \Big( \sum_{a |(n,r,m), \, (a,N)=1} \,  a^{k-1} c_{\phi} (\frac{nm}{a^2}, \frac{r}{a}) \Big) \, e(n \tau+rz).
\end{equation}

\subsection{An orthogonal basis for oldspace in degree 1} \label{deg1-old-basis}
In this section, we write down an arithmetic orthogonal basis for $S_{k+1}(N)$. For the new-space this is clear, so the issue is to handle the old-space.
For square-free levels, it is fortunately possible to describe this in a convenient manner, as given below. 
Let $S_\ell^{old}(N)\subset S_\ell(N)$ denote the subspace consisting of old forms. Then an orthogonal basis for $S_\ell^{old}(N)$ can be written as follows (see \cite[Section 2.1]{PY}).
\begin{equation}\label{oldbasis}
    \mc B_\ell^{old}(N):=\bigcup\nolimits_{LM=N, L>1}\{\mc B_{\ell}(L,g): g\in \mc B_{\ell}^{new}(M) \},
\end{equation}
where $\mc B_\ell(L,g)$ denotes an orthogonal basis for the subspace $S_\ell(L,g):=\mrm{span}\{ g|B_d: d|L\}$. 

For a group $G$, its dual is denoted by $G^{\wedge}$.
Here we tacitly assume that the set of divisors of $L$ has been endowed with the group structure via $d_1 * d_2 = d_1 d_2/(d_1,d_2)^2$ -- which induces an isomorphism of the set of divisors of $L$ onto $  (\mathbb{Z}/2\mathbb{Z})^{\omega(L)} $ via the map $d \mapsto (\epsilon_p)_{p|L}$, where $\epsilon_p = -1$ if $p|d$ and $+1$ otherwise. It is readily checked the $*$ operation respects the above map. Thus we can, and will tacitly identify each character $\sigma$ with $\prod_{p|L} \varepsilon_p$ where $\varepsilon_p$
vary over all sequences of $\varepsilon_p \in \{\pm 1\}$ with $p|L$. For $p|L$, therefore, $\sigma(p)\in\{+1,-1\}$ and $\sigma(d)=\prod_{p|d}\sigma(p)$. With the above conventions, one can take (see \cite{PY})
\begin{equation} \label{blg}
    \mc B_\ell(L,g)=\Big\{g_\sigma:=\sum\nolimits_{d|L}\sigma(d)g|W_N(d) \mid \sigma \in  \left( (\mathbb{Z}/2\mathbb{Z})^{\omega(L)} \right)^{\wedge} \Big\}.
\end{equation}

\subsection{The $\mrm{GL}_3\times \mrm{GL}_2$ $L$-function}
\label{Lfn-defs}
Let $f\in S_{2k}^{new}(N)$ and $g\in S_{k+1}^{new}(N_g)$ with $N_g|N$. 

For $\Re(s)>1$, we have
\begin{equation}
    L(f\otimes \mrm{sym}^2 g, s)= \prod_{p|N}L_p(f\otimes \mrm{sym}^2 g, s) \sum_{n,m\ge 1, (nm, N)=1}\frac{\lambda_f(n)A_g(n,m)}{(nm^2)^s},
\end{equation}
where for $(mn,N_g)=1$,
\begin{align}\label{GL3HeckRel}
    A_g(n,m)= \sum_{d|(n,m)}\mu(d) A_g(n/d,1) A_g(m/d,1); \q A_g(r,1)=\sum_{ab^2=r}\lambda_g(a^2).
\end{align}
For any $p|N_g$ we let
\begin{equation}
    A_g(p^t, 1):=\sum_{i=0}^{t}p^{i-t}=\sum_{d|p^t}\lambda_g(p^{2t}/d^2).
\end{equation}
The last equality follows by noting that $\lambda_g(p^2)=\lambda_g(p)^2=p^{-1}$ when $p|N_g$ \cite{ILS}. It follows by multiplicativity that for any $r| N_g^\infty$,
\begin{align}
   A_g(r, 1)  =  \sum_{r_1r_2=r} \frac{1}{r_2}.
\end{align}
From \eqref{complexfactorization} below and \cite{Böcherer_Schulze–Pillot_1996}, we can write
\begin{align}
    L_p(f\otimes \mrm{sym}^2g, s)^{-1}&= (1-\lambda_f(p)p^{-s})(1-\lambda_f(p)p^{-s-1}) & (p|N_g)\\
    L_p(f\otimes \mrm{sym}^2g, s)^{-1}&= (1-\lambda_f(p)\alpha_g(p)^2p^{-s})(1-\lambda_f(p)\beta_g(p)^2p^{-s})(1-\lambda_f(p)p^{-s})  &(p|(N/N_g).
\end{align}
When $p|(N/N_g)$, we have
\begin{equation}
 L_p(f\otimes \mrm{sym}^2g, s)= \sum_{t\ge 0}\lambda_f(p^t)\left(\sum_{i=0}^t\sum_{j=0}^i\alpha_g(p)^{2j}\beta_g(p)^{2i-2j}  \right)  p^{-ts}. 
\end{equation}
Thus for $p|(N/N_g)$, we can write
\begin{equation}
    L_p(f\otimes \mrm{sym}^2g, s)= \sum_{t\ge 0}\frac{\lambda_f(p^t)A_g(p^t,1)}{p^{ts}}.
\end{equation}
This gives us the following Dirichlet series for $L(f\otimes \mrm{sym}^2 g, s)$.
\begin{equation}\label{Drichletseries}
    L(f\otimes \mrm{sym}^2 g, s)=\sum_{r|N_g^\infty}\frac{\lambda_f(r)A_g(r,1)}{r^s}\cdot \underset{\ell|M_g^\infty, (nm,N)=1}{\sum_{n,m,\ell}} \frac{\lambda_f(n\ell)A_g(n\ell, m)}{(n\ell m^2)^s}.
\end{equation}
But $(r,n \ell)=1$ and $A(nn',mm')=A(n,m)A(n',m')$ if $(m,m')=1=(n,n')$. Therefore we can write
\begin{equation}
     L(f\otimes \mrm{sym}^2 g, s)= \sum_{n,m,\ell,r} \frac{\lambda_f(n\ell r)A_g(n\ell r, m)}{(n\ell r m^2)^s} = \underset{(m,N)=1}{\sum_{a,m}}\frac{  \lambda_f(a)A_g(a, m)}{(a m^2)^s}.
\end{equation}
Let us write $\gamma_{f\otimes \mrm{sym}^2 g}(s):=\Gamma_{\mbb C}(s+2k-3/2)\Gamma_{\mbb C}(s+k-1/2)\Gamma_{\mbb C}(s+1/2)$ and define the completed $L$-function $\Lambda(f\otimes \mrm{sym}^2 g, s)$ as
\begin{equation}
 \Lambda(f\otimes \mrm{sym}^2 g, s)=N^{3s/2} N_g^{s/2} \gamma_{f\otimes \mrm{sym}^2 g}(s)L(f\otimes \mrm{sym}^2 g, s).    
\end{equation}
\begin{rmk}
    We note here the difference in the expression for the completed $L$-function as compared to those in \cite{PV} and \cite{PVold}. This is due to the renormalization to get the central point as $s=1/2$.
\end{rmk}
The functional equation of $L(f\otimes \mrm{sym}^2 g, s)$ can be read off from the corresponding functional equation of the triple product $L$-function $L(f\otimes g \otimes g)$, which factorizes as
\begin{equation}\label{complexfactorization}
 L(f \otimes g \otimes g, s)= L(f,s)L(f \otimes \mathrm{Sym}^2(g),s),
\end{equation}
see e.g., \cite{watson2002rankin} or \cite[Remark 1.1 (4)]{chen2019deligne}. The proof of the functional equation for $L(f\otimes g \otimes g)$ can be found in \cite[Theorem 4.3]{Böcherer_Schulze–Pillot_1996}. It says (after a rearrangement to get the central point as $s=1/2$) that,
\begin{equation} \label{fnleq}
\Lambda(f\otimes \mrm{sym}^2 g, s)= \epsilon( f\otimes \mrm{sym}^2 g) \Lambda(f\otimes \mrm{sym}^2 g, 1-s),
\end{equation}
where $\epsilon(f\otimes \mrm{sym}^2 g)$ is the root-number of $f\otimes \mrm{sym}^2 g$. The reader can verify this from e.g., \cite[Section 7]{chen2019deligne}.

We provide some details about the  $\epsilon$-factors and the functional equation, as they do not seem to be available all in one place in the literature. From \eqref{complexfactorization} we get that $ \Lambda(f \otimes g \otimes g, s)= \Lambda(f,s)\Lambda(f \otimes \mathrm{Sym}^2(g),s)$ and  since the Gamma factors match,  we get $\epsilon(f\otimes g\otimes g)= \epsilon( f\otimes \mrm{sym}^2 g) \epsilon(f) $. From \cite[Theorem 4.3]{Böcherer_Schulze–Pillot_1996}),  we also have  $\epsilon(f\otimes g\otimes g)=-\prod_{p|N} w_f(p)\prod_{p|N_g} w_g(p)^2 = -\prod_{p|N} w_f(p) $. Since $f$ is of weight $2k$ with $k$ odd, $\epsilon(f)=i^{2k} \prod_{p|N} w_f(p)= -\prod_{p|N}w_f(p)$. Thus we get that $\epsilon( f\otimes \mrm{sym}^2 g)=1$.

\subsection{Pullbacks of Saito-Kurokawa lifts with level.}
Let $k$ be odd. We start from the expression
\begin{align} \label{z=0}
F^\circ(\tau,\tau'):=F(\begin{psmallmatrix} \tau & 0 \\ 0 & \tau' \end{psmallmatrix}) = \sum_{g,g' \in \mc{B}_{k+1}(N)} c_{g,g'} \, g(\tau) \otimes g'(\tau'),
\end{align}
where $c_{g,g'} \in \mathbb{C}$ and $g,g'$ run over an orthogonal basis $\mc{B}_{k+1}(N)$ of $S_{k+1}(N)$. We start with an SK newform of level $N$. In \cite{das-anamby2}, it is shown that the SK newforms of square-free level $N$ and weight $k+1$ are in one-one correspondence with those in $\jk(N)$ via the EZI map and hence to those in $S_{k+1/2}^{+, new}(4N)$ via the EZ-isomorphism (cf. \cite[Theorem~5.4]{MR}). In what follows, we therefore use the relation $h \leftrightarrow F$ from \eqref{ezi-corr}.

The following lemma is the analogue of \cite[Lemma~1.1]{Ich05} in the level aspect -- it allows us to show that the spectral expansion of $F^\circ$ in \eqref{z=0} is supported only on the diagonal. We include it because it is a bit delicate due to the involvement of the level.
Also, since this result is a consequence of multiplicity-one on $\GL{}{2}$, fortunately, in the lemma below we need to only consider $T(q)=T_N(q)$ for primes $q$ such that $q \nmid N$. 

\begin{lem}  \label{diag-period}
    Let $h\in S_{k+1/2}^{+, new}(4N)$ and $F\in \skk^{new}(N)$ be the associated Saito-Kurokawa lift. Then we have
    \begin{equation}
        (T(q)\otimes id) (F^\circ) = (id \otimes T(q)) (F^\circ)
    \end{equation}
    for all primes $q$ with $(q,N)=1$.
\end{lem}
\begin{proof}
    We first note that from the Fourier expansion of $F$,
    \begin{equation}
        F^\circ(\tau, \tau')= \sum_{n,m}\underset{r^2<4mn}{\sum_{r\in \mbb Z}}\underset{(d,N)=1}{\sum_{d|(n,r,m)}}d^{k-1} c_h(\frac{4mn-r^2}{d^2} )e(n\tau)e(m\tau').
    \end{equation}
From the definition of the Hecke operator $T(q)$ ($q\nmid N$) acting on $S_{k}(N)$, we see that
{\allowdisplaybreaks
\begin{align}
    (T(q)\otimes id) (F^\circ)&= \sum_{n,m}\underset{r^2<4mnq}{\sum_{r\in \mbb Z}}\Big(\underset{(d,N)=1}{\sum_{d|(nq,r,m)}}d^{k-1} c_h(\frac{4mnq-r^2}{d^2} )\Big)e(m\tau')\\
    &+q^{k-1} \sum_{n,m}\underset{r^2<4mn}{\sum_{r\in \mbb Z}}\Big(\underset{(d,N)=1}{\sum_{d|(n,r,m)}}d^{k-1} c_h(\frac{4mn-r^2}{d^2} )\Big)e(nq\tau)e(m\tau')\\
    &= \sum_{n,m}\underset{r^2<4mnq}{\sum_{r\in \mbb Z}}\Big(\underset{(d,Nq)=1}{\sum_{d|(n,r,m)}}d^{k-1} c_h(\frac{4mnq-r^2}{d^2} )\Big)e(n\tau)e(m\tau')\\
    &+ q^{k-1} \sum_{n,m}\underset{r^2<4mn}{\sum_{r\in \mbb Z}}\Big(\underset{(d,N)=1}{\sum_{d|(n,r,m)}}d^{k-1} c_h(\frac{4mn-r^2}{d^2} )\Big)e(n\tau)e(mq\tau')\\
    &+q^{k-1} \sum_{n,m}\underset{r^2<4mn}{\sum_{r\in \mbb Z}}\Big(\underset{(d,N)=1}{\sum_{d|(n,r,m)}}d^{k-1} c_h(\frac{4mn-r^2}{d^2} )\Big)e(nq\tau)e(m\tau')\\
    &= \sum_{n,m}\underset{r^2<4mnq}{\sum_{r\in \mbb Z}}\Big(\underset{(d,N)=1}{\sum_{d|(n,r,mq)}}d^{k-1} c_h(\frac{4mnq-r^2}{d^2} )\Big)e(n\tau)e(m\tau')\\
    &+q^{k-1} \sum_{n,m}\underset{r^2<4mn}{\sum_{r\in \mbb Z}}\Big(\underset{(d,N)=1}{\sum_{d|(n,r,m)}}d^{k-1} c_h(\frac{4mn-r^2}{d^2} )\Big)e(n\tau)e(mq\tau').\qedhere
\end{align}
}
\end{proof}

As an immediate corollary, we get the following.
\begin{cor} Let $F\in \skk^{new}(N)$ be a newform and $c_{g,g'}$ be as in \eqref{z=0}. Then
\begin{enumerate}
    \item  Let $g,g'\in \mc B_{k+1}^{new}(N)$. Then $c_{g,g'}=0$ for $g\neq g'$.
    \item Let $g,g'\in \mc B_{k+1}^{old}(N)$. Then $c_{g,g'}=0$ unless $g$ and $g'$ belong to the same eigenspace for all $T_n$ with $(n,N)=1$.
\end{enumerate} 
\end{cor}
In other words, via classical newform theory, $c_{g,g'} = 0$ unless both of  the oldforms $g,g'$ arise from the same underlying newform $g''$ of some level dividing $N$.

We conclude this section with following vanishing result for the inner product $\lan F^\circ, g_\sigma\otimes g_{\sigma'}\ran$ when $g_\sigma,g_{\sigma'} \in \mc B_{k+1}^{old}(N)$ and $\sigma\neq \sigma'$.

\begin{prop} \label{cv-00-lem}
    For a proper divisor $N_g|N$, let $g\in S_{k+1}^{new}(N_g)$ and $M_g=N/N_g$. Then for any two distinct characters $\sigma$, $\sigma'$ of $\left(\mbb Z/2\mbb Z\right)^{\omega(M_g)}$,
    \begin{equation}
        \lan F^\circ, g_\sigma\otimes g_{\sigma'}\ran=0.
    \end{equation}
\end{prop}
\begin{proof}
    Since $\sigma\neq \sigma'$, there exists a $p|M_g$ such that $\sigma(p)=1$ and $\sigma'(p)=-1$. From our choice of basis, we have that $g_\sigma|W_N(p)= \sigma(p)g_\sigma$ and $g_{\sigma'}|W_N(p)= \sigma'(p)g_{\sigma'}$. 

    Now using the fact that $F$ is an eigenfunction of $\mc W_N(p)= W_N(p)\times W_N(p)$ with eigenvalue $+1$ (\cite[Theorem~5.2(ii)]{schmidt-SKlift}), we see that
    \begin{align}
        \lan F^\circ, g_\sigma\otimes g_{\sigma'}\ran= \lan (F|\mc W_N(p))^\circ, g_\sigma\otimes g_{\sigma'}\ran&=\lan F^\circ| (W_N(p)\times W_N(p) ), g_\sigma\otimes g_{\sigma'}\ran\\
        &= \sigma\sigma'(p) \lan F^\circ, g_\sigma\otimes g_{\sigma'}\ran.
    \end{align}
We get the claim in the lemma since $\sigma\sigma'(p)=-1$ and this completes the proof.
\end{proof}


\section{Coset representatives} \label{coset-section}
Let $N$ be any square-free integer and $M|N$. As mentioned in the Introduction, we will have to choose carefully coset representatives of $\Gamma_0^{(2)}(M)$ in $\Gamma_0^{(2)}(N)$, which `more or less' commute with the pullback operation. We try to involve the (diagonal) lifts of $\SL{2}{\z} \times \SL{2}{\z}$ matrices for this purpose.
 Write $N=d_1d_2$ and let $d_1\bar{d_1}\equiv 1\mod d_2$. We define
\begin{align} \label{bnd-def}
    B_{d_1}(a) := \begin{pmatrix} 1 & 0 & 0 & 0 \\ 0 & 1 & 0 & 0\\ 0 & d_1\bar{d_1}a & 1 & 0\\ d_1\bar{d_1}a & 0 & 0 & 1\\  \end{pmatrix}\in \Gamma_0^{(2)}(d_1).
\end{align}
Next, note that we have a canonical bijection
\begin{equation} \label{coset-prod}
     \prod_{p|N}\Gamma_0(p)\backslash\sltwo \cong   \Gamma_0(N) \backslash \sltwo .
\end{equation}
We wish to make this precise below by specific choice of representatives, as we will need these later.
Let $\gamma(a)= \smat{0}{1}{-1}{a}\in \sltwo$. It is well known that $\{I_2, \gamma(a): a\bmod p\}$ form a set of closet representatives for $\Gamma_0(p)\backslash\sltwo$. Now let $d|N$ and for each each $p_i|(N/d)$ let $0\le a_i < p_i$. Let $a$ be such that $a\equiv a_i \mod p_i$. Then the above bijection can be realized via the map
\begin{equation}
    (\gamma_{p_1}(a_1),...\gamma_{p_t}(a_t),I_2,...I_2) \longmapsto \gamma_d(a),
\end{equation}
where we write $N/d=p_1...p_t$, and $\gamma_d(a)$ is uniquely  characterized by the congruences (upto $\Gamma_0(N)$):
\begin{equation}\label{gammada}
\begin{split}
     \gamma_d(a)&\equiv \smat{0}{1}{-1}{a_i} \bmod p_i, \text{ for } p_i|(N/d);\\
     \gamma_d(a)&\equiv \smat{1}{0}{0}{1} \bmod p_i \text{ for } p_i|d.
\end{split}  
\end{equation}
\begin{rmk}\label{rmk:MpM}
    We also note here that for $N=Mp$, the set $\{I_2, \gamma_M(a): a\bmod p \}$, forms set of coset representatives for $\Gamma_0(Mp)\backslash \Gamma_0(M)$.
\end{rmk}
Next, write $N=Mp_1p_2..p_t$ and we choose coset representatives $ \{\gamma\}$ of $\Gamma_0(N)\backslash\Gamma_0(M)$ such that $\gamma$ corresponds to $(\gamma_{p_1},....,\gamma_{p_t}, I_2,....,I_2)$, where $\gamma_{p_i}\in\Gamma_0(Mp_{i})\backslash\Gamma_0(M)$. This can be easily seen by fixing $d=M$ in the discussion after \eqref{coset-prod}.

Now we define the following subsets of $\Gamma_0^{(2)}(M)$.
\begin{align}
    \mc C_1(N,M)&:=\{\gamma\times I_2: \gamma\in \Gamma_0(N)\backslash\Gamma_0(M)\}\\
    \text{ and for any } d|(N/M) \text{ and } d>1\\
    \mc C_d(N,M)&:=\{ B_{N/d}(a)\cdot (\gamma\times \gamma_{N/d}(b)): a,b\bmod d, \gamma\in \Gamma_0(N)\backslash\Gamma_0(M)\} .
\end{align}
We set 
\begin{align}
    \mc C(N,M):= \mc C_1(N,M) \cup_{d|(N/M), d>1} \mc C_d(N,M).
\end{align}
\begin{rmk}
    We note here are for $N=p$ and $M=1$, $\mc C(p,1)$ is exactly the set $R_p(0)\cup R_p(1)\cup R_p(2)$ as in \cite{klein2004}.
\end{rmk}
\begin{lem}\label{lem:MpMCoset}
    Let $M$ be any square-free integer and $(p,M)=1$. Then $\mc C(Mp, M)$ is a set of coset representatives for $\Gamma_0^{(2)}(Mp)\backslash\Gamma_0^{(2)}(M)$.
\end{lem}
\begin{proof}
    We first note that $|\mc C_1|=(p+1)$ and $|\mc C_p|= p^2 (p+1)$. Since the sets $\mc C_i$ are disjoint, we have that $|\mc C(Mp,M)|= (p+1)(p^2+1) = [\Gamma_0^{(2)}(M):\Gamma_0^{(2)}(Mp)]$.

    Next, we show that any two elements of $\mc C(Mp,M)$ are $\Gamma_0^{(2)}(Mp)$-inequivalent by the following set of observations. 
    \begin{enumerate}
    \item The elements in the set $\mc C_1$ are clearly $\Gamma_0^{(2)}(Mp)$-inequivalent.
    \item An element $\gamma\times I_2\in \mc C_1$ and $B_{M}(a)\cdot (\gamma\times \gamma_{M}(b))\in \mc C_p$ are $\Gamma_0^{(2)}(Mp)$-equivalent if and only if
    \begin{equation}
        B_{M}(a)\in \Gamma_0^{(2)}(Mp), \gamma\gamma_1^{-1}\in \Gamma_0(Mp), \text{ and } \gamma_{M}(b)\in \Gamma_0(Mp).
    \end{equation}
    But this is not possible since $p$ is a prime. Thus $\mc C_1$ and $\mc C_p$ are inequivalent.
    \item For $a=(a_1,a_2)$ and $b=(b_1, b_2)$, let $A_{p}(a)=B_{M}(a_1)\cdot (\gamma_1\times \gamma_{M}(a_2))$, $A_p(b)=B_{M}(b_1)\cdot (\gamma_2\times \gamma_{M}(b_2))$, where $B_{M}(a_1),  B_{M}(b_1)$ are as in \eqref{bnd-def}. Let us write 
    \begin{align}
    A_{p}(a)A_{p}(b)^{-1}=\smat{A}{B}{C}{D}; \,\, \gamma_1\gamma_2^{-1}=\smat{u_1}{v_1}{Mw_1}{x_1}; \,\, &\gamma_{M}(a_2)\gamma_{M}(b_2)^{-1}= \smat{u_2}{v_2}{w_2}{x_2}\equiv \smat{1}{0}{M(a_2-b_2)}{1} \bmod p.
    \end{align}
    Then we see that
    \begin{equation}
        C= M\overline{M}\left(M\overline{M}\smat{0}{a_1}{a_1}{0}\smat{v_1}{0}{0}{v_2}+\smat{x_1}{0}{0}{x_2}\right)\smat{0}{a_2}{a_2}{0}+M\overline{M}\smat{0}{a_1}{a_1}{0}\smat{u_1}{0}{0}{u_2}+\smat{Mw_1}{0}{0}{w_2}.
    \end{equation}
\end{enumerate} 
From this, $C\equiv 0 \mod Mp$ if and only if
\begin{align}
   -M^2\overline{M}^2a_1a_2v_2+Mw_1&\equiv 0 \mod Mp;\\
   -M^2\overline{M}^2a_1a_2v_1+M(b_2-b_1)&\equiv 0 \mod Mp;\\
   -M\overline{M}a_2x_2+M\overline{M} a_1u_1&\equiv 0\mod Mp;\\
   -M\overline{M}a_2x_1+M\overline{M} a_1u_2&\equiv 0\mod Mp.
\end{align}
If either $\gamma_1$ or $\gamma_2\equiv I_2\mod p$, then 
 first congruence is never satisfied unless $\gamma_1\equiv \gamma_2\equiv I_2\mod p$. Thus we can assume that $\gamma_1$ and $\gamma_2$ are of the form $\gamma_M(a)$ as in Remark \ref{rmk:MpM}. Thus from \eqref{gammada} and Remark \ref{rmk:MpM}, we have that $p|v_1, p|v_2$. Thus it is immediate that $p|w_1$, $p|(b_1-b_2)$. This implies, $a_1u_2\equiv a_2u_1\mod p$ and $a_1u_1\equiv a_2u_2\mod p$ with $(u_i,p)=1$. This is possible iff $u_1u_2^{-1}\equiv 1\mod p$ and thus $a_1\equiv a_2\mod p$. Thus we get that 
\begin{equation}
    C\equiv 0\mod Mp \iff \gamma_1\gamma_2^{-1}\in \Gamma_0(Mp), a_1\equiv a_2\mod p \text{ and } b_1\equiv b_2\mod p.
\end{equation}
Thus, the elements of $\mc C_p$ are $\Gamma_0(Mp)$-inequivalent. This completes the proof of lemma.
\end{proof}

\begin{lem}\label{piNMdef}
Let $M$ be a square-free integer and let $N=Mp_1....p_t$ with $(p_i,M)=1$. Then the map $\pi(N,M):\Gamma_0^{(2)}(N)\backslash \Gamma_0^{(2)}(M) \longrightarrow \prod_{i=1}^{t} \Gamma_0^{(2)}(Mp_i)\backslash \Gamma_0^{(2)}(M)$ defined by
\begin{equation}
    \pi_{N,M}(\Gamma_0^{(2)}(N)g)=(\Gamma_0^{(2)}(Mp_1)g, ....., \Gamma_0^{(2)}(Mp_t)g)
\end{equation}
is a bijection of sets.
\end{lem}

\begin{proof}
The map is clearly well-defined and is injective since $(p_i,M)=1$. Thus it is a bijection since both sides have the same cardinality. 
\end{proof}
Of particular interest to us is the subset $\mc C(N,M)$ of $\Gamma_0^{(2)}(N)\backslash \Gamma_0^{(2)}(M)$ -- our goal next is to show that it is indeed a complete set of coset representatives. This result will be crucial in our calculations involving the trace operator (see Section \ref{IPRel}).

\begin{lem}\label{piMpsur}
   Let $N$ be odd and square-free. Let a prime $p$ and $M$ be divisors of $N$ with $(p,M)=1$. Then the set $\mc C(N,M)$ maps surjectively onto $\mc C(Mp,M)$ under the map $\pi_{Mp,M}$.
\end{lem}
\begin{proof}
First, consider the set $\mc C_1(N,M)$. Since $(p,M)=1$, $\pi_{Mp,M}(\mc C_1(N,M))\subset \mc C_1(Mp,M)$. From our choice of representatives of $\Gamma_0(N)\backslash\Gamma_0(M)$ and Remark \ref{rmk:MpM}, it is surjective.

Now for any $d|(N/M)$ with $(d,p)=1$ consider the set $\mc C_d(N,M)$. In this case, note that $B_{N/d}(a)\in \Gamma_0^{(2)}(Mp)$, $\gamma_{N/d}(b)\in \Gamma_0(Mp)$. Thus $\pi_{Mp,M}(\mc C_d(N,M))\subset \mc C_1(Mp,M)$.

Next, consider the case $p|d$. In this case, we have that
\begin{equation}
   B_{N/d}(a)\cdot (\gamma\gamma^{-1}\times \gamma_{N/d}(b)\gamma_{M}(b)^{-1})\cdot B_{M}(-a)\in \Gamma_0(Mp).
\end{equation}
This follows by noting that $\gamma_{N/d}(b)\gamma_{M}(b)^{-1}\equiv I_2 \mod p$ and $B_{N/d}(a)B_M(-a)\mod p\equiv I_4$.
\end{proof}
Putting together Lemma \ref{lem:MpMCoset} and \ref{piMpsur}, we get the following.

\begin{prop} \label{prop:cnm}
    Let $N$ be an odd and square-free integer and $M|N$. Then $\mc C(N,M)$ gives a complete set of coset representatives for $\Gamma_0^{(2)}(N)\backslash \Gamma_0^{(2)}(M)$.
\end{prop}
\begin{proof}
Let $N=Mp_1..p_t$. Then from Lemma \ref{piMpsur}, we have that $\pi_{N,M}= (\pi_{Mp_1,M}, ...., \pi_{Mp_t,M}).$ That is,  
\begin{equation} \label{crt-nm}
    \pi_{N,M}(\mc C(N,M))=(\mc C(Mp_1,M), ....,\mc C(Mp_t,M)).
\end{equation}

From \lemref{lem:MpMCoset}, $\mc C(Mp_i,M)$ is a complete set of coset representatives of $\Gamma_0^{(2)}(Mp_i)\backslash \Gamma_0^{(2)}(M)$ for all $1 \le i \le t$. Next, from \lemref{piNMdef}, $\pi_{N,M}$ is a bijection, and thus the right hand side of \eqref{crt-nm} has cardinality $|\Gamma_0^{(2)}(N)\backslash \Gamma_0^{(2)}(M)|$.
Therefore \eqref{crt-nm} implies that $\mc C(N,M)$ is a complete set of closet representatives for $\Gamma_0^{(2)}(N)\backslash \Gamma_0^{(2)}(M)$, as desired.
\end{proof}

\section{Vanishing of the coefficients \texorpdfstring{$c_{\sigma,\sigma}(f;g)$}{aa} } \label{cv-0-proof}

Keeping the notation of the previous sections, let us define the quantity
\begin{align}
  c_{\sigma,\sigma}(f;g):=  \lan F^\circ, g_\sigma\otimes g_{\sigma}\ran \q \q (F=F_f).
\end{align}
The goal in this section is to show that $c_{\sigma,\sigma}(f;g)=0$ unless the Atkin-Lehner eigenvalues of $f$ satisfy $w_f(p)=+1$ for all $p|N/N_g$ (that is, $N/N_g\in \mc L_f$). This is proved in Corollary~\ref{F0old-vanish}.

\subsection{Inner product relations.}\label{IPRel}
 For $p|N$, let $U_S(p)$ denote the Hecke operator given by the double closet $\Gamma_0^{(2)}(N)\smat{1}{0}{0}{p}\Gamma_0^{(2)}(N)$. We know that $U_S(d)=\prod_{p|d}U_S(p)$. Recall that the effect of $U_S(d)$ on the Fourier expansion of a modular form is 
 \begin{equation}
 U_S(d) \colon \, \summ_T a_F(T)e(TZ) \mapsto \summ_T a_F(dT)e(TZ).
 \end{equation}
 Let $w_F(d)$ and $\lambda_F(d)$ denote the eigenvalues of $F$ with respect to the Atkin-Lehner operator $\mc W_N(d)$ and $U_S(d)$ respectively. Then we have
\begin{prop}\label{prop:IPFggV}
    Let $F\in S_{k+1}^{(2),new}(N)$ be a newform and $g\in S_{k+1}(N/p)$ be an eigenfunction of $W_{N/p}(N/p)$. Then we have the following.
    \begin{align}
   & (i) & \left(1+p^{1-k}\lambda_F(p)w_F(N)\right) \cdot \lan F^\circ, g\otimes g\ran =0, \\
   & (ii) & \left(1+p^{1-k}\lambda_F(p)w_F(N)\right) \cdot \lan F^\circ, g\otimes g|W_N(p)\ran =0.
    \end{align}
    \end{prop}
    
\begin{proof}
   Since $F$ is a newform of level $N$, we have that $F|\mrm{Tr}^{(2)}(N,N/p)=0$. Thus from \propref{prop:cnm},
\begin{align}
    0&=\sum_{A\in \mc C(N,N/p)}F|A \n \\
    &=\sum_{\gamma\in \Gamma_0(N)\backslash\Gamma_0(N/p)}F|(\gamma\times I_2)+\sum_{a,b\bmod p}\sum_{\gamma\in \Gamma_0(N)\backslash\Gamma_0(N/p)}F|B_{N/p}(a)|(\gamma\times \gamma_{N/p}(b)). \label{tr0}
\end{align}
Next, we note that 
\begin{equation} \label{rest-equiv}
    (F|(\gamma_1\times\gamma_2))^\circ= F^\circ|(\gamma_1\times\gamma_2) \q (\text{for any } \gamma_1,\gamma_2\in \sltwo).
\end{equation}
 Further, it is easy to see that for any $\mc H_0 \in S_{k+1}(M) \otimes S_{k+1}(M)$, $g_1,g_2 \in S_{k+1}(M)$, and $\gamma_1, \gamma_2 \in \Gamma_0(M)$, one has
\begin{align} \label{axb-eqiv}
    \lan \mc H_0 | (\gamma_1 \times \gamma_2), g_1 \otimes g_2 \ran = \lan \mc H_0, g_1 \otimes g_2 | (\gamma_1^{-1} \times \gamma_2^{-1})\ran.
\end{align}
Here we remind the reader that our inner products are normalized by volume, so that \eqref{axb-eqiv} holds. 

Since $(p,M)=1$, $\gamma:=\smat{a}{1}{Nb/p}{p}\in \Gamma_0(N/p)\subset \Gamma_0(M)$. Thus, from Lemma \ref{lem:WBCom} we get that $g|W_N(p)= g|B_p$. Further, let $\gamma_{N/p}(a)=\smat{x}{y}{Nz/p}{w}$. Since $(p,N/p)=1$, from \eqref{gammada} we have $x\equiv 0\mod p$.
Thus we have
\begin{equation}\label{BdgammaND}
    B_p\gamma_{N/p}(a)^{-1}= \smat{pw}{-y}{-Nz/p}{x/p}\smat{1}{0}{0}{p}.
\end{equation}
Let us now restrict both sides of \eqref{tr0} to $z=0$ and take inner product with $g\otimes g$ on both  sides to get
\begin{align}
    \lan F^\circ, g\otimes g\ran+ \sum_{a\bmod p}  p \lan \left(F|B_{N/p}(a)\right)^\circ, g\otimes g\ran =0;
\end{align}
Next, taking inner product with $g\otimes g|B_p$ on both  sides of \eqref{tr0} and then using the commutation relation in \eqref{BdgammaND} we obtain:
\begin{align}
    \lan F^\circ, g\otimes g|B_p\ran+ \sum_{a\bmod p}  p \lan \left(F|B_{N/p}(a)\right)^\circ, g\otimes g|\smat{1}{0}{0}{p}\ran =0;
\end{align}
upon using \eqref{axb-eqiv} repeatedly.

Since $g$ is an eigenfunction of $W_{N/p}(N/p)$, we get that
\begin{equation}\label{gginvAL}
    (g\otimes g)|(W_{N/p}(N/p)\times W_{N/p}(N/p))=g\otimes g.
\end{equation}
Next, we note that 
\begin{align}
   \smat{1}{0}{0}{p}W_{N/p}(N/p)= \smat{0}{-1}{N}{0}= W_{N/p}(N/p)\smat{p}{0}{0}{1}=W_{N/p}(N/p)B_p.
\end{align}
As a consequence, we have that
\begin{equation}\label{ggBpinvAL}
    (g\otimes g|\smat{1}{0}{0}{p})|(W_{N/p}(N/p)\times W_{N/p}(N/p))= g\otimes g|B_p.
\end{equation}
From \eqref{gginvAL} and \eqref{ggBpinvAL}, combined with the fact that $F|\mc W_N(N)=w_F(N) F$, for $V\in \{ Id, B_p\}$ we get
\begin{align} 
&\lan F^\circ, g\otimes g|V\ran+ p\,w_F(N)\sum_{a\bmod p} \lan \left(F|\mc W_N(N)B_{N/p}(a)\mc W_{N/p}(N/p)\right)^\circ, g\otimes g|V\ran =0.\label{tr-eqngg}
\end{align}
Next, we note that (see \eqref{wnd-s}, \eqref{bnd-def})
\begin{align}
    \mc W_{N}(N)B_{N/p}(a)\mc W_{N/p}(N/p)= \begin{psmallmatrix}
-\frac{N}{p} & 0 & 0 & -\frac{N}{p}\overline{\left(\frac{N}{p}\right)} a \\
0 & -\frac{N}{p} & -\frac{N}{p}\overline{\left(\frac{N}{p}\right)} a & 0 \\
0 & 0 & -N & 0 \\
0 & 0 & 0 & -N
\end{psmallmatrix}.
\end{align}

Thus we get
\begin{equation}
    F|\mc W_{N}(N)B_{N/p}(a)\mc W_{N/p}(N/p)= F| \begin{psmallmatrix}
1 & 0 & 0 & \overline{\left(\frac{N}{p}\right)} a \\
0 & 1 & \overline{\left(\frac{N}{p}\right)} a & 0 \\
0 & 0 & p & 0 \\
0 & 0 & 0 & p
\end{psmallmatrix}.
\end{equation}

For convenience, let us put 
\begin{align} \label{a'-def}
a'= a'(p) := \overline{\left(\tfrac{N}{p}\right)} \, a,
\end{align}
and similarly $b',c'$. In \eqref{a'-def}, for any $l|N$, recall that $\bar{l}$ denotes the inverse of $l \bmod N/l$.

Since both $g$ and $g|B_p$ are invariant w.r.t the translation $T^{b'}=\smat{1}{b'}{0}{1}$, for $V\in \{ Id, B_p\}$ we can write (again using \eqref{rest-equiv} and \eqref{axb-eqiv}):
\begin{equation}
\begin{split}
    \lan \left(F|\mc W_N(N)B_{N/p}(a)\mc W_{N/p}(N/p)\right)^\circ, g\otimes g|V\ran &= \lan \left(F| \begin{psmallmatrix}
1 & 0 & 0 & a' \\
0 & 1 & a' & 0 \\
0 & 0 & p & 0 \\
0 & 0 & 0 & p
\end{psmallmatrix}|T^{b'}\times T^{c'}\right)^\circ, g\otimes g|V\ran\\
&= \lan \left(F|\smat{I_2}{S}{0}{pI_2} \right)^\circ, g\otimes g|V\ran,
\end{split}
\end{equation}
where $S=\smat{b'}{a'}{a'}{c'}$.
As $a,b,c$ vary $\mod p$, so do $a',b',c'$. Thus we can write
\begin{align}
    \sum_{S\bmod p} \lan \left(F| \smat{I_2}{S}{0}{pI_2}\right)^\circ, g\otimes g|V\ran &= p^{2-k} \lan (F|U_S(p))^\circ, g\otimes g|V\ran.
\end{align}
Since $F|U_S(p)=\lambda_F(p)F$, we have
\begin{align}
\lan (F|U_S(p))^\circ, g\otimes g|V\ran&= \lambda_F(p) \lan F^\circ, g\otimes g|V\ran.
\end{align}

Putting this back into \eqref{tr-eqngg} finishes the proof of the Proposition.
\end{proof}

\begin{prop}\label{prop:IPzero}
    Let $f\in S_{2k}^{new}(N)$ be a newform and $g\in S_{k+1}(N/p)$  be an eigenfunction of $W_{N/p}(N/p)$. Let $F\in S_{k+1}^{(2),new}(N)$ be the SK lift of $f$. Then, we have
    \begin{align}
    (1-w_f(p))\lan F^\circ, g\otimes g\ran&=0\\
        (1-w_f(p))\lan F^\circ, g\otimes g|W_N(p)\ran&=0,
    \end{align}
    where $w_f(p)$ denotes the eigenvalue of $f$ w.r.t the Atkin-Lehner operator $W_{N}(p)$.
\end{prop}
\begin{proof}
First, since $F$ is an SK lift, $w_F(p)=1$ for every $p|N$ (see e.g. \cite[Thm.~5.2(ii)]{schmidt-SKlift}). Next, we have the following chain of isomorphisms.
    \begin{equation} \label{new-iso}
        S_{2k}^{new}(N)\cong S_{k+1/2}^{+, new}(4N)\cong J_{k,1}^{cusp, new}(N)\cong \mrm{SK}_{k+1}^{new}(N).
    \end{equation}
    Let $h$ and $\phi$ be the corresponding half-integral and Jacobi forms under the above isomorphism, respectively. Furthermore, let $\lambda_\phi(p^2)$ and $\lambda_h(p^2)$ denote the eigenvalues of $\phi$ and $h$ under $U_J(p)$ and $U(p^2)$, respectively. Then we have $\lambda_F(p)=\lambda_\phi(p^2)=\lambda_h(p^2)$ (from \cite{Ibu-SK}). The proof now follows from Proposition~\ref{prop:IPFggV} by noting that $\lambda_h(p^2)=-w_f(p)p^{k-1}$ (see \cite{kohnen1982newforms}). 
\end{proof}

\subsection{The vanishing of $c_{\sigma,\sigma}(f;g)$}
As an immediate corollary, we have the following vanishing result about the inner product $\lan F^\circ, g\otimes g\ran$ for $g\in \mc B_{k+1}^{old}(N)$. Recall the oldbasis from \eqref{oldbasis} given by
\begin{equation}
    \mc B_\ell^{old}(N)=\bigcup_{LM=N, L>1}\{\mc B_{\ell}(L,g): g\in \mc B_{\ell}^{new}(M) \},
\end{equation}
where $\mc B_\ell(L,g)$ denotes an orthogonal basis for the subspace $S_\ell(L,g):=\mrm{span}\{ g|B_d: d|L\}$. 

\begin{cor}\label{F0old-vanish}
    Let $f$ and $F$ be as above. Let $g\in S_{k+1}^{new}(N_g)$. If $w_f(p)=-1$ for some $p|(N/N_g)$, that is $M_g\not\in \mc L_f$, then $\lan F^\circ, g_\sigma\otimes g_\sigma\ran=0$ for any $g_\sigma\in \mc B_{k+1}(N/N_g,g)$.
\end{cor}
\begin{proof}
   Let $M_g=N/N_g$. Then any $g_\sigma\in \mc B_{k+1}(M_g,g)$ is of the form
   \begin{equation}
       g_\sigma=\sum_{d|M_g}\sigma(d)g|W_N(d).
   \end{equation}
Let $g_1=\sum_{d|M_g, (d,p)=1} \sigma(d)g|W_N(d)$. Then we have that $g_1\in S_{k+1}(N/p)$ and it is an eigenfunction of $W_{N/p}(N/p)$. Further, $g_\sigma=g_1+\sigma(p)g_1|W_N(p)$. Thus
\begin{align}
    \lan F^\circ, g_\sigma\otimes g_\sigma\ran= 2\lan F^\circ, g\otimes g\ran+2\sigma(p)\lan F^\circ, g\otimes g|W_N(p)\ran.
\end{align}
The proof is now complete by Proposition \ref{prop:IPzero}.
\end{proof}
\section{Central-value formula for old-classes}\label{dMg}
In this section, we will prove all the main results of the paper, viz. Theorem~\ref{th:IPequiv} and Theorem~\ref{thm:CVOldclass}. As discussed in the Introduction, the step would essentially be the passage
\begin{align}
     \lan F^\circ, g \otimes g|W_N(M)\ran \leadsto \lan \phi^\circ,  g|W_N(M)\ran
\end{align}
and comparing the resulting expressions for $M$ and $M_g$.

\subsection{Unfolding of the restricted inner product.} \label{cv-old-sec}
We have the Fourier-Jacobi expansion
\begin{equation}
    F(Z)=\sum_m \phi_m(\tau, z) e(m\tau'),
\end{equation}
where $\phi_m=\phi|V_m$ for a newform $\phi\in J_{k+1,1}^{new}(N)$ (see section \ref{sec:VmOP}). Thus 
\begin{equation} \label{f0phim}
    F^\circ(\tau, \tau')= \sum_m \phi_m(\tau, 0) e(m\tau').
\end{equation}
Let $\phi\in J_{k+1}(N)$ and  for $d|N^\infty$, $\phi_d=\phi|V_d$. Then we have $\phi_d(\tau, 0)\in S_{k+1}(N)$ and 
\begin{equation}
    \phi_d(\tau, 0)= \sum_n \mc R(nd) c_\phi(nd) q^n,
\end{equation}
where $\mc R(t) =\# \{ r\in \mbb Z : r^2< 4t\}$. Thus for $d|N^\infty$, we can conclude that
\begin{equation} \label{phidN}
    \phi_d(\tau, 0)= \phi(\tau, 0)|U(d).
\end{equation}
Next consider that case $(m,N)=1$. In this case, we have
\begin{align}
\phi_m(\tau,0)&= \sum_n \sum_{r^2<4mn} \sum_{d|(n,m,r)} d^{k} c(nm/d^2, r/d) q^n.\\
&= \sum_n\sum_{d|(m,n)}d^k\mc R(mn/d^2) c(nm/d^2)q^n\\
&= \phi(\tau,0)|T(m).
\end{align}
Thus writing $m= d m_1$ with $(m_1, N)=1$ and $d|N^\infty$, we get 
\begin{align} \label{phimN}
\phi_m(\tau, 0)= \phi(\tau,0)|U(d)T(m_1) 
\end{align}
Put $\phi^\circ(\tau)  := \phi(\tau,0)$. Then we have the following. 
\begin{align}
    F^\circ(\tau, \tau') &=\sum_{m, (m,N)=1} \sum_{d|N^\infty} \phi^\circ| U(d)T(m)e(md\tau').
\end{align}
Let $g\in S_{k+1}(N_g)$ be a newform of level $N_g$. Then we have the following
\begin{align}
    \lan F^\circ, g\otimes g \ran &= \sum_{m, (m,N)=1}\sum_{d|N^\infty} \lan \phi^\circ|U(d)T(m)e(md\tau'), g\otimes g \ran \\
    &= \sum_{m, (m,N)=1}\sum_{d|N^\infty} \lan \phi^\circ|U(d)e(md\tau'), g|T(m)\otimes g \ran\,    \\
    &= \sum_{m, (m,N)=1} \lambda_g(m) \sum_{d|N^\infty} \lan \phi^\circ|U(d)e(md\tau'), g\otimes g \ran \\
    &= \sum_{d|N^\infty}d^{-(k+1)/2} \lan \phi^\circ|U(d) \otimes g_N|B_{d}, g\otimes g \ran, 
\end{align}
where $B_d$ is the level raising operator defined as in \eqref{BdOpr} and
\begin{align} \label{gN-def}
  g_{N}(z)&:= \sum_{(n,N)=1} \lambda_g(n) e(nz).
\end{align}
\begin{equation}
 \text{ Let } M_g:=N/N_g. \text{ For any } d|N^\infty, \mbox{ we write } d=d_g\ell_g, \mbox{ where } d_g|N_g^\infty \mbox{ and } \ell_g|M_g^\infty. 
\end{equation}
Then, since our inner products are normalized, we get
\begin{align}
    \lan F^\circ, g\otimes g \ran &=  \sum_{d_g|N_g^\infty}\sum_{\ell_g|M_g^\infty} (d_g\ell_g)^{-(k+1)/2}  \lan \phi^\circ|U(d_g\ell_g), g\ran \cdot \lan g_N|B_{d_g\ell_g}, g \ran.
\end{align}
Similarly, for any $M|M_g$ we have
\begin{align} \label{FggBd'}
    \lan F^\circ, g \otimes g|W_N(M)\ran &= \sum_{d_g|N_g^\infty}\sum_{\ell_g|M_g^\infty} (d_g\ell_g)^{-(k+1)/2}  \lan \phi^\circ|U(d_g\ell_g), g \ran \cdot \lan g_N|B_{d_g\ell_g}, g|W_N(M) \ran  \\
    &= \sum_{d_g|N_g^\infty}\sum_{\ell_g|M_g^\infty} (d_g\ell_g)^{-(k+1)/2} \lan \phi^\circ, g | W_N U(d_g\ell_g) W_N\ran \cdot \lan g_N|B_{d_g\ell_g}, g|B_{M} \ran.
\end{align}
To arrive at \eqref{FggBd'}, we use the fact that, for any $r|N$, the adjoint $U(r)^*$ of $U(r)$ is given by $U(r)^*= W_NU(r)W_N$ (see [pp. 187] \cite{diamond2005first}). We also use that $W_N(M)=\gamma B_M$ for some $\gamma\in \Gamma_0(N_g)$, since $(M,N_g)=1$.

As we can see, there are two inner products present in \eqref{FggBd'}. The next few sections will be devoted towards evaluating the one involving $g_N$, and unraveling the one involving $\phi^\circ$, so as to compare \eqref{FggBd'} as $d$ varies across the divisors of $N^\infty$. This will ultimately lead us to the desired central-value formulae for all old forms in complete uniformity.
\subsection{Calculation of the inner product involving $g_N$} \label{gN-sec}
From the definition of $g_N$ in \eqref{gN-def} we can remove the coprimality condition $(n,N)=1$ by inserting the M\"obius function and write,
\begin{equation} \label{gn-mu}
    g_N= \sum_{l|N}(-1)^{\omega(l)}l^{-(k+1)/2}g|U(l)B_l = \sum_{l|N} \mu(l) l^{-(k+1)/2}g|U(l)B_l .
\end{equation}
Now we assume that $g$ is a newform of level $N_g$. Thus $g$ is an eigenfunction of $U(l)$ for $l|N_g$. Recall that $N=N_f=N_gM_g$. We write $l=d_1d_2$ uniquely ($(d_1,d_2)=1$), with $d_1|N_g$ and $d_2|M_g$.
Let $\lambda_g(d_2)$ denote the $U(d_2)$-eigenvalue of $g$. For any $L,M \geq  1$, thus we can write using \eqref{gn-mu} that 
\begin{align}
     \lan g_N|B_{L}, g|B_{M} \ran= \sum_{d_2|M_g}\sum_{d_1|N_g}\mu(d_1d_2)(d_1d_2)^{-(k+1)/2}\lambda_g(d_1)\lan g|U(d_2)B_{Ld_1d_2}, g|B_{M}\ran.
\end{align}
On level $N_g$, for any $p\nmid N_g$, we have that $U(p)=T(p)-p^{(k-1)/2}B_p$. Thus for $d_2|M_g$, we can write
\begin{equation}\label{UdTdBd}
    U(d_2)=\sum_{\ell|d_2}\mu(\ell)\ell^{(k-1)/2}T(d_2/\ell)B_\ell.
\end{equation}
We write $d_2=\ell d_3$, eliminate $d_2$, and rename the $\ell$ as the new $d_2$. Note that $(d_2,d_3)=1$, since $(\ell, d_3)=1$ and $d_2$ is square-free. This gives us the relation
\begin{align}\label{Bd1d2}
     & \lan g_N|B_{L}, g|B_{M} \ran \\  &\q=\sum_{d_1|N_g}\sum_{d_2,d_3|M_g, (d_2,d_3)=1}\mu(d_1d_3)(d_1d_3)^{-(k+1)/2}d_2^{-1}\lambda_g(d_1)\lambda_g(d_3)\lan g|B_{Ld_1d_2^2d_3}, g|B_{M}\ran.
\end{align}
For any $L,M  \ge 1$ and any prime $p$ such that $(p,L)=1$, \eqref{Bd1d2} thus leads us to consider the inner product
\begin{equation}
    A_p(L,M,t):=\lan g|B_{Lp^t}, g|B_M\ran.  
\end{equation} 

From \cite[Theorem.~6]{schulze2018petersson}, have the following lemma to evaluate $ A_p(L,M,t)$.
\begin{lem} \label{lem-pilot}
Let $l=(Lp^t, M)$. Then,
\begin{align} \label{pilot}
    A_p(L,M,t)&= \frac{l^{k+1}\lambda_g(1, Lp^t/l)\lambda_g(1, M/l)}{(Lp^tM)^{(k+1)/2}}\prod_{q|(Lp^tM/l^2), q\nmid N_g}(1+q^{-1})^{-1}\lan g,g\ran.
\end{align}
\end{lem}

In the sequel, we would need to invoke \lemref{lem-pilot} under various mutually exclusive divisibility conditions on $p$. We summarize them in a lemma given below. First, to simplify the notation, we introduce the function $\eta$ defined by
\begin{align}
    \eta(p, t):=\begin{cases}
     \lambda_g(1, p^t) \, p^{-(k+1)t/2} \, \left(1+\frac{1}{p}\right)^{-1} &\text{ for } t\ge 1 \text{ and } p|M_g; \\
     \lambda_g(1, p^t) \, p^{-(k+1)t/2}  &\text{ for } t\ge 1 \text{ and } p|N_g;\\
     1 & \text{ otherwise. }
    \end{cases}   
\end{align}

\begin{lem} \label{lem-almt}
Let $L,M \ge 1$, $M$ square-free and $(p,L)=1$. For $p\nmid M$, 
\begin{align}
    A_p(L,M,t)=\eta(p,t)A_p(L,M,0).
\end{align}
When $p|M$, we have the following cases.
\begin{align}
A_p(L,M,0)&= \eta(p,1)A_p(L, M/p, 0); \label{p|m1}\\
A_p(L,M,1)&=\eta(p,0) A_p(L, M/p, 0); \label{p|m2} \\
A_p(L,M,t)&=\eta(p,t-1)A_p(L, M/p, 0) \text{ when } t\ge 2. \label{p|m3}
\end{align}
\end{lem}

\begin{proof}
Put $l=(Lp^t,M)$. Then we note that
\begin{align}
    A_p(L,M,t)&=\frac{\lambda_g(1,p^t)}{p^{(k+1)t/2}} A_p(L,M,0) = \frac{\lambda_g(p^t)}{p^{(k+1)t/2}}  A_p(L,M,0)\q \text{ when }  p|N_g; \label{p|ng} \\
A_p(L,M,t)&= \frac{\lambda_g(1, p^t)}{p^{(k+1)t/2}}(1+p^{-1})^{-1}A_p(L, M, 0)\q\q\q\q\q\q \text{ when }  p\nmid MN_g.
 \label{pmng} 
\end{align}
The relation \eqref{p|ng} follows by noting that $l=(L,M)$, and we already have $(p,L)=1$. Moreover since $p|N_g$, one has $\lambda_g(1,p^t)=\lambda_g(p^t)$ for all $t \ge 0$.
Similarly for \eqref{pmng}, again we have $l=(L,M)$. The $t=0$ case is trivial. Further, if $t \ge 1$, then $p$ occurs in the product over $q$ as in \eqref{pilot}, and has to be isolated.

Among the $p|M$ assertions, \eqref{p|m1} is straightforward, and since $M$ is square-free, \eqref{p|m2} follows from \eqref{pilot} by writing $l=(Lp,M)=p l'$, where $l'=(L,M/p)$. The first equality in \eqref{p|m3} follows exactly as that in \eqref{p|m2}. That is, by noting that
\begin{equation}
    A_p(L,M,t)=A_p(L, M/p, t-1)= \frac{\lambda_g(1,p^{t-1})}{p^{(k+1)(t-1)/2}}(1+p^{-1})^{-1}A_p(L, M/p, 0) \text{ when } t\ge 2.
\end{equation}
The second equality above follows by using the same definition of $l'$ as above, the multiplicative property of $\lambda_g(1,n)$ viz., $\lambda_g(1, Lp^{t-1} /l') = \lambda_g(1, L/l') \lambda_g(1, p^{t-1})$, and from the fact that $p| L p^{t-1} (M/p)/l'^2$, $p \nmid N_g$.
\end{proof}
\begin{lem}\label{lemBLBM}
Let $M|M_g$ and $L\ge 1$. Then
\begin{align}
    \lan g|B_L, g|B_M\ran&= \lan g, g\ran\prod_{q|N_g, q^t||L}\eta(q,t) \prod_{q|(M_g/M), q^t||L}\eta(q,t) \prod_{q|M, q^t||L}\eta(q,t-1)\prod_{q|M, q\nmid L}\eta(q,1). 
\end{align}
\end{lem}

\begin{proof}
    Let $L=p_1^{t_1}...p_r^{t_r}$. We will prove this by induction on $r$. When $r=1$, the identity holds by applying  Lemma \ref{lem-almt} to $A_{p_1}(1,M,t_1)$. Let the lemma hold for all $L$ having $r-1$ prime factors. 

    Let $L=L_1p_r^{t_r}$ with $t_r>0$. Then one can easily see that,
    \begin{equation}
     \lan g|B_L, g|B_M\ran= A_{p_r}(L_1,M,t_r).   
    \end{equation}
    The proof now follows by Lemma \ref{lem-almt} and induction hypothesis.\qedhere
\end{proof}

In the case at hand, let us remind the reader (cf. \eqref{Bd1d2}) that 
\begin{equation}
L=d_g\ell_gd_1d_2^2d_3 \ \text{ with } \ d_g|N_g^\infty, \ \ell_g|M_g^\infty, \ d_1|N_g \text{ and }  \ d_2,d_3|M_g. 
\end{equation}
Then we have the following lemma.
\begin{lem}
Let $d_g|N_g^\infty$ and $\ell_g|M_g^\infty$. Then for any $M|M_g$, we have
\begin{align}\label{BLBdEval}
    \lan g_N|B_{d_g\ell_g}, g|B_M\ran = \frac{\lan g, g\ran}{\zeta^{N_g}(2)}\frac{\lambda_g(d_g)}{d_g^{(k+1)/2}}  &\underset{(d_2,d_3)=1}{\sum_{d_2,d_3|M_g}}\frac{\mu(d_3)\lambda_g(d_3)}{d_2d_3^{(k+1)/2}} \times \\
  & \q \times \underset{p^t||\ell_gd_2^2d_3}{\prod_{p|(M_g/M)}}\eta(p, t)\underset{p^t||\ell_gd_2^2d_3}{\prod_{p|M}}\eta(p, t-1)\underset{p\nmid \ell_gd_2^2d_3}{\prod_{p|M}}\eta(p,1).
\end{align}    
\end{lem}
\begin{proof}
Putting $L=d_g\ell_gd_1d_2^2d_3 $ in Lemma \ref{lemBLBM} we have
\begin{align} \label{blabla}
    \lan g|B_{d_g\ell_gd_1d_2^2d_3}, g|B_M\ran&=\lan g, g\ran\frac{\lambda_g(d_gd_1)}{(d_gd_1)^{(k+1)/2}} \underset{p^t||\ell_gd_2^2d_3}{\prod_{p|(M_g/M)}}\eta(p, t)\underset{p^t||\ell_gd_2^2d_3}{\prod_{p|M}}\eta(p, t-1)\underset{p\nmid \ell_gd_2^2d_3}{\prod_{p|M}}\eta(p,1).
\end{align}
Substituting $d_g \ell_g$ for $L$ on the LHS of \eqref{Bd1d2}, and then substituting \eqref{blabla} in its RHS, we get
\begin{align}
    \lan g_N|B_{d_g\ell_g}, g|B_M\ran&=\lan g, g\ran \sum_{d_1|N_g}\underset{(d_2,d_3)=1}{\sum_{d_2,d_3|M_g}}\frac{\mu(d_1d_3)}{(d_1d_3)^{(k+1)/2}}d_2^{-1}\lambda_g(d_3)\frac{\lambda_g(d_g)\lambda_g(d_1)^2}{(d_gd_1)^{(k+1)/2}}\\
    & \times\underset{p^t||\ell_gd_2^2d_3}{\prod_{p|(M_g/M)}}\eta(p, t)\underset{p^t||\ell_gd_2^2d_3}{\prod_{p|M}}\eta(p, t-1)\underset{p\nmid \ell_gd_2^2d_3}{\prod_{p|M}}\eta(p,1).
\end{align}
The proof is now complete by noting that $\lambda_g(d_1)^2=d_1^{k-1}$ and the sum over $d_1$ evaluates to
\begin{equation}
    \sum_{d_1|N_g} \mu(d_1) d_1^{-(k+1)} d_1^{k-1}= \sum_{d_1|N_g}\mu(d_1)d_1^{-2}= \prod_{p|N_g}(1-p^{-2})=\zeta_{(N_g)}(2)^{-1}.\qedhere
\end{equation}
\end{proof}
\subsection{Calculation of the inner product involving $\phi$ -- Step I} \label{cdphi-sec}
For $d|N^\infty$, let
\begin{equation}
    C(d, \phi): = \lan \phi^\circ, g |W_N U(d) W_N\ran.
\end{equation}
Looking at \eqref{FggBd'}, we see that the remaining inner product that has to be understood is the quantity $ C(d, \phi)$.
The primary goal of this subsection is to conduct the first step towards expressing the quantities $C(d, \phi)$ in terms of $\lan \phi^\circ, g|W_N(M_g)\ran$. This information, when fed into \eqref{FggBd'} will give us \thmref{th:IPequiv}. The second, and final step, will be the subject of the next subsection~\ref{cmnl-sec}. Here, we will show the following.

\begin{prop} \label{Cdphidef-prop}
Let $d|N^\infty$. Write $d=d_g \ell_g$, where $d_g|N_g^\infty, \ell_g|M_g^\infty$. Put $l_g:= (M_g, \ell_g)$.Then,
\begin{equation}\label{Cdphidef}
    C(d,\phi)=C(d_g\ell_g, \phi)= \lambda_g(d_g)l_g^{(k+1)/2}\lan \phi^\circ, g|U(\ell_g/l_g) W_N(l_g)\ran.
\end{equation}
\end{prop}

\begin{proof}
We have $W_N=W_N(N_g) \, W_N(M_g)$. The Atkin-Lehner operator $W_N(N_g)$ can be written as
\begin{equation}\label{WNWNg}
    W_N(N_g)=\smat{N_ga}{b}{Nc}{N_g}= \smat{b}{-a}{N_g}{-M_g}\smat{0}{-1}{N_g}{0}= \gamma W_{N_g}(N_g).
\end{equation}
Since $N_ga-M_gbc=1$, we can conclude that $\gamma\in \Gamma_0(N_g)$. Let $d,d_g,\ell_g$ be as iin \propref{Cdphidef-prop}. Thus using the commuting property of operators $W$ and $U$, we get
\begin{align}
    W_NU(d)W_N&= W_N(N_g)\,W_N(M_g)\,U(d_g)\,U(\ell_g)\,W_N(N_g)\,W_N(M_g)\\
    &=W_N(N_g)\,U(d_g)\,W_N(N_g)\,W_N(M_g)\,U(\ell_g)\,W_N(M_g).
\end{align}
Using \eqref{WNWNg} and the fact that $g$ is a newform of level $N_g$, we can now write
\begin{align}
    g|W_NU(d)W_N&= g|\gamma W_{N_g}(N_g)\,U(d_g)\,\gamma W_{N_g}(N_g)\,W_N(M_g)\,U(\ell_g)\,W_N(M_g)\\
    &=\lambda_g(d_g) \, g|W_N(M_g)\,U(\ell_g)\,W_N(M_g). \label{wudw-1}
\end{align}
Since $(N_g, M_g)=1$, we have $W_N(M_g)=\gamma B_{M_g}$ (see Lemma \ref{lem:WBCom}). Moreover, it is elementary that $B_\ell U(\ell)= \ell^{(k+1)/2} Id$. We also recall that $l_g$ is the square-free part of $\ell_g$, i.e., $l_g:= (M_g, \ell_g)$. Then,
\begin{align}
B_{M_g} U(\ell_g)= B_{l_g} B_{M_g/l_g} \cdot U(l_g) U(\ell_g/l_g)= B_{l_g} U(l_g) B_{M_g/l_g} U(\ell_g/l_g)=l_g^{(k+1)/2}B_{M_g/l_g} U(\ell_g/l_g)  . 
\end{align} 
Thus we can conclude that
\begin{align}
    g|W_N(M_g)U(\ell_g)W_N(M_g) &=l_g^{(k+1)/2} g|B_{M_g/l_g} U(\ell_g/l_g)W_N(M_g)\\
    &= l_g^{(k+1)/2}g|W_N(M_g/l_g) U(\ell_g/l_g)W_N(M_g)\\
    &=l_g^{(k+1)/2} g|U(\ell_g/l_g)W_N(M_g/l_g)W_N(M_g) \q(\text{since } (\ell_g/l_g, M_g/l_g)=1)\\
    &=l_g^{(k+1)/2} g|U(\ell_g/l_g)W_N(l_g). \label{wudw-2}
\end{align}
Thus combining the results from \eqref{wudw-1} and \eqref{wudw-2}, we get
\begin{align}
    g|W_NU(d_g\ell_g)W_N &= \lambda_g(d_g)l_g^{(k+1)/2} g|U(\ell_g/l_g)W_N(l_g) ,\label{WUW}
    \end{align}
as desired. Taking inner product now with $\phi^\circ$ now completes the proof of \propref{Cdphidef-prop}.
\end{proof}

\subsection{Calculation of the inner product involving $\phi$ -- Step II: Unfolding the quantities $\mc C_\phi(\ell_g/l_g,1, l_g)$} \label{cmnl-sec}

The aim of this subsection is to finish our treatment of the inner product involving $\phi$ (cf. \eqref{FggBd'}). We pick up the notation and setting from the last subsection.

For any $(m,n)=1$, and $L||N$, define
\begin{equation} \label{cmnl-def}
    \mc C_\phi(m,n, L):= \lan \phi^\circ, g|U(mn) W_N(L) \ran.
\end{equation}
Then note that our object of interest in this subsection is (cf. \eqref{Cdphidef})
\begin{equation}
   \mc C_\phi(\ell_g/l_g,1, l_g)=\lan \phi^\circ, g|U(\ell_g/l_g) W_N(l_g) \ran. 
\end{equation}

We first state two lemmas, which will be used to obtain information about the quantity $\mc C_\phi(\ell_g/l_g,1, l_g)$ (\propref{cdphi-eval}), whose proofs are deferred until the end.

\begin{lem}\label{lem:cphilg}
Let $\ell_g|N_g^\infty$ and $l_g=(\ell_g, M_g)$. Then
   \begin{align}
    \mc C_\phi (\ell_g/l_g,1, l_g) = \sum_{\ell|(\ell_g/l_g)}\mu(\ell)\ell^{(k-1)/2}\lambda_g(\ell_g/l_g\ell)\mc C_\phi\left(1,1,l_g/\ell\right)\label{Cphimlg}.
    \end{align}
\end{lem}


\begin{lem} \label{cosetj}
 Let $M|N$. A complete set of coset representatives of $\Gamma_0(N)^J$ in $\Gamma_0(M)^J$ can be taken to be $(\gamma, [0,0]) $ where $\gamma$ runs over $\Gamma_0(N) \backslash \Gamma_0(M)$. 
\end{lem}

\begin{lem} \label{phi0-ortho}
    Let $N=N_gM_g$ with $(N_g,M_g)=1$ and $\phi\in J_{k,1}^{new}(N)$. Then for any proper divisor $l_g|M_g$ and $g\in S_{k+1}(N_g)$, 
    \begin{equation}
        \lan \phi^\circ, g|W_N(l_g)\ran=0
    \end{equation}
\end{lem}
\begin{rmk}
     Note that Lemma \ref{phi0-ortho} does not imply that $\phi^\circ$ is new, as the lemma does not say anything for the operator $W_N(M_g)$. But $\phi^\circ$ is a linear combination of newforms, see Corollary~\ref{phi-new}.
\end{rmk}
The proof of the Lemmas are deferred to the end of this section. The main result of this subsection is as follows. 
 \begin{prop} \label{cdphi-eval}
 Let $d_g|N_g^\infty$ and $\ell_g|M_g^\infty$, and $d=d_g \ell_g$. Then
 \begin{equation}\label{CdEval}
   C(d, \phi)= C(d_g\ell_g, \phi)=\begin{cases}
        0 & \text{ if } M_g\nmid \ell_g;\\
       M_g^{(k+1)/2} \lambda_g(d_g)\lambda_g(\ell_g/M_g) \lan \phi^\circ, g|W_N(M_g)\ran & \text{ if } (\ell_g, M_g)=M_g.
    \end{cases}
\end{equation}    
 \end{prop}

\begin{proof}
From \eqref{cmnl-def}, recall that (replacing $L$ with $l_g/\ell$)
\begin{equation}
    \mc C_\phi\left(1,1,l_g/\ell \right)= \lan \phi^\circ, g|W_N(l_g/\ell)\ran. 
\end{equation}
In view of \lemref{phi0-ortho}, in \eqref{Cphimlg}, only terms with $l_g/\ell=M_g$ survive. In other words, $\mc C_\phi\left(1,1,l_g/\ell \right)$ survives only when $\ell=1$ and $l_g=M_g$.
Thus going back to the expression for $C(d, \phi)$ from \eqref{Cdphidef} in \propref{Cdphidef-prop}, we get the \propref{cdphi-eval}.
\end{proof}

\subsubsection{Proof of Lemma \ref{lem:cphilg}}
When $m$ is fixed and $n=p^t$, let us put, for convenience,
\begin{equation}
    \mc C_p(t,l_g):= \mc C_\phi(m,p^t, l_g).
\end{equation}
Using \eqref{UdTdBd}, when $(p,m)=1$, we see that $g|U(p^tm)= \lambda_g(p) g|U(p^{t-1}m)- p^{(k-1)/2}g|B_pU(p^{t-1}m)= \lambda_g(p) g|U(p^{t-1}m)- p^kg|U(p^{t-2}m)$. Thus we have the recurrence relations
\begin{align}
    \mc C_p(t, l_g) &= \lambda_g(p) \mc C_p(t-1,l_g)-p^k \mc C_p(t-2, l_g)  \q (t \ge 2) \\
    \mc C_p(1, l_g) &= \lambda_g(p) \mc C_p(0,l_g)-  p^{(k-1)/2} \lan \phi^\circ, g|B_p U(m)W_N(l_g)\ran.  
\end{align}

Let us now return to the situation at hand.
Recall that $p | \ell_g | M_g^\infty$, which implies that $B_p = \gamma W_N(p)$ for some $\gamma \in \Gamma_0(N_g)$. Thus for $t=1$,we have
\begin{align}
    \mc C_p(1,l_g)= \lambda_g(p)\mc C_p(0,l_g) - p^{(k-1)/2} \lan \phi^\circ, g|W_N(p)U(m)W_N(l_g)\ran.
\end{align}
In the case at hand, we have $\ell_g/l_g=p^tm$ with $(p,m)=1$. Since $l_g$ is the square-free part of $\ell_g$, we have that $p|l_g$. Thus $W_N(p)U(m)W_N(l_g)= U(m)W_N(p)W_N(l_g)= U(m)W_N(l_g/p)$.

Then for $t=1$, we have
\begin{align}
    \mc C_p(1, l_g)&= \lambda_g(p)\mc C_p(0,l_g)- p^{(k-1)/2}\mc C_p(0, l_g/p).
\end{align}

Solving the recurrence, we get, for $t \ge 2$, that
\begin{align}\label{Cprecurrence}
    \mc C_p(t,l_g)&=\mc C_p(1,l_g)\frac{\alpha_g(p)^{t}-\beta_g(p)^{t}}{\alpha_g(p)-\beta_g(p)}-p^k \mc C_p(0,l_g)\frac{\alpha_g(p)^{t-1}-\beta_g(p)^{t-1}}{\alpha_g(p)-\beta_g(p)}\\
    &= \lambda_g(p^{t-1})\mc C_p(1,l_g)-p^k\lambda_g(p^{t-2})\mc C_p(0,l_g)\\
    &= \lambda_g(p^t)\mc C_p(0,l_g) -p^{(k-1)/2}\lambda_g(p^{t-1}) \mc C_p(0, l_g/p).
\end{align}
Here $\alpha_g(p), \beta_g(p)$ are the Satake parameters at $p$. That is, $\alpha_g(p)+\beta_g(p)=\lambda_g(p)$ and $\alpha_g(p)\beta_g(p)=p^k$. As a result, we get
\begin{equation}
    \mc C_\phi(m, p^t, l_g)= \lambda_g(p^t)\mc C_\phi(m,1,l_g)-p^{(k-1)/2}\lambda_g(p^{t-1})\mc C_\phi(m,1,l_g/p).\label{Cphim1}
\end{equation}
This allows us to obtain a formula for $\mc C_\phi (\ell_g/l_g,1, l_g)$ inductively by removing one prime at a time as illustrated below.

We induct on the number of prime factors of $\ell_g/l_g$. If $\ell_g/l_g=p^t$, then the identity follows from \eqref{Cphim1} with $m=1$.
Next, let $\ell_g/l_g=mp^t$ with $(m,p)=1$. Then again from \eqref{Cphim1}, using the induction hypothesis, we have 
\begin{align}
    \mc C_\phi(m, p^t, l_g)&= \lambda_g(p^t)\mc C_\phi(m,1,l_g)-p^{(k-1)/2}\lambda_g(p^{t-1})\mc C_\phi(m,1,l_g/p)\\
    &=\sum_{\ell|m}\mu(\ell)\ell^{(k-1)/2}\lambda_g(mp^t/\ell)\mc C_\phi\left(1,1,l_g/\ell\right)\\
    &\q- p^{(k-1)/2}\sum_{\ell|m}\mu(\ell)\ell^{(k-1)/2}\lambda_g(mp^{t-1}/\ell)\mc C_\phi\left(1,1,l_g/p\ell\right).
\end{align}
It is readily checked that this gives the required identity, and thus finishes the proof of Lemma \ref{lem:cphilg}.\qed
\subsubsection{Proof of Lemma \ref{cosetj}}
This follows essentially by slightly modifying the proof of \cite[Lemma 6.2(i)]{das-anamby2}. First of all, the elements $(\gamma, [0,0]) $, with $\gamma$ as in the Lemma, are clearly in-equivalent modulo $\Gamma_0(N)^J$.
If $(P',X') \in \Gamma_0(M)^J$, then we can always left multiply it with $(1_2,X) \in  \Gamma_0(N)^J$, where $X$ is chosen such that $XP'+X'=0$. Thus we can assume $X'=0$. Then left multiplication by a suitable element $(g,0)$ with $g \in \Gamma_0(N)$ gives the Lemma. \QEDB

\subsubsection{Proof of Lemma \ref{phi0-ortho}}
        Let $\mrm{Tr}^J(N,N_gl_g): J_{k,1}(N)\longrightarrow J_{k,1}(N_gl_g)$ denote the Jacobi trace operator. Since $\phi$ is new of level $N$ and $N_gl_g\neq N$, we see that
        \begin{equation}
            \phi|\mrm{Tr}^J(N,N_gl_g)=0.
        \end{equation}
We next invoke Lemma~\ref{cosetj} with $M=N_gl_g$ to work with a convenient set of coset representatives pertaining to the $\mrm{Tr}^J$ operator, which gives
\begin{align}
    \sum_{\gamma\in \Gamma_0(N)\backslash \Gamma_0(N_gl_g)}\phi|(\gamma, [0,0])=0.
\end{align}
Next, note that for $\psmb a & b  \\ c & d \psme \in \sltwo$, one has
$\phi|\left(\smat{a}{b}{c}{d},[0,0] \right)= e(\frac{-cz^2}{c\tau+d})(c\tau+d)^{-k-1} \phi\left(\frac{a\tau+b}{c\tau+d}, \frac{z}{c\tau+d}  \right)$. Thus $\phi|\left(\smat{a}{b}{c}{d},[0,0] \right)|_{z=0}= \phi^\circ |\smat{a}{b}{c}{d}$. Thus we get
\begin{align}\label{phi0Tr}
    \sum_{\gamma\in \Gamma_0(N)\backslash \Gamma_0(N_gl_g)}\phi^\circ|\gamma=0.
\end{align}
Since $g\in S_{k+1}(N_g)$, we see that $g|W_N(l_g)\in S_{k+1}(N_gl_g)$. Further, $N_gl_g\neq N$ as $l_g$ is a proper divisor of $M_g$. Thus taking the inner product with $g|W_N(l_g)$ in \eqref{phi0Tr}, we get
\begin{align}
    0&=\sum_{\gamma\in \Gamma_0(N)\backslash \Gamma_0(N_gl_g)}\lan \phi^\circ|\gamma, g|W_N(l_g)\ran\\
    0&= [\Gamma_0(N_gl_g): \Gamma_0(N)] \lan \phi^\circ, g|W_N(l_g)\ran.
\end{align}
This completes the proof of Lemma \ref{phi0-ortho}.\qed

The stage is now set for the proof of the main result of this section.
\subsection{Proof of Theorem \ref{th:IPequiv}} \label{th:IPequiv-proof}
Recall the inner product $\lan F^\circ, g\otimes g|B_{M}\ran$ from  \eqref{FggBd'}.
\begin{align}
    \lan F^\circ, g\otimes g|B_{M}\ran &=\sum_{d_g|N_g^\infty}\sum_{\ell_g|M_g^\infty} (d_g\ell_g)^{-(k+1)/2} C(d_g\ell_g, \phi)\lan g_N|B_{d_g\ell_g}, g|B_M\ran.
\end{align}
Thus using \eqref{BLBdEval} and \eqref{CdEval}, we get
\begin{align} \label{FBM-prelim}
    \lan F^\circ, g\otimes g|B_{M}\ran&=\frac{\lan g, g\ran}{\zeta^{N_g}(2)}  \lan \phi^\circ, g|W_N(M_g)\ran\sum_{d_g|N_g^\infty}\sum_{\ell_g|M_g^\infty}\frac{\lambda_g(d_g)\lambda_g(\ell_g)}{(d_g\ell_g)^{(k+1)/2}}\\
    &\times\frac{\lambda_g(d_g)}{d_g^{(k+1)/2}}\sum_{d_2,d_3|M_g, (d_2,d_3)=1}\frac{\mu(d_3)\lambda_g(d_3)}{d_2d_3^{(k+1)/2}} \times\underset{p^t||\ell_gM_gd_2^2d_3}{\prod_{p|(M_g/M)}}\eta(p,t) \underset{p^t||\ell_gM_gd_2^2d_3}{\prod_{p|M}}\eta(p,t-1).
    \end{align}
Notice that in \eqref{BLBdEval}, there are three products over $p$, however there are only two in \eqref{FBM-prelim}. This is because the last product was over all $p|M$ such that $p\nmid \ell_gM_gd_2^2d_3$. However since $M|M_g$, there are no such primes.

Since $d_g|N_g$, we have $\lambda_g(d_g)^2=d_g^{k-1}$ (cf. \cite[p.~72]{ILS}). The $d_g$ sum therefore contributes 
$\sum_{d_g}d_g^{-2}= \zeta^{N_g}(2)$.

Thus $\lan F^\circ, g\otimes g|B_{M}\ran$ can now be written as
\begin{align}
    \lan F^\circ, g\otimes g|B_{M}\ran
    &= \lan g,g\ran \lan \phi^\circ, g|W_N(M_g)\ran\underset{(d_2,d_3)=1}{\sum_{d_2,d_3|M_g }}\sum_{\ell_g|M_g^\infty}\mu(d_3)\frac{\lambda_g(\ell_g)\lambda_g(d_3)}{d_2(\ell_gd_3)^{(k+1)/2}} \\
    &\times\underset{p^t||\ell_gM_gd_2^2d_3}{\prod_{p|(M_g/M)}}\eta(p,t) \underset{p^t||\ell_gM_gd_2^2d_3}{\prod_{p|M}}\eta(p,t-1).
\end{align}
Let $D(\ell_g, M_g)$ denote the sum over $d_2, d_3$. That is,
\begin{equation}
    D(\ell_g, M_g):=\underset{(d_2,d_3)=1}{\sum_{d_2,d_3|M_g }} \mu(d_3)\frac{\lambda_g(d_3)}{d_2d_3^{(k+1)/2}}\times\underset{p^t||\ell_gM_gd_2^2d_3}{\prod_{p|(M_g/M)}}\eta(p,t) \underset{p^t||\ell_gM_gd_2^2d_3}{\prod_{p|M}}\eta(p,t-1).
\end{equation}
We evaluate the Euler product for $D(\ell_g, M_g)$ in the following lemma.
\begin{lem} \label{dlgmg-lem}
    \begin{align}
   D(\ell_g,M_g) &=\prod_{p|(M_g/M), p^r||\ell_g} \left(\eta(p,r+1)+p^{-1} \eta(p,r+3) -p^{-(k+1)/2}\lambda_g(p)\eta(p,r+2) \right)\\
    & \times \prod_{p|M,p^r||\ell_g} \left(\eta(p,r)+p^{-1} \eta(p,r+2) -p^{-(k+1)/2}\lambda_g(p)\eta(p,r+1) \right).
\end{align}
\end{lem}
\begin{proof}[Proof of Lemma \ref{dlgmg-lem}]
    We prove by inducting on the number of prime factors of $M_g$. Let $M_g=p$. In this case, clearly
    \begin{equation}
        D(\ell_g, p)= \eta(p,r)+p^{-1} \eta(p,r+2) -p^{-(k+1)/2}\lambda_g(p)\eta(p,r+1),
    \end{equation}
where $r$ is such that $p^{r-1}|| \ell_g$ if $M=1$ and $p^r||\ell_g$ if $M=p$. Now assume the identity to hold when $M_g$ has $n-1$ prime factors. Now consider the case $M_gq$, where $(q,M_g)=1$ and $M_g$ has $(n-1)$ prime factors. Let $r$ be such that $q^{r-1}|| \ell_g$ if $q\nmid M$ and $q^r||\ell_g$ if $q|M$. Then we have
\begin{align}
    D(\ell_g, M_gq)&=\eta(q, r)\underset{(d_2,d_3)=1}{\sum_{d_2,d_3|M_g }} \frac{\mu(d_3)\lambda_g(d_3)}{d_2d_3^{(k+1)/2}}\times\underset{p^t||\ell_gM_gd_2^2d_3}{\prod_{p|(M_g/M)}}\eta(p,t) \underset{p^t||\ell_gM_gd_2^2d_3}{\prod_{p|M}}\eta(p,t-1)\\
    &+ q^{-1}\eta(q, r+2)\underset{(d_2,d_3)=1}{\sum_{d_2,d_3|M_g }} \frac{\mu(d_3)\lambda_g(d_3)}{d_2d_3^{(k+1)/2}}\cdot\underset{p^t||\ell_gM_gd_2^2d_3}{\prod_{p|(M_g/M)}}\eta(p,t) \underset{p^t||\ell_gM_gd_2^2d_3}{\prod_{p|M}}\eta(p,t-1)\\
    &- q^{-(k+1)/2}\lambda_g(q)\eta(q,r+1)\underset{(d_2,d_3)=1}{\sum_{d_2,d_3|M_g }} \frac{\mu(d_3)\lambda_g(d_3)}{d_2d_3^{(k+1)/2}}\cdot\underset{p^t||\ell_gM_gd_2^2d_3}{\prod_{p|(M_g/M)}}\eta(p,t) \underset{p^t||\ell_gM_gd_2^2d_3}{\prod_{p|M}}\eta(p,t-1).
\end{align}
The induction now completes the proof of the Lemma by noting that the above reduces to
\begin{equation}
    D(\ell_g, M_gq)= D(\ell_g, M_g)\cdot \left(\eta(q,r)+q^{-1} \eta(q,r+2) -q^{-(k+1)/2}\lambda_g(q)\eta(q,r+1)\right).\qedhere
\end{equation}
\end{proof}
\subsubsection{Back to the proof of Theorem \ref{th:IPequiv}}
Next, we note the following Hecke relations (see \cite[Lemma 4.5.7 and eq. 4.5.14]{miyake2006modular}).
\begin{align}
    \lambda_g(1,p^2)&=\lambda_g(1,p)^2-(p+1)p^{k-1} \q\text{ for } (p,N_g)=1.\\
    \lambda_g(1, p^t)&= \lambda_g(1,p)\lambda_g(1,p^{t-1})-p^k\lambda_g(1,p^{t-2}) \q \text{ for } t\ge 3 \text{ and } (p,N_g)=1.\\
   \lambda_g(n)&=\sum_{d^2|n}
    d^{k-1}\lambda_g(1,n/d^2) \text{ for } (n,N_g)=1.\label{LnL1n}
\end{align}
Thus for $r\ge 1$, we have
\begin{align}
    \left(\eta(p,r)+p^{-1} \eta(p,r+2) -p^{-(k+1)/2}\lambda_g(p)\eta(p,r+1) \right)\\
    = \frac{(p-1)}{(p+1) p^{(k+1)(r+2)/2}}&\left(p^{k+1}\lambda_g(1,p^r)- \lambda_g(1,p^{r+2})\right).
\end{align}
For $r=0$,
\begin{equation}
\left(\eta(p,0)+p^{-1} \eta(p,2) -p^{-(k+1)/2}\lambda_g(p)\eta(p,1) \right)=  \frac{(p-1)}{(p+1) p^{k+1}}\left(p^{k}(p+1)- \lambda_g(1,p^{2})\right).  
\end{equation}
Let 
\begin{equation}\label{Agdef}
 A(g):=\prod_{p|M_g}\frac{(p-1)}{(p+1)} \cdot \lan g,g\ran \lan \phi^\circ, g|W_N(M_g)\ran.   
\end{equation} 
For $r\ge 1$, define 
\begin{align}
L_p(0,g)&:= \left(p^{k}(p+1)- \lambda_g(1,p^{2})\right)\label{L0def}\\
L_p(r,g)&:= \left(p^{k+1}\lambda_g(1,p^r)- \lambda_g(1,p^{r+2})\right)\q \text{ for } r\ge 1.
\end{align}   
Thus $\lan F^\circ, g\otimes g|B_{M}\ran$ can now be written as
\begin{align} \label{main-bridge}
 \lan F^\circ, g\otimes g|B_{M}\ran
&= A(g) \prod_{p|M}\sum_{r\ge 0}\frac{\lambda_g(p^r) L_p(r, g)}{p^{(r+1)(k+1)}}\times \prod_{p|(M_g/M)}\frac{1}{p^{(k+1)/2}}\sum_{r\ge 0}  \frac{ \lambda_g(p^r)L_p(r+1, g)}{p^{(r+1)(k+1)}}.
\end{align}
Using the M\"obius inversion in \eqref{LnL1n}, for $(n,N_g)=1$ we write
\begin{equation}
    \lambda_g(1,n)=\sum_{d^2|n}\mu(d)d^{k-1}\lambda_g(n/d^2).
\end{equation}
Then $L_p(r,g)$ can  be written as below.
\begin{align}
    L_p(0,g)&= p^{k+1}+p^k+p^{k-1}- \lambda_g(p^{2})\\
L_p(r,g)&= p^{k+1}\lambda_g(p^r)- \lambda_g(p^{r+2})-p^{k-1}\Big(p^{k+1}\lambda_g(p^{r-2})- \lambda_g(p^{r})\Big)\q \text{ for } r\ge 1.
\end{align}
Here we follow the convention that $\lambda_g(p^{-r})=0$ for $r>0$. In the following calculations, we also use the convention $L_p(r,g)=0$ for $r<0$.

For a $p|M_g$, we now relate $L_p(r, g)$ and $L_p(r+1, g)$. In this direction, we have the following lemma.
\begin{lem}
Let $L_p(r,g)$ be as above. Then for $t\ge 0$, we have the following relation.
\begin{equation}\label{Lprecurrence}
    \lambda_g(p)L_p(t,g)=L_p(t+1, g)+p^k L_p(t-1, g).
\end{equation}
\end{lem}
\begin{proof}
Using the Hecke relations, for $r\ge 2$, we get
\begin{align}
&\lambda_g(p)\left(p^{k+1}\lambda_g(p^r)- \lambda_g(p^{r+2})-p^{2k}\lambda_g(p^{r-2})+ p^{k-1}\lambda_g(p^{r})\right)\\
&\q=p^{k+1}(\lambda_g(p^{r+1})+p^k\lambda_g(p^{r-1}))-(\lambda_g(p^{r+3})+p^k\lambda_g(p^{r+1}))\\
&\q\q-p^{2k}(\lambda_g(p^{r-1})+p^k\lambda_g(p^{r-3}))+p^{k-1}(\lambda_g(p^{r+1})+p^k\lambda_g(p^{r-1})) .
\end{align}
This gives us the required relations and completes the proof. 
\end{proof}
Let us define
\begin{equation}
    S(g,p):= \sum_{r\ge 0}p^{-(r+1)(k+1)}\lambda_g(p^r) L_p(r, g).
\end{equation}

\begin{lem}
   \begin{equation}
       \sum_{t=0}^{r}\frac{\lambda_g(p^r)L_p(r,g)}{p^{(r+1)(k+1)}}=\left(1+\frac{1}{p}\right)\sum_{t=0}^{r}\frac{1}{p^t}+\frac{\lambda_g(p^r)^2}{p^{r(k+1)+2}}-\frac{\lambda_g(p^r)\lambda_g(p^{r+2})}{p^{(r+1)(k+1)}}
   \end{equation}
and thus
\begin{equation}\label{SgpExp}
    S(g,p)=\frac{p+1}{p-1}.
\end{equation}
\end{lem}
\begin{proof}
    We prove the first assertion by induction on $r$. When $r=0$, the identity clearly holds by \eqref{L0def}. Next, we have
    \begin{align}
     \sum_{t=0}^{r}\frac{\lambda_g(p^r)L_p(r,g)}{p^{(r+1)(k+1)}}=\left(1+\frac{1}{p}\right) \sum_{t=0}^{r-1}\frac{1}{p^t}+  \frac{\lambda_g(p^{r-1})^2}{p^{(r-1)(k+1)+2}}-\frac{\lambda_g(p^{r-1})\lambda_g(p^{r+1})}{p^{r(k+1)}}+\frac{\lambda_g(p^r)L_p(r,g)}{p^{(r+1)(k+1)}}.
    \end{align}
Then $ \frac{\lambda_g(p^{r-1})^2}{p^{(r-1)(k+1)+2}}-\frac{\lambda_g(p^{r-1})\lambda_g(p^{r+1})}{p^{r(k+1)}}+\frac{\lambda_g(p^r)L_p(r,g)}{p^{(r+1)(k+1)}}$ is given by
\begin{align}\label{Indstepr}
 \frac{\lambda_g(p^{r-1})^2}{p^{(r-1)(k+1)+2}}-\frac{\lambda_g(p^{r-1})\lambda_g(p^{r+1})}{p^{r(k+1)}}+\frac{\lambda_g(p^r)^2}{p^{r(k+1)}}-\frac{\lambda_g(p^r)\lambda_g(p^{r+2})}{p^{(r+1)(k+1)}}-\frac{\lambda_g(p^r)\lambda_g(p^{r-2})}{p^{(r-1)(k+1)+2}}+\frac{\lambda_g(p^r)^2}{p^{r(k+1)+2}}.  
\end{align}
Next, we note that
\begin{align}
    \lambda_g(p^r)^2-\lambda_g(p^{r-1})\lambda_g(p^{r+1})&=p^k \left(\lambda_g(p^{r-1})^2 -\lambda_g(p^{r-2})\lambda_g(p^{r}) \right)\\
    &= p^{(r-1)k}(\lambda_g(p)^2-\lambda_g(p^2))\\
    &= p^{rk}.
\end{align}
The proof now follows as \eqref{Indstepr} simplifies to
\begin{equation}
    \frac{1}{p^r}+\frac{1}{p^{r+1}}+\frac{\lambda_g(p^r)^2}{p^{r(k+1)+2}}-\frac{\lambda_g(p^r)\lambda_g(p^{r+2})}{p^{(r+1)(k+1)}}.
\end{equation}
The second assertion follows by letting $r\longrightarrow \infty$ and noting that $\lambda_g(p^r)\ll rp^{rk/2}$.
\end{proof}
As a consequence, we have the following decomposition of the inner product.

\begin{thm}\label{Jacobipullbackrel}
 Let $F$ be the SK lift of $f\in S_{2k}^{new}(N)$. For any proper divisor $N_g|N$, let $g\in S_{k+1}^{new}(N_g)$ and $M_g=N/N_g$. Then
\begin{equation}\label{FggBMgFinal}
\lan F^\circ, g\otimes g|B_{M_g}\ran =\lan g,g\ran \lan \phi^\circ, g|W_N(M_g)\ran.   
\end{equation}
\end{thm}

\begin{proof}
    The proof follows by noting that
    \begin{equation}
        \lan F^\circ, g\otimes g|B_{M_g}\ran =A(g)\prodd_{p|M_g}S(g,p)
    \end{equation}
and considering the expression for $S(g,p)$ from \eqref{SgpExp}.
\end{proof}
Now consider any proper divisor $M|M_g$. Then for any $p|(M_g/M)$ we have
\begin{align}
    \sum_{r\ge 0}  \frac{ \lambda_g(p^r)L_p(r+1, g)}{p^{(r+1)(k+1)}}&=\frac{L_p(1, g)}{p^{k+1}}+\frac{ \lambda_g(p)L_p(2, g)}{p^{2(k+1)}}+ \lambda_g(p)\sum_{r\ge 2}\frac{ \lambda_g(p^r)L_p(r, g)}{p^{(r+1)(k+1)}}\\
    &\q -p^k\sum_{r\ge 2}\frac{ \lambda_g(p^r)L_p(r-1, g)}{p^{(r+1)(k+1)}}.
\end{align}
Next,
\begin{align}
\sum_{r\ge 2}\frac{ \lambda_g(p^r)L_p(r, g)}{p^{(r+1)(k+1)}}&= \sum_{r\ge 0}\frac{ \lambda_g(p^r)L_p(r, g)}{p^{(r+1)(k+1)}} -  \frac{L_p(0,g)}{p^{k+1}}- \frac{\lambda_g(p)L_p(1,g)}{p^{2(k+1)}}\\
&=S(p,g)- \frac{L_p(0,g)}{p^{k+1}}-  \frac{\lambda_g(p)L_p(1,g)}{p^{2(k+1)}}.  
\end{align}
Similarly $\sum_{r\ge 2}\frac{ \lambda_g(p^r)L_p(r-1, g)}{p^{(r+1)(k+1)}}$ is given by
\begin{align}
 &=  \lambda_g(p)\sum_{r\ge 2}\frac{ \lambda_g(p^{r-1})L_p(r-1, g)}{p^{(r+1)(k+1)}} - p^k\sum_{r\ge 2}\frac{ \lambda_g(p^{r-2})L_p(r-1, g)}{p^{(r+1)(k+1)}}\\
 &=\frac{\lambda_g(p)}{p^{k+1}}\sum_{r\ge 1}\frac{ \lambda_g(p^{r-1})L_p(r-1, g)}{p^{r(k+1)}}- \frac{\lambda_g(p)L(0,p)}{p^{2(k+1)}}-\frac{1}{p^{k+2}}\sum_{r\ge 0}\frac{ \lambda_g(p^{r})L_p(r+1, g)}{p^{(r+1)(k+1)}}\\
 &= \frac{\lambda_g(p)}{p^{k+1}}S(p,g)- \frac{\lambda_g(p)L(0,p)}{p^{2(k+1)}}-\frac{1}{p^{k+2}}\sum_{r\ge 0}\frac{ \lambda_g(p^{r})L_p(r+1, g)}{p^{(r+1)(k+1)}}. 
\end{align}
Thus we get
\begin{align}
    (1-\frac{1}{p^{2}})\sum_{r\ge 0}  \frac{ \lambda_g(p^r)L_p(r+1, g)}{p^{(r+1)(k+1)}}&=\frac{L_p(1, g)}{p^{k+1}}+\frac{ \lambda_g(p)L_p(2, g)}{p^{2(k+1)}}+\lambda_g(p)S(p,g)- \frac{\lambda_g(p)L_p(0,g)}{p^{k+1}}\\
    &-\frac{\lambda_g(p)^2L_p(1,g)}{p^{2(k+1)}}-\frac{\lambda_g(p)}{p}S(p,g)+\frac{\lambda_g(p)L(0,p)}{p^{k+2}}.
\end{align}
Using $\eqref{Lprecurrence}$ for $L_p(2,g)$ and  $L_p(1,g)$, we get
\begin{align}
 (1-\frac{1}{p^{2}})\sum_{r\ge 0}  \frac{ \lambda_g(p^r)L_p(r+1, g)}{p^{(r+1)(k+1)}}&= \frac{L_p(1, g)}{p^{k+1}}- \frac{\lambda_g(p)L_p(0,g)}{p^{k+1}}+\lambda_g(p)S(g,p)\left(1-\frac{1}{p}\right)\\
 &=\lambda_g(p)S(g,p)\left(1-\frac{1}{p}\right).
\end{align}
Thus we get
\begin{align}
 \sum_{r\ge 0}  \frac{ \lambda_g(p^r)L_p(r+1, g)}{p^{(r+1)(k+1)}}=\left(1-\frac{1}{p}\right)^{-1} \lambda_g(p).  
\end{align}
Substituting back in \eqref{main-bridge}, we get
\begin{equation}\label{FggBdFinal}
  \lan F^\circ, g\otimes g|B_{M}\ran =A(g)\prod_{p|M}\left(\frac{p+1}{p-1}\right) \prod_{p|(M_g/M)}\left(1-\frac{1}{p}\right)^{-1}\frac{\lambda_g(p)}{p^{(k+1)/2}}.
\end{equation}
Comparing it with the expression for $\lan F^\circ, g\otimes g|B_{M_g}\ran$ from \eqref{FggBMgFinal}, we get that
\begin{equation}
 \lan F^\circ, g\otimes g|B_{M}\ran = \frac{M^{(k+1)/2}\lambda_g(M_g/M)}{M_g^{(k+1)/2}}\prod_{p|(M_g/M)}\left(1+\frac{1}{p}\right)^{-1}\cdot  \lan F^\circ, g\otimes g|B_{M_g}\ran.
\end{equation}
This completes the proof of Theorem \ref{th:IPequiv}. \qed

\subsection{Proof of Theorem \ref{thm:CVOldclass}} \label{thm2.3proof}
The first two parts of the theorem, that is \eqref{cv-00} and \eqref{cv-0}, follow from Proposition \ref{cv-00-lem} and Corllary \ref{F0old-vanish}, respectively. 

For the third part of the theorem, 
let $g\in S_{k+1}^{new}(N_g)$ and $N=N_gM_g$. When $M_g\in \mc L_f$,  for any $\sigma$ we have
\begin{align}
    \lan F^\circ, g_\sigma\otimes g_\sigma\ran&=\sum_{d_1,d_2|M_g}\sigma(d_1)\sigma(d_2)\lan F^\circ, g|W_N(d_1)\otimes g|W_N(d_2)\ran\\
    &=2^{\omega(M_g)}\sum_{d|M_g} \sigma(d)\lan F^\circ, g\otimes g|W_N(d)\ran\\
    &=2^{\omega(M_g)}\lan F^\circ, g\otimes g|B_{M_g}\ran\left(\sum_{d|M_g}\sigma(d)\frac{d^{(k+1)/2}\lambda_g(M_g/d)}{M_g^{(k+1)/2}}\prod_{p|(M_g/d)}\left(1+\frac{1}{p}\right)^{-1}\right)\\
    &= 2^{\omega(M_g)}\sigma(M_g)\lan F^\circ, g\otimes g|B_{M_g}\ran\left(\sum_{d|M_g}\frac{\sigma(d)\lambda_g(d)}{d^{(k+1)/2}\prod_{p|d}(1+1/p)}\right).
\end{align}
On the other  hand, we have
\begin{align}
\lan g_\sigma, g_\sigma\ran &= 2^{\omega(M_g)}\sum_{d|M_g}\sigma(d)\lan g, g|W_N(d)\ran
= 2^{\omega(M_g)}\sum_{d|M_g}\frac{\sigma(d) \lambda_g(d)}{d^{(k+1)/2}\prod_{p|d}(1+1/p)}\lan g, g\ran.
\end{align}
Thus we get
\begin{equation}\label{sigma-BMrel}
 \frac{\lan F^\circ, g_\sigma\otimes g_\sigma\ran}{\lan g_\sigma ,g_\sigma \ran} = \sigma(M_g)  \frac{\lan F^\circ, g\otimes g|B_{M_g}\ran}{\lan g ,g \ran}.
\end{equation}
This relation combined with Theorem \ref{th:IPequiv} gives us the required relation between the inner products and the central values as in Theorem \ref{thm:CVOldclass}.\qed

As a corollary of Theorem \ref{Jacobipullbackrel}, we have the following central value formula for the pullback of Jacobi forms.

\begin{cor}
 Let $F$ be the SK lift of $f\in S_{2k}^{new}(N)$ and $\phi$ be the Jacobi form associated with $F$ via the EZI map. For any proper divisor $N_g|N$, let $g\in S_{k+1}^{new}(N_g)$ and $M_g=N/N_g$. Then
 \begin{align}
     \Lambda(f\otimes \mrm{sym}^2 g, \frac{1}{2})= \frac{2^{k+1-\omega(M_g)}M_g^{7/4}}{N} \prod_{p|N_g}(p+1)^2\frac{\langle f,f \rangle}{\langle h,h\rangle} 
 |\lan \phi^\circ, g|W_N(M_g)\ran|^2 &\text{ whenever }   M_g \in \mc L_f. \label{cvJacobi}  
 \end{align}
\end{cor}
We can say something interesting about the modular form $\phi^\circ$ from our calculations. 
\begin{cor} \label{phi-new}
Let $\phi=\phi_f$ be the EZI lift of a newform $f \in S_{2k}(N)$. Then,
\begin{align}
   \phi^\circ&=   \sum_{L|N}2^{\omega(L)}\sum_{g\in \mc B_{k+1}^{new}(N/L)}\lan \phi^\circ, g|B_L\ran g|B_L.
    \end{align}
\end{cor}

\begin{proof}
Using the orthogonal basis from section \ref{deg1-old-basis},    first we write
    \begin{align}
      \phi^\circ=   \sum_{L|N}\sum_{g\in \mc B_{k+1}^{new}(N/L)}\sum_{\sigma}c_{g_\sigma} \frac{ g_\sigma}{\lan g_\sigma, g_\sigma \ran},
    \end{align}
where $c_{g_\sigma}= \lan \phi^\circ, g_\sigma\ran$. From Lemma \ref{phi0-ortho}, we get $c_{g_\sigma}=\sigma(L)\lan \phi^\circ, g|B_L\ran$. Thus we get
\begin{align}
  \phi^\circ&=   \sum_{L|N}\sum_{g\in \mc B_{k+1}^{new}(N/L)}\lan \phi^\circ, g|B_L\ran\sum_{\sigma}\sigma(L)  g_\sigma(\tau).
\end{align}
Now consider the sum over $\sigma$. We get
\begin{align}
\sum_{\sigma}\sigma(L)  g_\sigma= \sum_{\sigma}\sigma(L)  \sum_{d|L}\sigma(d)(g|W_N(d))= \sum_{\sigma} \sum_{d|L} \sigma(d)g|W_N(L/d) =2^{\omega(L)} g|W_N(L) . 
\end{align}
This proves the Lemma. 
\end{proof}

\begin{rmk}
    One can easily adapt the calculations of this paper to compute the periods $\lan (G|V)^\circ, g\otimes g$, where $G$ is a SK lift of level $M|N, M<N$ and $V$ is a level raising operator, in terms of $\lan G^\circ, g\otimes g \ran$, up to elementary factors. Here $g$ is new of level say $M'$. The period vanishes unless $M'|M$. Thus, one does not get new instances of central-value formulae from here. Some computations towards this can be found in \cite{casazza2023p}.
\end{rmk}

\section{\texorpdfstring{$L^2$}{a}-mass of the pullback}

In this section, we obtain the expression for the $L^2$ mass as in Theorem \ref{norm-thm}. When $N=1$ such an expression can be found in \cite{liu-young}, but the level aspect is more subtle.
Towards this, let $f \in S_{2k}^{new}(\Gamma_0(N))$ and $g \in S_{k+1}^{new}(\Gamma_0(N_g))$, with $N_g|N$, be two normalized newforms. Let $F$ denote the Saito-Kurokawa lift of $f$. First,  recall the definition of $N(F)$.
\begin{equation} \label{nf-def-proof}
   N(F):= \frac{v_2\lan F^\circ, F^\circ\ran}{v_1^2\lan F,  F\ran }.
\end{equation}
Next, from Ichino's pullback formula for square-free levels (see \cite{chen}, \cite[Remark~1.2]{PV}), we wish to express the RHS of \eqref{nf-def-proof} in terms of central $L$-values. When $N_g=N$, the pullback formula is given by (after renormalization)
\begin{equation}\label{CVnew}
    \Lambda(f\otimes \mrm{sym}^2 g, \frac{1}{2})= 2^{k+1}\prod_{p|N}(p+1)^2 \frac{\langle f,f \rangle}{\langle h,h\rangle} 
 \frac{|\langle F^\circ, g\otimes g\rangle|^2}{\langle g,g\rangle ^2}.
\end{equation}
When $N_g<N$,  write $N=N_gM_g$. If the Atkin-Lehner eigenvalue $w_f(p)$ of $f$ is $+1$ for all primes $p|M_g$ (that is $M_g\in \mc L_f$), then (see \cite{PVold}) 
\begin{equation}\label{old}
    \Lambda(f\otimes \mrm{sym}^2 g, \frac{1}{2})= \frac{2^{k+1-\omega(M_g)}M_g^{7/4}}{N} \prod_{p|N_g}(p+1)^2\frac{\langle f,f \rangle}{\langle h,h\rangle} 
 \frac{|\langle F^\circ, g\otimes g|B_{M_g}\rangle|^2}{\langle g,g\rangle ^2}.
\end{equation}

\subsection{Proof of Theorem \ref{norm-thm}}
Recall the orthogonal basis for the space $S_{k+1}(N)$.
\begin{equation}
    \mc B_{k+1}(N)= \mc B_{k+1}^{new}(N) \bigcup \left(\cup_{LM=N, L>1}\cup_{g\in \mc B_{k+1}^{new}(M)}\mc B_{k+1}(L, g)\right).
\end{equation}
Now,  we note that
\begin{equation}
    F^\circ(\tau, \tau')=\sum_{g\in \mc{B}_{k+1}(N)} \left\lan F^\circ, \frac{g}{\norm{g}}\otimes\frac{g}{\norm{g}}\right\ran \frac{g(\tau)}{\norm{g}}\frac{g(\tau')}{\norm{g}}.
\end{equation}
Recall the set  $ \mc L_f=\{ L|N: w_f(p)=+1\, \text{ for all }\, p|L, L\ne 1\}.$
Then from Theorem \ref{thm:CVOldclass} we have that 
\begin{equation}
    \lan F^\circ, g_\sigma\otimes g_\sigma\ran=0 \text{ whenever } g_\sigma\in \mc B_{k+1}(L,g) \text{ and } L\notin \mc L_f.
\end{equation}
Thus we have
\begin{align}
F^\circ(\tau, \tau') &=\sum_{g\in \mc{B}_{k+1}^{new}(N)} \left\lan F^\circ, \frac{g}{\norm{g}}\otimes\frac{g}{\norm{g}}\right\ran \frac{g(\tau)}{\norm{g}}\frac{g(\tau')}{\norm{g}}\\
&\q + \sum_{L\in \mc L_f}\sum_{g\in \mc B_{k+1}^{new}(N/L)}\sum_{\sigma}\left\lan F^\circ, \frac{g_\sigma}{\norm{g_\sigma}}\otimes\frac{g_\sigma}{\norm{g_\sigma}}\right\ran \frac{g_\sigma(\tau)}{\norm{g_\sigma}}\frac{g_\sigma(\tau')}{\norm{g_\sigma}}.  
\end{align}
As a result, we can write
\begin{align}
    N(F)&= \frac{v_2}{v_1^2}\sum_{g\in \mc{B}_{k+1}^{new}(N) } \frac{|\langle F^\circ, g\otimes g\rangle|^2}{\lan F, F \ran \lan g, g \ran ^2}+ \frac{v_2}{v_1^2}\sum_{L\in \mc L_f}\sum_{g\in \mc B_{k+1}^{new}(N/L)}\sum_{\sigma} \frac{|\langle F^\circ, g_\sigma \times g_\sigma\rangle|^2}{\lan F, F\ran \lan g_\sigma, g_\sigma \ran ^2}.
\end{align}
Let us denote the first and second terms above by $N(F)^{new}$ and $N(F)^{old}$, respectively. Then we have 
\begin{align}
    N(F)^{new}= \frac{v_2}{v_1^2}\frac{1}{2^{k+1}\prod_{p|N}(p+1)^2}\sum_{g\in \mc{B}_{k+1}^{new}(N)} \frac{\lan h, h\ran}{\lan f, f
    \ran \lan F, F \ran} \Lambda(f\otimes \mrm{sym}^2 g, \frac{1}{2})
\end{align}
and
\begin{equation}
 N(F)^{old}=\frac{v_2}{v_1^2}\sum_{L\in \mc L_f}\sum_{g\in \mc B_{k+1}^{new}(N/L)}\frac{4^{\omega(L)}}{2^{k+1}L^{7/4}\prod_{p|(N/L)}(p+1)^2}\frac{\lan h, h\ran}{\lan f, f
    \ran \lan F, F \ran} \Lambda(f\otimes \mrm{sym}^2 g, \frac{1}{2}).
\end{equation}
We also have the relation (from  \cite{das-anamby2} after suitable normalizations)
\begin{align} \label{petreln}
   \lan h, h \ran = \frac{1}{4^{k}}\frac{(4\pi)^{k+1} \pi^2  \zeta(2)^{-1}}{\Gamma(k+1) L(f, \frac{3}{2})} \cdot N \left(\prod\nolimits_{p|N}\frac{(p^2+1)}{(p-1)^{2}(p+1)} \right) \lan F, F \ran .
\end{align}
Thus we get
\begin{equation}
   N(F)^{new}=  \frac{v_2}{v_1^2}\cdot \frac{12 N \pi^{k+1}}{\Gamma(k+1) 2^{k}}\cdot \left(\prod_{p|N}\frac{(p^2+1)}{(p^2-1)^{2}(p+1)}\right)\sum_{g\in \mc{B}_{k+1}^{new}(N)} \frac{\Lambda(f\otimes \mrm{sym}^2 g, \frac{1}{2})}{\lan f, f\ran L(f,\frac{3}{2})}
\end{equation}
and
\begin{equation}
N(F)^{old}= \frac{v_2}{v_1^2}\cdot \frac{12 N \pi^{k+1}}{\Gamma(k+1) 2^{k}} \left(\prod_{p|N}\frac{(p^2+1)}{(p^2-1)^2}\right)\sum_{L\in \mc L_f}\sum_{g\in \mc B_{k+1}^{new}(\frac{N}{L})}\frac{4^{\omega(L)}\prod_{p|L}(p+1)^2}{L^{7/4}}\frac{\Lambda(f\otimes \mrm{sym}^2 g, \frac{1}{2})}{\lan f, f\ran L(f,\frac{3}{2})}.    
\end{equation}
One also has
\begin{align}
    \lan f, f\ran = \frac{2}{\pi}(4\pi )^{-2k} \Gamma(2k) N&\prod_{p|N}(p+1)^{-1} L( \mrm{sym}^2 f, 1)\q \text{( see  \cite[Lemma 2.5] {ILS})};\\
    \Lambda(f\otimes \mrm{sym}^2 g, 1/2)&=N^{3/4} N_g^{1/4}  \Gamma_{\mbb C}(2k)\Gamma_{\mbb C}(k)L(f\otimes \mrm{sym}^2 g, 1/2);\\
    \left(\prod\nolimits_{p|N}\frac{(p^2+1)}{(p^2-1)^{2}}\right)&= \frac{1}{N^2}\frac{\zeta_{(N)}(2)^3}{\zeta_{(N)}(4)}.
\end{align}
Thus we can write
\begin{equation}\label{NFnewExp}
    N(F)^{new}=  \frac{v_2}{v_1^2}\cdot \frac{24 \pi}{kN}\frac{\zeta_{(N)}(2)^3}{\zeta_{(N)}(4)}\sum_{g\in \mc{B}_{k+1}^{new}(N)}\frac{L(f\otimes \mrm{sym}^2 g, \frac{1}{2})}{L(\mrm{sym}^2 f, 1) L(f,\frac{3}{2})}.
\end{equation}
and
\begin{equation}\label{NFoldExp}
N(F)^{old}=\frac{v_2}{v_1^2}\cdot \frac{24 \pi}{kN}\frac{\zeta_{(N)}(2)^3}{\zeta_{(N)}(4)}\sum_{L\in \mc L_f}\frac{4^{\omega(L)}\prod_{p|L}(p+1)^2}{L^{2}}\sum_{g\in \mc B_{k+1}^{new}(N/L)}\frac{L(f\otimes \mrm{sym}^2 g, \frac{1}{2})}{L(\mrm{sym}^2 f, 1) L(f,\frac{3}{2})}.  
\end{equation}
Putting together \eqref{NFnewExp} and \eqref{NFoldExp}, we can write
\begin{align}
    N(F)&=\frac{v_2}{v_1^2}\cdot \frac{24 \pi}{kN}\frac{\zeta_{(N)}(2)^3}{\zeta_{(N)}(4)}\sum_{L\in \mc L_f\cup \{1\}}\frac{4^{\omega(L)}\prod_{p|L}(p+1)^2}{L^{2}}\sum_{g\in \mc B_{k+1}^{new}(N/L)}\frac{L(f\otimes \mrm{sym}^2 g, \frac{1}{2})}{L(\mrm{sym}^2 f, 1) L(f,\frac{3}{2})}.  
\end{align}
Introducing $\prod_{p|L}(1+w_f(p))^2$ to compensate for the $L\not\in \mc L_f\cup\{1\}$, we get the desired expression for $N(F)$. That is,
\begin{align}
    N(F)=\frac{v_2}{v_1^2}\cdot \frac{24 \pi}{kN}\frac{\zeta_{(N)}(2)^3}{\zeta_{(N)}(4)}\sum_{L|N}\frac{\prod_{p|L}(p+1)^2(1+w_f(p))^2}{L^{2}}\sum_{g\in \mc B_{k+1}^{new}(N/L)}\frac{L(f\otimes \mrm{sym}^2 g, \frac{1}{2})}{L(\mrm{sym}^2 f, 1) L(f,\frac{3}{2})}.
\end{align}

\subsection{The norm conjecture}\label{sec:norm-conj}
Now we follow \cite[Section 4]{conrey2005integral} to conjecture on the size of $N(F)$. Towards this, for any $L\in \mc L_f\cup\{1\}$, put $M=N/L$, $w_g= \frac{k}{2\pi^2}L(\mrm{sym}^2 g, 1)$ and consider
\begin{equation}
    N_L(F):= \sum_{g\in \mc B_{k+1}^{new}(M)}w_g^{-1}L(\mrm{sym}^2 g, 1)L(f\otimes \mrm{sym}^2 g, \frac{1}{2}).
\end{equation}

Next, we need the Petersson formula for newforms as in \cite{ILS}. Let $\kappa\ge 1$ be an even integer and let 
\begin{equation}
\Delta_{\kappa,N}(n_1, n_2)= \frac{\Gamma(\kappa-1)}{(4\pi)^{\kappa-1}[\sltwo:\Gamma_0(N)]}\sum_{f\in \mc B_{\kappa}(N)}\frac{\lambda_f(n_1)\lambda_f(n_2)}{\lan f, f \ran}    
\end{equation}
\begin{equation}
    \Delta^*_{\kappa,N}(n_1,n_2):=\frac{12\zeta(2)}{\kappa-1}\sum_{f\in \mc B_{\kappa}^{new}(N)}\frac{\lambda_f(n_1)\lambda_f(n_2)}{L(\mrm{sym}^2 f, 1)}.
\end{equation}
Then from the Petersson formula for newforms \cite[Prop. 2.8]{ILS}, for $(n_1,N)=1$ and $(n_2,N^2)|N$ we have the following expression.
\begin{align}\label{PetNew-old}
  \Delta^*_{\kappa,N}(n_1,n_2)= \sum_{LM=N}\frac{\mu(L)M}{\nu((n_2,L))}\sum_{\ell|L^\infty} \ell^{-1}\Delta_{\kappa,M}(n_1\ell^2, n_2).
\end{align}

From \eqref{Drichletseries},  writing $r=r_1r_2$, $\ell=a_1a_2^2$, $n=d_1d_2^2/d$ and $m=d_3d_4^2/d$, we have
\begin{equation}
   L(f\otimes \mrm{sym}^2 g, s)= \underset{a_1,a_2|L^\infty}{\sum_{a_1,a_2}}\underset{(dd_1d_2d_3d_4,N)=1}{\sum_{r_1,r_2,d, d_1,d_2,d_3,d_4}}\frac{\mu(d)\lambda_f(r_1r_2a_1a_2^2d_1d_2^2/d)\lambda_g(r_1^2a_1^2d_1^2)\lambda_g(d_3^2)}{(d^3r_1r_2a_1a_2^2d_1d_2^2d_3^2d_4^4)^s} 
\end{equation}
and
\begin{equation}\label{DSsym2}
    L(\mrm{sym}^2g,s) :=\zeta^{(M)}(2s)\sum_{n \ge  1}\frac{\lambda_g(n^2)}{n^s} = \zeta^{(M)}(2s)\zeta_{(M)}(1+s)\sum_{(n,M)=1}\frac{\lambda_g(n^2)}{n^s}.
\end{equation}
Using the approximate functional equation at $s=1/2+\alpha $ for $L(f\otimes \mrm{sym}^2 g, s)$ and using  \eqref{DSsym2} at $s=1+2\alpha$ for $L(\mrm{sym}^2g,s)$, $N_L(F)$ can now be formally written as
\begin{align}
\zeta^{(M)}(2+4\alpha)\zeta_{(M)}(2+2\alpha)\underset{a_1,a_2|L^\infty}{\sum_{a_1,a_2}}\underset{(dd_1d_2d_3d_4,N)=1}{\sum_{r_1,r_2,d, d_1,d_2,d_3,d_4}}&\frac{\mu(d)\lambda_f(r_1r_2a_1a_2^2d_1d_2^2/d)}{r_1(n^2d^3r_1r_2a_1a_2^2d_1d_2^2d_3^2d_4^4)^{\frac{1}{2}+\alpha}}\\
&\sum_{g}w_g^{-1}\lambda_g(n^2)\lambda_g(a_1^2d_1^2)\lambda_g(d_3^2)+\cdots
\end{align}
Following the recipe from \cite[Section~4.1]{conrey2005integral}, we use the Petersson over $g\in \mc B_{k+1}^{new}(M)$ and retain only the diagonal terms in \eqref{PetNew-old} to get
\begin{align}
S_f(\alpha):=\phi(M)\zeta^{(M)}(2+4\alpha)\zeta_{(M)}(2+2\alpha)&\underset{a_1,a_2|L^\infty}{\sum_{r_1,r_2,a_1,a_2}} \underset{(dd_1d_2d_3d_4,N)=1}{\sum_{n,d, d_1,d_2,d_3,d_4}}\\
&\q\times\underset{n^2c^2=a_1^2d_1^2d_3^2}{\sum_{c|(d_1^2,d_3^2)}}\frac{\mu(d)\lambda_f(r_1r_2a_1a_2^2d_1d_2^2/d)}{r_1(n^2r_1r_2d^3a_1a_2^2d_1d_2^2d_3^2d_4^4)^{\frac{1}{2}+\alpha}}+\cdots.  
\end{align}
Substituting for $n$ from $n^2c^2=a_1^2d_1^2d_3^2$, $S_f(\alpha)$ can be written as
\begin{align}
\phi(M)\zeta^{(M)}(2+4\alpha)\zeta_{(M)}(2+2\alpha) \underset{a_1,a_2|L^\infty}{\sum_{a_1,a_2}} \underset{(dd_1d_2d_3,N)=1}{\sum_{r_1,r_2,d, d_1,d_2,d_3}}\sum_{c|(d_1^2,d_3^2)}\frac{\mu(d)c^{1+2\alpha}\lambda_f(r_1r_2a_1a_2^2d_1d_2^2/d)}{r_1(r_1r_2d^3a_1^3a_2^2d_1^3d_2^2d_3^4)^{\frac{1}{2}+\alpha}}+\cdot\cdot.
\end{align}
Evaluating the sum over $d_4$, $S_f(\alpha)$ can be written as
\begin{align}
S_f(\alpha)=\phi(M)\zeta^{(N/L)}(2+4\alpha)\zeta_{(N/L)}(2+2\alpha)\zeta^{(N)}(2+4\alpha)\prod_{p|M}L_p(f,1/2+\alpha)L_p(f,3/2+\alpha)  T_f(\alpha),   
\end{align}
where
\begin{align}
    T_f(\alpha)&= \underset{a_1,a_2|L^\infty}{\sum_{a_1,a_2}} \underset{(dd_1d_2d_3,N)=1}{\sum_{d, d_1,d_2,d_3}}\sum_{c|(d_1^2,d_3^2)}\frac{\mu(d)\lambda_f(da_1d_1a_2^2d_2^2)}{(d^2a_1^3d_1^3a_2^2d_2^2d_3^4/c^2)^{\frac{1}{2}+\alpha}}+\cdots\\
    &= \zeta_{(L)}(2+2\alpha)\prod_{p|L}L_p(f,3/2+3\alpha)\underset{(dd_1d_2d_3,N)=1}{\sum_{d, d_1,d_2,d_3}}\sum_{c|(d_1^2,d_3^2)}\frac{\mu(d)\lambda_f(dd_1d_2^2)}{(d^2d_1^3d_2^2d_3^4/c^2)^{\frac{1}{2}+\alpha}}+\cdots.
\end{align}
Next, we note that the sum over $d, d_1,d_2,d_3$ has the Euler product
$\displaystyle \prod_{p\nmid N} T_{f,p}(\alpha)$,
where with $x=p^{-1/2-\alpha}$
\begin{equation}
    T_{f,p}(\alpha)= \sum_{e_1,e_2,e_3\ge 0}\sum_{e=0,1} \sum_{e_4\le \min(2e_1,2e_3)}(-1)^e\lambda_f(p^{e+e_1+2e_2})x^{2e+3e_1+2e_2+4e_3-2e_4}.
\end{equation}
Thus as in \cite[eq. 4.12]{liu-young}, we have
\begin{equation}
   T_{f,p}(\alpha)=\frac{1-x^8}{(1-x^2)(1-\alpha_p^2x^2)(1-\beta_p^2x^2)(1-\alpha_px^3)(1-\beta_px^3)}.
\end{equation}
Since $((1-x^2)(1-\alpha_p^2x^2)(1-\beta_p^2x^2))^{-1}$ is the $p$-th Euler factor of $L(\mrm{sym}^2f,s)$, we have 
\begin{align}
    T_f(\alpha)&=\zeta_{(L)}(2+2\alpha)\zeta^{(N)}(4+8\alpha)^{-1}\prod_{p|L}L_p(f,3/2+3\alpha)\prod_{p\nmid N}L_p(\mrm{sym}^2 f, 1+2\alpha) L_p(f, 3/2+3\alpha)\\
    &= \frac{\zeta_{(L)}(2+2\alpha)L(\mrm{sym}^2 f, 1+2\alpha)}{\zeta^{(N)}(4+8\alpha)\zeta_{(N)}(2+2\alpha)}\prod_{p|L}L_p(f,3/2+3\alpha)\prod_{p\nmid N} L_p(f, 3/2+3\alpha)
\end{align}
and consequently,
\begin{align}
 S_f(\alpha)&=\frac{\phi(M)\zeta^{(N/L)}(2+4\alpha)\zeta_{(N/L)}(2+2\alpha)\zeta^{(N)}(2+4\alpha)\zeta_{(L)}(2+2\alpha)}{\zeta^{(N)}(4+8\alpha)\zeta_{(N)}(2+2\alpha)}L(\mrm{sym}^2 f, 1+2\alpha)\\
 &\times \prod_{p|M}L_p(f,1/2+\alpha)L_p(f,3/2+3\alpha)\prod_{p|L}L_p(f,3/2+3\alpha)\prod_{p\nmid N} L_p(f, 3/2+3\alpha).  
\end{align}
Putting $\alpha=0$, we see that
\begin{equation}
    N_L(f)\sim 2\frac{\phi(N/L)\zeta(2)\zeta^{(N)}(2)\zeta_{(L)}(2)}{\zeta^{(N)}(4)\zeta_{(N)}(2)} L(f, 3/2) L(\mrm{sym}^2 f, 1) \prod_{p|(N/L)}L_p(f,1/2).
\end{equation}
Substituting back in the expressions for $N(F)^{new}$ and $N(F)^{old}$ from \eqref{NFnewExp} and \eqref{NFoldExp} respectively, we get that
{\allowdisplaybreaks
\begin{align}
 N(F)^{new}&=\frac{v_2}{v_1^2}\frac{12}{N\pi}\frac{\zeta_{(N)}(2)^3}{\zeta_{(N)}(4)}\frac{N_1(f)}{L(\mrm{sym}^2 f, 1) L(f,\frac{3}{2})}\\
 &\sim  \frac{\zeta(4)}{2\pi \zeta(2)}\cdot\frac{12}{N\pi}\frac{\zeta_{(N)}(2)^3}{\zeta_{(N)}(4)}\cdot 2\frac{\phi(N)\zeta(2)\zeta^{(N)}(2)}{\zeta^{(N)}(4)\zeta_{(N)}(2)}\prod_{p|N}L_p(f,1/2)\\
 &= \frac{2\phi(N)\zeta_{(N)}(2)}{N}\prod_{p|N}\frac{1}{(1+w_f(p)p^{-1})}.
\end{align}
For the old part, we get
\begin{align}
    N(F)^{old}&=\frac{v_2}{v_1^2}\frac{12}{N\pi}\frac{\zeta_{(N)}(2)^3}{\zeta_{(N)}(4)}\sum_{L\in \mc L_f}\frac{4^{\omega(L)}\prod_{p|L}(p+1)^2}{L^{2}}\frac{N_L(f)}{L(\mrm{sym}^2 f, 1) L(f,\frac{3}{2})} \\
    &\sim \frac{\zeta(4)}{2\pi \zeta(2)}\cdot\frac{12}{N\pi}\frac{\zeta_{(N)}(2)^3}{\zeta_{(N)}(4)}\sum_{L\in \mc L_f}\frac{4^{\omega(L)}\prod_{p|L}(p+1)^2}{L^{2}}\frac{2\phi(\frac{N}{L})\zeta(2)\zeta^{(N)}(2)\zeta_{(L)}(2)}{\zeta^{(N)}(4)\zeta_{(N)}(2)}\prod_{p|(N/L)}L_p(f,\frac{1}{2})\\
    &= \frac{2\phi(N)\zeta_{(N)}(2)}{N}\left(\prod_{p|N}L_p(f,1/2)\right)\sum_{L\in \mc L_f}\frac{4^{\omega(L)}\zeta_{(L)}(2)}{\phi(L)L^{2}}\prod_{p|L}\frac{(p+1)^2}{L_p(f,\frac{1}{2})}\\
    &=\frac{2\phi(N)\zeta_{(N)}(2)}{N}\prod_{p|N}\frac{1}{(1+w_f(p)p^{-1})}\sum_{L\in \mc L_f}\frac{4^{\omega(L)}\zeta_{(L)}(2)}{\phi(L)}\prod_{p|L}\left(1+\frac{1}{p}\right)^3.
\end{align}
Now considering $\mathcal L_f\cup \{1\}$, we can make the conjecture on the size of $N(F)$.
\begin{align}
    N(F)&\sim \frac{2\phi(N)\zeta_{(N)}(2)}{N}\prod_{p|N}\frac{1}{(1+w_f(p)p^{-1})}\sum_{L\in \mc L_f\cup \{1\}}\frac{4^{\omega(L)}\zeta_{(L)}(2)}{\phi(L)}\prod_{p|L}\left(1+\frac{1}{p}\right)^3\\
    &= \frac{2\phi(N)\zeta_{(N)}(2)}{N}\prod_{p|N}\frac{1}{(1+w_f(p)p^{-1})} \sum_{d|N}\frac{4^{\omega(d)}\zeta_{(d)}(2)}{\phi(d)}\prod_{p|d}\left(1+\frac{1}{p}\right)^3\frac{(1+w_f(p))^2}{4}\\
    &= \frac{2\phi(N)\zeta_{(N)}(2)}{N}\prod_{p|N}\frac{1}{(1+w_f(p)p^{-1})} \sum_{d|N}\frac{\zeta_{(d)}(2)}{\phi(d)} \prod_{p|d}\left(1+\frac{1}{p}\right)^3(1+w_f(p))^2.\label{NFConj-proof}
\end{align}

\section{Non-vanishing of pullbacks}
In this section we present several results which provide information about the non-vanishing of the pullbacks. For pullbacks of newforms, we expect none of them to vanish. In weight $2$, indeed none of them vanish, and for other weights we have partial results.

\subsection{The case of a single $f$} \label{kernel-nonvanish}
The analysis of the non-vanishing of the pullback of a SK lift $F$ is not so straightforward as in \cite{liu-young}. To deal with this, one usually works with Jacobi forms, which are intermediary between Siegel modular forms and elliptic modular forms. 
Let $D_j \colon J^{cusp}_{k+1,1}(N) \to S_{k+j}(N)$ be restriction of certain differential operators (see \cite{EZ}) defined by
\begin{align}
D_0(\phi)= \phi^\circ:=\phi(\tau,0); \q D_2(\phi) = \left(\frac{k}{2 \pi i} \frac{\partial^2 \phi}{\partial z^2} - 2 \frac{\partial \phi}  {\partial \tau} \right) (\tau,0).
\end{align}
Then it is known (cf. \cite{EZ}) that for $k \ge 1$, the following map is injective:
\begin{align} \label{d0d2}
    D_0 \oplus D_2 \colon J^{cusp}_{k,1}(N) \hookrightarrow S_k(N) \oplus S_{k+2}(N).
\end{align}
In level $1$ moreover this map is known to be an isomorphism -- something which is not true for higher levels.
In \cite{liu-young}, a convenient explicit isomorphism due to Skoruppa (cf. \cite{sko-comp}) going from $S_k(1) \oplus S_{k+2}(1) \to J^{cusp}_{k,1}$ for $k \ge 4$, along with the Taylor expansion of $\phi(\tau,z)$ around $z=0$ were utilized to show that at least $1/2$ in proportion of the newforms have non-vanishing pullbacks.

Our goal in this subsection is to investigate the same for higher levels. Here the description of $\ker D_0$ is rather subtle. For instance, a non-trivial result of Arakawa-B\"ocherer \cite{arakawa2003vanishing} says that when $k=2$, the map $D_0$ itself is  injective for $N$ square-free. Whereas for $k>2$, it is known that $D_0$  mostly fails to be injective, and also the map \eqref{d0d2} may not be an isomorphism. Thus one needs to argue differently than that in \cite{liu-young}. We provide a few approaches -- in the realm of integral weights -- especially aimed to deal with newforms.

Let $\mc W \colon \skkn \to S_{k+1}(N) \otimes S_{k+1}(N)$ be the pullback (or also known as the Witt) map. Thus $\mc W (F)=F^\circ$. We are interested in understanding the $\ker \mc W$. We observe the following.

\begin{lem}
    $\ker \mc W \simeq \ker D_0$ via the EZI map $\phi \mapsto F$.
\end{lem}

\begin{proof}
    This follows from our earlier calculations -- viz \eqref{f0phim} which shows that $F^\circ=0$ implies $\phi^\circ=0$; and \eqref{phidN}, \eqref{phimN} for the other direction. The lemma can also be seen directly from the definition of the Hecke-type operators $V_m$.
\end{proof}

As a consequence, we have the following generalization of \cite[Theorem 1.9]{liu-young} to levels. For this purpose, we define the proportion of non-vanishing as
\begin{align} \label{prop-def}
\alpha(k,N) := \frac{\#\{f\in \mc B_{2k}^{new}(N): F^\circ\neq 0\}}{\mrm{dim}(S_{2k}^{new}(N))} .
\end{align}
\begin{thm}
    Let $N$ be odd and square-free and $F$ denote the SK lift of $f\in S_{2}^{new}(N)$. Then $\alpha(1,N)=1$. That is, $N(F)\neq 0$ for all $f\in S_{2}^{new}(N)$.
\end{thm}
\begin{proof}
    From \cite{arakawa1999vanishing}, we know that $\ker(D_0) \simeq S_{k}(N, \bar{\omega})$, where $\omega$ arises from the multiplier system of $\eta^6$. Further, when $k=1$,  from \cite[Theorem 3.4]{arakawa2003vanishing}, $S_{1}(N, \bar{\omega})=\{0\}$. Thus $\ker \mc W =\{0\}$. Thus $F_f^\circ\neq 0$ for all $f\in S_{2}^{new}(N)$.
\end{proof}
For a general weight $k$, we have the following non-vanishing result.
\begin{thm}\label{Positiveprop}
   Let $k>2$ be odd and $p$ be an odd prime. Let $F_f$ denote the SK lift of $f\in S_{2k}^{new}(p)$. Then
   \begin{enumerate} 
       \item For $p=3$, $\alpha(k,3)> 1-\frac{1}{3}\left(\frac{k}{k-\lfloor \frac{2k}{3}\rfloor} \right)$. Thus, we get a positive proportion of non-vanishing of $N(F_f)$ for $3\nmid k$.
       
       \item For $p\equiv 1, 7 \bmod 12$, $\alpha(k,p)> 1- \frac{3}{5}\left(\frac{p+1}{p-1} \right)$.
       
       \item For $p\equiv 5,11\bmod 12$, $\alpha(k,p)>1- \frac{5}{9}\left(\frac{p+1}{p-1} \right)$ if $2k\equiv 1, 2 \bmod 3$. 
       
       If $3|k$, then $\alpha(k,p)>1- \frac{3(p+1)}{5(p-1)-8} $. 
   \end{enumerate}
\end{thm}

\begin{proof}
Since $\ker(D_0) \simeq S_{k}(N, \bar{\omega})$, where $\omega$,  
by general results about modular forms with (finite order) multiplier systems, we know (see \cite[Theorem 4.2.1]{rankin1977modular}) that $\dim S_{k}(N, \bar{\omega}) \le k [\sltwo: \Gamma_0(N)]/12$. Therefore, 
via the multiplicity-one property of the newforms, at least
\begin{align} \label{weak-ineq}
    \dim (\skkn^{new}) - \dim (\skkn^{new} \cap \ker \mc W)  & \ge \dim  (S_{2k}^{new}(N))  - \dim S_{k}(N, \bar{\omega}).
\end{align}
new $f$ exist for which $N(F_f) \neq 0$. Now, since $k$ is odd, using the dimension formula for $S_{2k}^{new}(N)$ from \cite[Theorem 1]{martin2005dimensions} (with notations as in loc.cit), we have
\begin{align}
    \dim(S_{2k}^{new}(p))=\frac{(2k-1)(p-1)}{12}-\frac{1}{4}v_2^\#(p)+ c_3(k)v_3^\#(p).
\end{align}
The inequalities in (1), (2), (3) follow by explicitly calculating the quantities $v_2^\#(p)$, $v_3^\#(p)$ and $c_3(k)$.
\end{proof}
\begin{rmk}
\begin{enumerate}
    \item  In particular, $\alpha(k,p)> \frac{1}{7}$ uniformly for all primes $p>5$ and $k\ge 3$. Perhaps this proportion can be improved via the theory of $L$-functions, but this seems to be a rather difficult proposition at the moment. See the discussion on the `error-terms'.
    
    
    \item By a closer inspection of the dimension formula from \cite{martin2005dimensions}, one can get a better proportion of $f\in S_{2k}^{new}(p)$ for which $N(F_f)\neq 0$. For example, as $k\rightarrow \infty$ we get that $\alpha(k,13) \rightarrow \frac{5}{12}$.
    
    \item From the same arguments as in Theorem \ref{Positiveprop}, we also get that $\alpha(k, N)> 1- \frac{3}{5}\prod_{p|N}\left(\frac{p+1}{p-1} \right)$, when $2k\equiv 1,2\bmod 3$ or when all the prime factors of $N$ are $\equiv 1, 7\bmod 12$. However, when $N$ has a large number of prime factors, $\alpha(k,N)$ may be negative, and thus doesn't give anything useful.

    \item 
    Of course, we have used the weakest possible bound in the inequality \eqref{weak-ineq}. The subsequent subsection contains some possibilities for improvement in this regard.
\end{enumerate}  
\end{rmk}


\subsection{An alternative approach} We want to work with the more familiar spaces of modular forms with integral weights and Dirichlet characters, rather than the spaces of modular forms $M_{k}(N, \bar{\omega})$ with $\eta$-multipliers, which are not very well-studied. Let $F_f$ be an SK lift of $f\in S_{2k}^{new}(N)$ and $\phi:=\phi_f\in J_{k+1,1}^{new}(N)$ be the corresponding Jacobi form, that is $F_f$ is the EZI-lift of $\phi_f$ (as in the notation of \cite{das-anamby2}). Then we can write from the theta-decomposition of $\phi$ that (see \cite[Section~6.1]{das-anamby2})
\begin{align}\label{phi=htheta}
 \phi(\tau, z) &= h_0(\tau)\theta_0(\tau,z)+h_1(\tau)\theta_1(\tau,z) \\
    &=\frac{1}{2}\left(h\left( \frac{\tau-1}{4}\right)\left( \theta_0(\tau, z)-i\theta_1(\tau,z)\right)+h\left( \frac{\tau+1}{4}\right)\left( \theta_0(\tau, z)+i\theta_1(\tau,z)\right)\right),
\end{align}
where $h\in S_{k+1/2}^{+}(4N)$. Now there are two approaches. Put $\displaystyle \vartheta_0(\tau)=\sum_{n}q^{n^2}$.

\subsubsection{Proof of Theorem~\ref{nonv-algo}}\label{sec:nonv-algo}
Define the map $\Theta:S_{k+1/2}^{+}(4N)\longrightarrow  S_{k+1}(4N)$; \q $\Theta(h):= h \cdot \vartheta_0$.
Let us denote the image $ \Theta(S_{k+1/2}^{+}(4N)) \subset S_{k+1}(4N)$ by $S^*_{k+1}(4N)$.

We note that $\phi\in \ker D_0$ iff $\Theta(h)(\tau)=-\Theta(h)(\tau+1/2)$. This follows by from \eqref{phi=htheta} by making a change of variable $\tau\mapsto 4\tau+1$ and noting that $\theta_0(4\tau+1,0)=\theta_0(4\tau,0)=\vartheta_0(\tau)$, $
\theta_1(4\tau+1,0)=i\theta_1(4\tau,0)=\vartheta_0(\tau+1/2)$.
Thus  $\phi\in \ker D_0$ iff the following equation holds.
 \begin{align}
     h\left( \tau\right)\vartheta_0(\tau)&= - h\left(\tau+ \frac{1}{2}\right)\vartheta_0(\tau+ \frac{1}{2}).
\end{align}
We put $H=\Theta(h)$. We see that,
 \begin{align}
     F_f^\circ=0 \iff D_0(\phi)=0 \iff \Theta(h)(\tau)=-\Theta(h)(\tau+1/2) \iff H(\tau) + H(\tau+1/2)=0.
 \end{align}
Now $\displaystyle H(\tau) + H(\tau+1/2)= \sum_{n \ge 1} (1+(-1)^n) a_H(n)q^n = 2\sum_{n \ge 1, n \text{ even}} a_H(n)q^n = 2 H|U_2B_2$. Therefore we get Theorem~\ref{nonv-algo} since (with the notation as in \eqref{ezi-corr})
 \begin{align}
     F^\circ=0 \iff H|U_2B_2 = 0 \iff H|U_2=0.
 \end{align}
 and $H|U_2B_2$ is the even part of $H$. This proves the first part of Theorem~\ref{nonv-algo}. The rest follows simply by noting that the cusp form $H|U_2B_2$ has level $4N$, and applying the Sturm's bound at $p=\infty$ to $H$.
 \QEDB

Theorem~\ref{nonv-algo} can come quite handy in checking the non-vanishing of pullbacks, especially for newforms $h$, and of course we expect that no newform $\phi$ can be in the $\ker D_0$.

In the corollary below, we demonstrate a few examples where Theorem~\ref{nonv-algo} immediately gives non-vanishing of $F_f^\circ$. First, recall that $h\in S_{k+1/2}^+(4N)$ has Fourier coefficients supported on $n\equiv 0, 3 \bmod 4$.
 
\begin{cor} \label{algo-cor}
Let $f$, $h$ and $F_f$ be as in Theorem~\ref{nonv-algo}. Then we have the following.
\begin{enumerate}
\item 
If $h(\tau ) = cq^{4m} +   O(q^{4m+3})$ with $c \neq 0$, then $F_f^\circ \neq  0$.

\item 
If $h(\tau ) = cq^{4m-1} + dq^{4m} + O(q^{4m+3})$ and the coefficient $2c + d$ of $q^{4m}$ in $H$ is non-zero,
then $F_f^\circ \neq  0$.

\item 
If the Fourier expansion of $h$ is supported only on odd indices (necessarily on $n \equiv -1 \bmod 4$),
then $F_f^\circ \neq  0$.

\end{enumerate}
\end{cor}
\begin{rmk}
    From the above, it is enough to consider those newforms $h$ which look like 
$h(\tau ) = c ( q^{4m-1} - 2 q^{4m} + \ldots)$ with
$c \neq  0$.
\end{rmk}

\subsubsection{Another formulation} \label{k16n}
Analogously to the previous subsection, we now provide another interesting formulation of the non-vanishing of the pullback $F_f^\circ$. Let us consider the map $\Theta_1 : S_{k+1/2}^{+}(4N) \to S_k(16N, \chi_{-4})$ defined by $\Theta_1(h) = h/\vartheta_1$, where $\vartheta_1(\tau) = \vartheta_0(\tau+1/2)=\sum_{n \in \z} (-1)^n q^{n^2}$. Put $H_1:= \Theta_1(h)$.
We claim that:
\begin{align} \label{h-s**}
     \phi \in  \ker D_0 \iff  \Theta_1(h+1/4)= - i \Theta_1(h) \iff a_{H_1}(n) =0 \text{ unless } n \equiv 3 \bmod 4.
\end{align}

To see the claim, let $\phi \in  \ker D_0$ -- so that $\displaystyle 0 = h_0\theta_0+h_1\theta_1$ and following \cite{arakawa2003vanishing}, define $\varphi$ by
\begin{align}
    \varphi(\tau) = \frac{h_0(\tau)}{\theta_0(\tau)} = - \frac{h_1(\tau)}{\theta_1(\tau)}.
\end{align}
Next from \cite{arakawa1999vanishing}, we observe that, $\displaystyle\varphi(\tau+1)=-i \varphi(\tau)$, which implies, upon iteration, that $\displaystyle \varphi(\tau+4)= \varphi(\tau)$.
It is easy to see then  $\varphi_*(\tau)$ as defined below, satisfies
\begin{align}
 \varphi_*(\tau):= \varphi(4\tau)=\frac{h(\tau)}{\theta_0(4\tau)-\theta_1(4\tau)}= -\frac{h\left(\tau+ \frac{1}{2}\right)}{\theta_0(4\tau)+\theta_1(4\tau)}.   
\end{align}
We write them as
\begin{align} \label{htheta1}
  \varphi_*(\tau)=\frac{h(\tau)}{\vartheta_1(\tau)}= -\frac{h\left(\tau+ \frac{1}{2}\right)}{\vartheta_0(\tau)} \q \q (\vartheta_1(\tau) \text{ as above}),
\end{align}
which settles the first part of the claim. Now \eqref{htheta1} also implies that $\varphi_* (\tau+1)=\varphi_*(\tau)$. Therefore if we write the Fourier expansion of $\varphi_*(\tau)$ as $\displaystyle \varphi_*(\tau)= \sum_{n \ge 0} a_n q^n$,
then we get that 
\begin{align}
    \varphi(\tau) = \sumn_{n \ge 0} a_n q^{n/4},
\end{align}
which upon using the property $\displaystyle\varphi(\tau+1)=-i \varphi(\tau)$ of $\phi$, shows that $\displaystyle a_n=0, \text{   unless   } n \equiv 3 \bmod 4$, proving the claim, since the above steps are reversible. \QEDB

We expect that the above relation \eqref{h-s**} is not possible for newforms $h$. One might even consider working out certain Shimura images of $H_1 \cdot \vartheta_1=h$, and there are some results in the literature on this topic, all of which, however, assume that $H_1$ to be a Hecke eigenform at all primes. So this doesn't immediately give us anything.

\subsection{Non-vanishing on average}\label{sec:avgnonvan}
In this subsection, we are concerned with the average order of the quantities $N(F_f)$, as $f$ traverses over newforms of level $N$. As the reader might notice (cf. \cite{bky}) that this amounts to proving an asymptotic formula for $\sum_f\sum_g L(\mrm{sym}^2 g\otimes f, 1/2)$. A standard approach (see \cite{bky}) is  tracing the path Petersson over $f\rightarrow \mrm{GL(3)}$ Voronoi $\rightarrow$ Petersson over $g$. However, the level aspect $\mrm{GL(3)}$ Voronoi summation formulas are not readily available, even for square-free $N$. General results of Corbett \cite{corbett2021voronoi} seems not easy to translate to the present situation when the technical condition `$(c,N)>1$' arises. The ones available in the literature (for example \cite{zhou2018voronoi}) do not serve our purpose, since we need the `$p|c$' case. The partial Voronoi summation as in \cite{buttcane20154} also doesn't give the required power saving error terms. We hope to come back to this issue later.

We partially get around this difficulty by relating this $L^2$-norm problem to the same about the $L^\infty$-norm problem (cf. \cite{das-anamby1} for similar ideas) -- this at least allows us to get the lower bound (although not optimal) for the $L^2$-norm problem. To be precise, we show the following.

\begin{prop} \label{avg-nonzero}
    For an odd, square-free $N$, we have
    \begin{align}
        \frac{1}{N}\ll \sum\nolimits_{f\in \mc B_{2k}^{new}(N)}N(F_f).
    \end{align}
    Further, when $N=p$, an odd prime, we have
    \begin{align}
         \frac{1}{p}\ll \sum\nolimits_{f\in \mc B_{2k}^{new}(p)}N(F_f)\ll p^{1+\epsilon}.
    \end{align}
\end{prop}
\begin{proof}
    First, we recall that
    \begin{equation}
    F_f^\circ(\tau, \tau')=\sum\nolimits_{g\in \mc{B}_{k+1}(N)}c_g\frac{g(\tau)}{\norm{g}}\frac{g(\tau')}{\norm{g}},
\end{equation}
where $c_g= \left\lan F_f^\circ, \frac{g}{\norm{g}}\otimes\frac{g}{\norm{g}}\right\ran $. Thus we have
\begin{align} \label{F0gg}
    \frac{v_2 |(vv')^{(k+1)/2}F_f^\circ(\tau,\tau')|^2 }{v_1^2\lan F_f, F_f\ran} \le \frac{v_2 }{v_1^2} (\sum_g \frac{|c_g|^2 }{\lan F_f, F_f\ran}) \big( \sum_g \frac{v^{k+1} |g(\tau)|^2}{\norm{g}^2}\frac{v'^{k+1}|g(\tau')|^2}{\norm{g}^2} \big).
\end{align}
The term in the second braces above is bounded by $\sup(S_{k+1}(N)\otimes S_{k+1}(N))\ll \sup(S_{k+1}(N))^2 \ll_k N^2$. Therefore, we get
\begin{align}\label{pullbackNFLB}
    \frac{ |F_f^\circ(\tau,\tau')|^2 }{\lan F_f, F_f\ran} \ll  N^2 \cdot N(F).
\end{align}
Summing over $f$ first and then taking the supremum, we get
\begin{align}
\sup\big(\sum_{f}\frac{(vv')^{(k+1)/2} |F_f^\circ(\tau,\tau')|^2 }{\lan F_f, F_f\ran}\big)\ll N^2 \sum_f N(F_f). 
\end{align}
Since $F^\circ(\tau, \tau')=\sum_{n,m} \left(\sum_r a_F(n,r,m)\right)e(n\tau+m\tau')$, using the arguments as in \cite{das-anamby1}, we see that

\begin{align}
\sup\big(\sum_{f}\frac{(v_0v_0')^{(k+1)/2} |F^\circ(\tau,\tau')|^2 }{\lan F, F\ran}\big)\ge   \frac{(v_0v_0')^{k} }{e^{4\pi(v_0+v_0')}} \big(\sum_\phi\frac{|c_\phi(4)|^2}{\lan \phi, \phi\ran}+\sum_\phi\frac{|c_\phi(3)|^2}{\lan \phi, \phi\ran} -2 \big |\sum_\phi\frac{c_\phi(4)\overline{c_\phi(3)}}{\lan \phi, \phi\ran} \big |\big).  \n
\end{align}
Now using half-integral weight Poincar\'e series (see \cite[Proposition 4] {kohnen1985fourier}) and arguments as in \cite[Lemma 9.1]{das-anamby2} (after normalization of Petersson norm), we get that  the first and second terms are
$\gg N$. Similarly, the third term is $\ll N^\epsilon$ (by estimating the off-diagonal of the Poincar\'e series using \cite[Theorem 3.6]{das2012nonvanishing}). Now choosing $v_0=v_0'=k/4\pi$, we get the required lower bound. Therefore, we get for all $N$ odd, square-free
\begin{align}
    \sup\big(\sum_{f}\frac{ |F^\circ(\tau,\tau')|^2 }{\lan F, F\ran}\big) \gg N.
    \end{align}

For the second part, that is, when $N=p$, recall that $N(F_f)$ equals
\begin{align}
\frac{C(k,p)}{L(\mrm{sym}^2 f, 1) L(f,\frac{3}{2})}\big(\sum_{g\in \mc B_{k+1}^{new}(p)}L(f\otimes \mrm{sym}^2 g, \frac{1}{2})+ \frac{(1+p)^2(1+w_f(p))^2}{p^2}\sum_{g\in \mc B_{k+1}}L(f\otimes \mrm{sym}^2 g, \frac{1}{2}) \big).
\end{align}
When $g\in \mc B_{k+1}$, $L(f\otimes \mrm{sym}^2 g, \frac{1}{2})\ll p^{3/4+\epsilon}$ by convexity arguments. Thus, the sum over $g\in \mc B_{k+1}$ is $\ll p^{-1/4+\epsilon}$.  Thus it is enough to consider the sum over $g\in \mc B_{k+1}^{new}(p)$.
The upper bound now follows from \cite[Theorem~1.2]{buttcane20154}.
\end{proof}

\subsection{Non-vanishing of central $L$-values}
The non-vanishing of the pullback $F_f^\circ$ is intimately related to the non-vanishing of the central values $L(f\otimes \mrm{sym}^2 g, \frac{1}{2})$. When $k=2$, we get a nice unconditional result about the $L$-values. Note that the result below does not follow from those in \cite{chen, PV}.

\begin{prop}\label{CVAr-Bo}
   Let $f\in S_2^{new}(N)$ be such that the Atkin-Lehner eigenvalues at $p|N$, $w_f(p)=-1$. Then there exists a newform $g\in S_{2}^{new}(N)$ such that $L(f\otimes \mrm{sym}^2 g, \frac{1}{2})\neq 0$.
\end{prop}
\begin{proof}
    When $f$ is of weight $2$, the pullback of the corresponding Jacobi form $\phi^\circ$ is non-zero and thus $F_f^\circ\neq 0$. As a consequence, we get $N(F_f)\neq 0$.

    Now recall that
\begin{align} \label{nfnot0}
N(F_f)&= \frac{C(k,N)}{L(\mrm{sym}^2 f, 1) L(f,\frac{3}{2})}\left(\sum_{L|N}\frac{\prod_{p|L}(p+1)^2(1+w_f(p))^2}{L^{2}}\sum_{g\in \mc B_{2}^{new}(N/L)}L(f\otimes \mrm{sym}^2 g, \frac{1}{2})  \right).
\end{align}   
Now, since $w_f(p)=-1$ for all primes $p|N$, we have that
\begin{align}
N(F_f)&= \frac{C(k,N)}{L(\mrm{sym}^2 f, 1) L(f,\frac{3}{2})}\sum_{g\in \mc B_{2}^{new}(N)}L(f\otimes \mrm{sym}^2 g, \frac{1}{2}) .    
\end{align}
Thus we get a $g\in \mc B_{k+1}^{new}(N)$ such that $L(f\otimes \mrm{sym}^2 g, \frac{1}{2})\neq 0$.
\end{proof}

\begin{rmk} \label{rmk-bp1}
    Let us mention here that in \cite{pil-bo-yoshida}, the Yoshida liftings for level $N$ starting from pairs of elliptic newforms were studied, along with conditions which ensure their non-vanishing. The shapes of the $L$-functions of the Saito–Kurokawa liftings are the same as those of Yoshida liftings from pairs of a cusp form $f$ of weight $2k $ and the Eisenstein series of weight $2$. As noted by Ibukiyama \cite[p.~142]{Ibu-SK}, it does not seem clear if a Yoshida lifting of this type is exactly equal to the Saito–Kurokawa lifting in our sense, but anyway an Yoshida lifting of this type vanishes when $L(1/2, f ) = 0 $. Thus, if at all there is any relationship between the results of this paper with those in \cite{pil-bo-yoshida}, the result of this paper are distinct from loc. cit.. This is explicated further in the next remark.
\end{rmk}

\begin{rmk} \label{rmk-bp2}
    In \cite[Theorem 5.3]{arakawa2003vanishing}, a result similar to that in Theorem \ref{CVAr-Bo} is available with an additional condition of $L(f,1/2)\neq 0$.  \propref{CVAr-Bo} replaces the condition on $L(f, 1/2)$ by that on the Atkin-Lehner eigenvalues of $f$. This helps us to cover more cases than those in \cite[Theorem 5.3]{arakawa2003vanishing}. This can be illustrated by taking an odd, square-free $N$ with an even number of prime factors. In this case, with the condition on Atkin-Lehner eigenvalues as in \propref{CVAr-Bo}, we see that the functional equation of $L(f,s)$ is of sign $-1$ and thus $L(f, 1/2)=0$. We still get the non-vanishing of $L(f\otimes \mrm{sym}^2 g, \frac{1}{2})$ for some $g\in S_2^{new}(N)$ in this case, which is not covered by \cite[Theorem 5.3]{arakawa2003vanishing}.
\end{rmk}

\begin{rmk} \label{rmk-bp3}
    We can relax the condition on Atkin-Lehner eigenvalues in \propref{CVAr-Bo}. Its clear that the statement would then read: ``Let $f\in S_2^{new}(N)$. Then for some $L|N$, there exists a newform $g\in S_{2}^{new}(L)$ such that $L(f\otimes \mrm{sym}^2 g, \frac{1}{2})\neq 0$ and that $\omega_f(p)=+1$ for all $p|N/L$.''
\end{rmk}

\section{Appendix -- The main term for the average} \label{appendix}
{\footnotesize
In this section, we get the main term for the average $\sum_{f} N(F_f)$ that is in accordance with the main term predicted by the Conjecture \ref{NFConj}. Towards this, we consider the quantity
\begin{equation}
    \mc S:=\frac{12}{(2k-1)\phi(N)}\sum_{f\in \mc{B}_{2k}^{new}(N)}N(F_f).
\end{equation}
This corroboration is to be seen as the first step towards proving Conjecture~\ref{norm-conj}.

\subsection{Computation of the actual main term on average}
Recall the set of divisors $\mc L_f$ of $N$ from \eqref{Lf-def}. Writing $N(F_f)$ in terms of the central values from \eqref{NFnewExp} and \eqref{NFoldExp}, we see that
\begin{equation}
    \mc S=\frac{12C(k,N)}{(2k-1)\phi(N)}\sum_{f\in \mc{B}_{2k}^{new}(N)}\sum_{L\in \mc L_f\cup\{1\}}\frac{4^{\omega(L)}\prod_{p|L}(p+1)^2}{L^{2}}\sum_{g\in \mc B_{k+1}^{new}(N/L)}\frac{L(f\otimes \mrm{sym}^2 g, \frac{1}{2})}{L(\mrm{sym}^2 f, 1) L(f,\frac{3}{2})},
\end{equation}
where we write $C(k, N):= \frac{v_2}{v_1^2}\cdot \frac{24 \pi}{kN}\frac{\zeta_{(N)}(2)^3}{\zeta_{(N)}(4)}$. We can rewrite this as
\begin{equation}
\mc S= \frac{12C(k,N)}{(2k-1)\phi(N)}\sum_{f\in \mc{B}_{2k}^{new}(N)}\sum_{L|N}\frac{\prod\nolimits_{p|L}(p+1)^2(1+w_f(p))^2}{L^{2}}\sum_{g\in \mc B_{k+1}^{new}(N/L)}\frac{L(f\otimes \mrm{sym}^2 g, \frac{1}{2})}{L(\mrm{sym}^2 f, 1) L(f,\frac{3}{2})}.    
\end{equation}
We first note that
\begin{align}
    L(f, 3/2)^{-1}&= \prod\nolimits_{p|N}(1-\lambda_f(p)p^{-3/2})\prod\nolimits_{p\nmid N}(1-\lambda_f(p)p^{-3/2}+p^{-3})= \underset{(ab,N)=1}{\sum_{(a,b)=1}}\frac{\mu(a)\mu(b)^2\lambda_f(a)}{a^{3/2}b^3}\sum_{d|N}\frac{\mu(d)\lambda_f(d)}{d^{3/2}}. \label{L3/2N}
\end{align}
Next, we interchange the sums over $f$ and $g$. Then $S$ can be written as
\begin{equation}
    \mc S=\frac{12C(k,N)}{(2k-1)\phi(N)}\sum_{L|N}\frac{\prod_{p|L}(p+1)^2}{L^{2}}\sum_{g\in \mc B_{k+1}^{new}(N/L)}\sum_{f\in \mc{B}_{2k}^{new}(N)}\prod\nolimits_{p|L}(1+w_f(p))^2\frac{L(f\otimes \mrm{sym}^2 g, \frac{1}{2})}{L(\mrm{sym}^2 f, 1) L(f,\frac{3}{2})}.
\end{equation}
We have that
\begin{equation} \label{wp-expand}
    \prod\nolimits_{p|L}(1+w_f(p))^2= 2^{\omega(L)}\sum_{d|L}\mu(d)\lambda_f(d)d^{1/2}.
\end{equation}
Thus we can write
\begin{align}\label{SExp}
\mc S=\frac{12C(k,N)}{(2k-1)\phi(N)}\sum_{L|N}\frac{2^{\omega(L)}\prod\nolimits_{p|L}(p+1)^2}{L^{2}}&\sum_{g\in \mc B_{k+1}^{new}(N/L)}\underset{(ab,N)=1}{\sum_{(a,b)=1}}\frac{\mu(a)\mu(b)^2}{a^{3/2}b^3}\sum_{d_1|N}\sum_{d_2|L}\frac{\mu(d_1)\mu(d_2)d_2^{1/2}}{d_1^{3/2}}\\
&\times \sum_{f\in \mc{B}_{2k}^{new}(N)}\lambda_f(ad_1d_2)\frac{L(f\otimes \mrm{sym}^2 g, \frac{1}{2})}{L(\mrm{sym}^2 f, 1)} ,
\end{align}
where the $d_1$ variable comes from \eqref{L3/2N} and $d_2$ from \eqref{wp-expand}, respectively.

It would be convenient for us to have $d_1d_2$ to be square-free, in order to apply a suitable version of Petersson formula for the space of newforms. Towards this, next we note that
\begin{align}
\sum_{d_1|N}\sum_{d_2|L}\frac{\mu(d_1)\mu(d_2)d_2^{1/2}}{d_1^{3/2}} \lambda_f(d_1d_2)&= \prod\nolimits_{p|(N/L)} (1-\lambda_f(p)p^{-3/2}) \prod\nolimits_{p|L}(1-\lambda_f(p)p^{-3/2}) (1-\lambda_f(p)p^{1/2})\\
&=\prod\nolimits_{p|(N/L)} (1-\lambda_f(p)p^{-3/2}) \prod\nolimits_{p|L}(1+p^{-2})(1-\lambda_f(p)p^{1/2}) \\
&= \prod\nolimits_{p|L}(1+p^{-2})\sum_{d_1|(N/L)}\sum_{d_2|L}\frac{\mu(d_1)\mu(d_2)d_2^{1/2}\lambda_f(d_1d_2)}{d_1^{3/2}}.
\end{align}
For any $(s,N)=1$ and $r|N$ , we now consider the following sum over $f$.
\begin{equation}
    \mc M_g(s,r)= \frac{12}{(2k-1)\phi(N)}\sum_{f\in \mc{B}_{2k}^{new}(N)}\lambda_f(sr)\frac{L(f\otimes \mrm{sym}^2 g, \frac{1}{2})}{L(\mrm{sym}^2 f, 1)}.
\end{equation}
That is, we can write
\begin{align}\label{S-S_g}
    \mc S= C(k,N)\sum_{L|N}\psi(L)\sum_{g\in \mc B_{k+1}^{new}(N/L)}\underset{(ab,N)=1}{\sum_{(a,b)=1}}\frac{\mu(a)\mu(b)^2}{a^{3/2}b^3}\sum_{d_1|(N/L)}\sum_{d_2|L}\frac{\mu(d_1)\mu(d_2)d_2^{1/2}}{d_1^{3/2}} \mc M_g(a, d_1d_2),
\end{align}
where we put $\psi(L)=\frac{2^{\omega(L)}\prod_{p|L}(p+1)^2(p^2+1)}{L^{4}}$. Now recall that
\begin{equation}
     L(f\otimes \mrm{sym}^2 g, s)=\sum_{r|N_g^\infty}\frac{\lambda_f(r)A_g(r,1)}{r^s}\cdot \underset{\ell|M_g^\infty, (nm,N)=1}{\sum_{n,m,\ell}} \frac{\lambda_f(n\ell)A_g(n\ell, m)}{(n\ell m^2)^s}.
\end{equation}
Thus from the approximate functional equation, we have 
\begin{align}\label{SgAFE}
    \mc M_g(s,r)&= \frac{24}{(2k-1)\phi(N)}\sum_{f\in \mc{B}_{2k}^{new}(N)}\underset{(nm,N)=1}{\sum_{n,m\ge 1}}\sum_{\ell|L^\infty}\sum_{r|(N/L)^\infty}\sum_{s_1|(s,n)}\frac{A_g(n\ell,m)\lambda_f(r\ell)\lambda_f(sn/s_1^2)}{(nr\ell)^{1/2}mL(\mrm{sym}^2 f, 1)}W\left(\frac{m^2nr\ell}{N^2}\right),
\end{align}
where $W(x)= \frac{1}{2\pi i}\int_{(2)} x^{-u}\frac{\gamma_{f\otimes\mrm{sym^2}g}(1/2+u)}{\gamma_{f\otimes\mrm{sym^2}g}(1/2)}\frac{du}{u}$ and $W$ satisfies the growth properties as in \cite[Proposition 5]{iwaniec2021analytic}.

With our aim being to reduce to the Petersson formula for newforms from section \ref{sec:norm-conj}, we carry out the following simplifications. First, we note that for $r|N_g$, one has $A_g(r,1)= \sum_{d|r} \frac{1}{d} = \sum_{r_1r_2=r} \frac{1}{r_2}$. Further, since $|\lambda_f(p)|=p^{-1/2}$ for $p|N$, we have that
\begin{align}
    \sum_{r|(N/L)^\infty}\frac{\lambda_f(r)A_g(r,1)}{r^{1/2}}W\left(\frac{m^2nr\ell}{N^2}\right)&=\sum_{r_1,r_2|(N/L)^\infty}\frac{\lambda_f(r_1r_2)}{r_1^{1/2}r_2^{3/2}}W\left(\frac{m^2nr_1r_2\ell}{N^2}\right)\\
    &= \sum_{r_1,r_2|(N/L)}{\sum_{r_3,r_4|(N/L)^\infty}}\frac{\lambda_f(r_1r_2)}{r_3^2r_4^4r_1^{1/2}r_2^{3/2}}W\left(\frac{m^2n\ell r_1r_2r_3^2r_4^2}{N^2}\right).
\end{align}
For the $\ell|L^\infty$ sum, we note that
\begin{align}
    \sum_{\ell|L^\infty}\frac{A_g(\ell, 1)\lambda_f(\ell)}{\ell^{1/2}}&=\sum_{\ell_1|L,\ell_2|L^\infty}\frac{A_g(\ell_1\ell_2^2,1)\lambda_f(\ell_1)}{\ell_1^{1/2}\ell_2^{2}}.   
\end{align}
Substituting back in \eqref{SgAFE}, we get 
\begin{align}
     \mc M_g(s,r)&= \frac{24}{(2k-1)\phi(N)}\sum_{f\in \mc{B}_{2k}^{new}(N)}\underset{(nm,N)=1}{\sum_{n,m\ge 1}}\sum_{s_1|(s,n)}\sum_{\ell_1|L}\sum_{\ell_2|L^\infty}\frac{A_g(n\ell_1\ell_2^2,m)\lambda_f(r\ell_1)\lambda_f(sn/s_1^2)}{(n\ell_1)^{1/2}\ell_2^{2}mL(\mrm{sym}^2 f, 1)}\\
   &\times\sum_{r_1,r_2|(N/L)}{\sum_{r_3,r_4|(N/L)^\infty}}\frac{\lambda_f(r_1r_2)}{r_3^2r_4^4r_1^{1/2}r_2^{3/2}}W\left(\frac{m^2n\ell_1\ell_2^2 r_1r_2r_3^2r_4^2}{N^2}\right).
\end{align}

Now putting $\ell_3=(r,\ell_1)$, $\ell_1=\ell_3\ell_4$, $r_7=(r_1,r_2)$, $r_1=r_5r_7$ and $r_2=r_6r_7$ (note here that $(r_7,r_5r_6)=1$), we have that
\begin{equation}
    \sum_{
    \ell_1|L}\sum_{r_1,r_2|(N/L)}\frac{\lambda_f(r\ell_1r_1r_2)}{\ell_1^{1/2}r_1^{1/2} r_2^{3/2}}=\sum_{\ell_3|(r,L)}\underset{(\ell_4,r)=1}{\sum_{\ell_4|L}}\underset{(r_5,r_6)=1}{\sum_{r_5,r_6|(N/L)}}\sum_{r_7|(N/Lr_5r_6)} \frac{\lambda_f(\ell_4 \, r_5r_6 \, (r/\ell_3) /\, (r/\ell_3,r_5r_6)^2)}{\ell_3^{3/2}\ell_4^{1/2}(r/\ell_3,r_5r_6)r_7^3r_5^{1/2}r_6^{3/2}}.
\end{equation} 
Thus we can write
\begin{align}
     \mc M_g(s,r)&= \frac{24}{(2k-1)\phi(N)}\sum_{f\in \mc{B}_{2k}^{new}(N)}\underset{(nm,N)=1}{\sum_{n,m\ge 1}}\sum_{s_1|(s,n)}\underset{(\ell_4,r)=1}{\sum_{\ell_4|L}}\sum_{\ell_2|L^\infty}\sum_{r_3,r_4|(N/L)^\infty}\sum_{\ell_3|(r,L)}\underset{(r_5,r_6)=1}{\sum_{r_5,r_6|(N/L)}}\sum_{r_7|(N/Lr_5r_6)}\\
     &\frac{A_g(n\ell_3\ell_4\ell_2^2,m)\lambda_f(sn/s_1^2)\lambda_f( \ell_4 \, r_5r_6 \, (r/\ell_3) \, (r/\ell_3,r_5r_6)^2  ))}{n^{1/2}m\,\ell_2^{2} \ell_3^{3/2}\ell_4^{1/2}(r/\ell_3,r_5r_6)r_3^2r_4^4r_5^{1/2}r_6^{3/2}r_7^3\,L(\mrm{sym}^2 f, 1)}W\left(\frac{m^2n\ell_3\ell_4\ell_2^2r_3^2r_4^2r_5r_6r_7^2}{N^2}\right).
\end{align}
We note here that $ \ell_4 \, r_5r_6 \, (r/\ell_3) /\, (r/\ell_3,r_5r_6)^2$ is square-free and thus divides $N$.

Now coming back to $\mc M_g(s,r)$, using the Petersson formula for newforms from \eqref{PetNew-old} with 
\begin{align}
  n_1= sn/s_1^2,  \q n_2= \ell_4 \, r_5r_6 \, (r/\ell_3)/ \, (r/\ell_3,r_5r_6)^2,  
\end{align}
We see that the main terms coming from each of the quantities $\Delta(n_1 \ell^2, n_2)$ survive only when $\ell=1$ and $n=s =s_1$ and $r\ell_4r_5r_6/\ell_3(r/\ell_3,r_5r_6)^2=1$. This is because $n_1 \ell^2=n_2$ with $(n_1,N)=1$, $n_2|N$ and being square-free, implies that $\ell=1$. Further $n_1=n_2$ forces them to be equal to $1$.

Since $(\ell_4, rr_5r_6)=1$, the second equality holds iff $\ell_4=1$ and $r/\ell_3=r_5r_6$. We also note here that since $r_5r_6|N/L$, $(r/\ell_3, L)=1$. Thus we can write
\begin{equation}
    \mc M_g(s,r)= \mc M_g^{(1)}(s,r)+\mc M_g^{(2)}(s,r).
\end{equation}
The main term now corresponds to $\mc M_g^{(1)}(s,r)$ and is given by
\begin{align}
    \mc M_g^{(1)}(s,r)&= \frac{2}{\zeta(2)}\frac{1}{r^{3/2}}\underset{(m,N)=1}{\sum_{m\ge 1}}\sum_{\ell_2|L^\infty}\sum_{\ell_3|(r,L)}\sum_{r_3,r_4|(N/L)^\infty}\sum_{r_7|(\ell_3N/rL)}\sum_{r_6|(r/\ell_3)}\frac{A_g(s\ell_3\ell_2^2,m)}{s^{1/2}m\ell_2^{2}r_3^2r_4^4r_6r_7^3}W\left(\frac{sm^2r\ell_2^2r_3^2r_4^2r_7^2}{N^2}\right).
\end{align}

Let us denote by $\mc S^{(1)}$, the term in $\mc S$, corresponding to $\mc M_g^{(1)}(s,r)$. Before we proceed, need the Dirichlet polynomial approximation for $L(\mrm{sym}^2 h, 1)$ for any $h\in \mc{B}_{k+1}^{new}(M)$ from \cite[Lemme 3]{royer2001statistique}. To be precise, there exists $\delta_1, \delta_2, \delta_3 >0$ such that for all but $M^{\delta_3}$ forms in $\mc{B}_{k+1}^{new}(M)$, we have
\begin{equation}\label{L1sym2DP}
    L(\mrm{sym}^2 h, 1)= \sum_{b_1,b_2: (b_2, M)=1}\frac{\lambda_h(b_1^2)}{b_1b_2^2} \exp(-b_1b_2^2/M^{\delta_1})+O(M^{-\delta_2}).
\end{equation}

From \eqref{S-S_g}, we note that $s=a$ and $r=d_1d_2$. Since $d_1|(N/L)$, we can write the condition $\ell_3|(r, L)$ as $\ell_3|(d_2, L)$, which is equivalent to $\ell_3|d_2$ since $d_2|L$. Now write $d_2=\ell_3d_3$, then we have that $r/\ell_3= d_1d_3$. But since $(r/\ell_3, L)=1$, we must have that $d_3=1$ and consequently $\ell_3=d_2$. Further, the condition $r_6|r/\ell_3$ can now be written as $r_6|d_1$. Thus we have
\begin{align}
    \mc S^{(1)}&= \frac{2C(k,N)}{\zeta(2)}\sum_{L|N}\psi(L)\underset{(ab,N)=1}{\sum_{(a,b)=1}}\frac{\mu(a)\mu(b)^2}{b^3}\sum_{d_1|(N/L)}\sum_{d_2|L}\frac{\mu(d_1)\mu(d_2)}{d_1^{3}d_2}\\
    &\underset{(m,N)=1}{\sum_{m\ge 1}}\sum_{\ell_2|L^\infty}\sum_{r_3,r_4|(N/L)^\infty}\sum_{r_7|(N/d_1L)}\sum_{r_6|d_1}\sum_{g\in \mc B_{k+1}^{new}(N/L)}\frac{A_g(ad_2\ell_2^2,m)}{a^2m\ell_2^{2}r_3^2r_4^4r_6r_7^3}W\left(\frac{am^2d_1d_2\ell_2^2r_3^2r_4^2r_7^2}{N^2}\right).
\end{align}
Introducing $L(\mrm{sym}^2 g ,1)$ to aid the Petersson over $g$, we can now write
\begin{align}
    \mc S^{(1)}&= C(k,N)\sum_{L|N}\psi(L)\underset{(ab,N)=1}{\sum_{(a,b)=1}}\frac{\mu(a)\mu(b)^2}{a^{3/2}b^3}\sum_{d_1|(N/L)}\sum_{d_2|L}\frac{\mu(d_1)\mu(d_2)d_2^{1/2}}{d_1^{3/2}}\\
    &\times \sum_{g\in \mc B_{k+1}^{new}(N/L)}\frac{1}{L(\mrm{sym}^2 g ,1)}\mc M_g^{(1)}(a,d_1d_2)+ O(N^{\delta_3-1+\epsilon}) + O( N^{-\delta_2+\epsilon} ),
\end{align}
where the first error term comes from the exceptional newforms $g$ for which \eqref{L1sym2DP} may not be applicable, and the second error term accounts for the error in applying the approximation to the remaining $g$. We have used the rather crude bound $\# \mc B_{k+1}^{new}(N/L) < (k-1)/12 \cdot N + O(1) $, see e.g. \cite[Theorem 6]{martin2005dimensions}.

Let us denote the sum over $g$ by $\mc T(a, d_2,\ell_2,m)$. 
Since $(\ell_2d_2, am)=1$, we have the following multiplicative Hecke relation:
\begin{equation}
A_g(ad_2\ell_2^2,m)=A_g(d_2\ell_2^2,1)A_g(a,m).
\end{equation}
Using the $\mrm{GL}(3)$- Hecke relations from \eqref{GL3HeckRel}, we now have
\begin{align}
\mc T(a, d_2,\ell_2,m)&=\underset{(b_1b_2, N/L)=1}{\sum_{b_1,b_2}}\sum_{b_3|(N/L)^\infty}\sum_{c|(a,m)}\underset{a_3a_4^2=m/c}{\sum_{a_1a_2^2=a/c}}\sum_{a_5a_6^2=d_2\ell_2^2}\sum_{c_1|(a_1^2,a_3^2)}\frac{\mu(c)}{b_1b_2^2b_3^2}\exp(-\frac{b_1b_2^2b_3L^{\delta_1}}{N^{\delta_1}})\\
&\q\q\times \sum_{g\in \mc B_{k+1}^{new}(N/L)}\frac{\lambda_g(b_1^2)\lambda_g(a_1^2a_3^2a_5^2/c_1^2)}{L(\mrm{sym}^2 g, 1)}.\label{g-pet}
\end{align}
Corresponding to the main and error terms in the Petersson formula for newforms, we write $\mc T(a, d_2,\ell_2,m)= \mc T_1(a, d_2,\ell_2,m)   +\mc T_2(a, d_2,\ell_2,m)$
and correspondingly write $\mc S^{(1)}=\mc S^{(11)}+\mc S^{(12)}+O(N^{\delta_3-1+\epsilon}) + O( N^{-\delta_2+\epsilon} ) .$

The main term in the Petersson formula in \eqref{g-pet} survives only when $b_1=a_1a_3a_5/c_1$. Thus, we can write
\begin{equation}
\begin{split}
 \mc T_1(a, d_2,\ell_2,m)  &= \frac{k\phi(N/L)}{12\zeta(2)} \underset{(b_2, N/L)=1}{\sum_{b_2}}\sum_{b_3|(N/L)^\infty}\sum_{c|(a,m)}\sum_{a_1a_2^2=a/c}\sum_{a_3a_4^2=m/c}\sum_{a_5a_6^2=d_2\ell_2^2}\\
 &\q \times \sum_{c_1|(a_1^2,a_3^2)}\frac{\mu(c)c_1}{a_1a_3a_5b_2^2b_3^2}\exp(-\frac{a_1a_3a_5b_2^2b_3L^{\delta_1}}{c_1N^{\delta_1}}).
 \end{split}
\end{equation}
First, we make a change of variable $a=ca_0$ and $m=cm_0$. Next, putting $a_0=a_1a_2^2$, $m_0=a_3a_4^2$, we get
\begin{equation}
\begin{split}
\mc S^{(11)}&= \frac{kC(k,N)}{6\zeta(2)^2}\sum_{L|N}\phi(N/L)\psi(L)\underset{(ca_1a_3a_4,N)=1}{\sum_{c,a_1,a_3,a_4\ge 1}}\underset{(ca_1,b)=1}{\sum_{(b,N)=1}}\sum_{(b_2, N/L)=1}\sum_{b_3,r_3,r_4|(N/L)^\infty}\sum_{d_1|(N/L)}\sum_{r_7|(N/d_1L)}\sum_{\ell_2|L^\infty}\sum_{d_2|L}\sum_{r_6|d_1}\\
&\sum_{c_1|(a_1^2,a_3^2)}\sum_{a_5a_6^2=d_2\ell_2^2}\frac{\mu(b)^2\mu(c)\mu(ca_1d_1d_2)c_1}{b^3c^3a_1^3a_3^2a_4^2a_5b_2^2b_3^2d_1^{3}d_2\ell_2^{2}r_3^2r_4^4r_6r_7^3}\exp(-\frac{a_1a_3a_5b_2^2b_3L^{\delta_1}}{c_1N^{\delta_1}}) W\left(\frac{c^3a_1a_3^2a_4^4d_1d_3\ell_2^2\ell_3r_3^2r_4^2r_7^2}{N^2}\right)
\end{split}
\end{equation}
Using Mellin inversion, we get
\begin{equation}
\begin{split}
\mc S^{(11)}&=\frac{kC(k,N)}{6\zeta(2)^2}\sum_{L|N}\phi(N/L)\psi(L)\underset{(ca_1a_3a_4,N)=1}{\sum_{c,a_1,a_3,a_4\ge 1}}\underset{(ca_1,b)=1}{\sum_{(b,N)=1}}\sum_{(b_2, N/L)=1}\sum_{b_3,r_3,r_4|(N/L)^\infty}\sum_{d_1|(N/L)}\sum_{r_7|(N/d_1L)}\sum_{\ell_2|L^\infty}\sum_{d_2|L}\sum_{r_6|d_1}\\
&\sum_{c_1|(a_1^2,a_3^2)}\sum_{a_5a_6^2=d_2\ell_2^2}\int_{(1)}\int_{(1)}\frac{\mu(b)^2\mu(c)\mu(ca_1d_1\ell_3d_3)c_1^{1+v}N^{2u+\delta_1v}L^{-\delta_1v}\widetilde{W}(u)\Gamma(v)}{b^3c^{3+u}a_1^{3+u+v}a_3^{2+v}a_4^2a_5^{1+v}b_2^{2+2v}b_3^{2+v}d_1^{3+u}d_2^{1+u}\ell_2^{2+2u}r_3^{2+2u}r_4^{4+2u}r_6r_7^{3+2u}}\frac{du}{2\pi i} \frac{dv}{2\pi i}   
\end{split}
\end{equation}

Shifting the contours to (say) $\Re(u)=\Re(v)=-1/10$ and picking up the poles of $\widetilde{W}$ and $\Gamma$ at $u=v=0$, we get
\begin{align}
 \mc S^{(11)}&=\frac{kC(k,N)}{6\zeta(2)^2}\sum_{L|N}\phi(N/L)\psi(L)\underset{(ca_1a_3a_4,N)=1}{\sum_{c,a_1,a_3,a_4\ge 1}}\sum_{(ca_1,b)=(b,N)=1}\sum_{(b_2, N/L)=1}\sum_{b_3,r_3,r_4|(N/L)^\infty}\sum_{d_1|(N/L)}\sum_{r_7|(N/d_1L)}\\
&\sum_{\ell_2|L^\infty}\sum_{d_2|L}\sum_{r_6|d_1}\sum_{c_1|(a_1^2,a_3^2)}\sum_{a_5a_6^2=d_2\ell_2^2}\frac{\mu(b)^2\mu(c)\mu(ca_1d_1d_2)c_1}{b^3c^3a_1^3a_3^2a_4^2a_5b_2^2b_3^2d_1^{3}d_2\ell_2^{2}r_3^2r_4^4r_6r_7^3}+O(N^{-(2+\delta_1)/10+\epsilon}).
\end{align}
First we recall that $C(k, N):= \frac{v_2}{v_1^2}\cdot \frac{24 \pi}{kN}\frac{\zeta_{(N)}(2)^3}{\zeta_{(N)}(4)}=\frac{\zeta(4)}{\zeta(2)}\frac{12}{kN}\frac{\zeta_{(N)}(2)^3}{\zeta_{(N)}(4)}$. Now executing the $a_4, b_2, b_3, r_3,r_4$ sums, we get
\begin{equation}\label{M1Penult}
\begin{split}
\mc S^{(11)}&=\frac{2\zeta(4)}{\zeta(2)^3}\frac{\zeta_{(N)}(2)^3}{N\zeta_{(N)}(4)}\zeta(2)\zeta^{(N)}(2)\sum_{L|N}\phi(N/L)\psi(L)\zeta_{(N/L)}(2)\zeta_{(N/L)}(4)\underset{(ca_1a_3,N)=1}{\sum_{c,a_1,a_3\ge 1}}\sum_{(ca_1,b)=(b,N)=1}\\
&
\sum_{d_1|(N/L)}\sum_{r_7|(N/d_1L)}\sum_{\ell_2|L^\infty}\sum_{d_2|L}\sum_{r_6|d_1}\sum_{c_1|(a_1^2,a_3^2)}\sum_{a_5a_6^2=d_2\ell_2^2}\frac{\mu(b)^2\mu(c)\mu(ca_1d_1d_2)c_1}{b^3c^3a_1^3a_3^2a_5d_1^{3}d_2\ell_2^{2}r_6r_7^3}+O(N^{-(2+\delta_1)/10+\epsilon}).
\end{split}
\end{equation}
Now consider the sums over $c, a_1,a_3, b$. As in \cite{liu-young}, we have the following
\begin{align}
\underset{(ca_1a_3,N)=1}{\sum_{c,a_1,a_3\ge 1}}\sum_{(b,Nca_1)=1}\sum_{c_1|(a_1^2,a_3^2)}   \frac{\mu(b)^2\mu(c)\mu(ca_1)c_1}{b^3c^3a_1^3a_3^2}=\prod_{p\nmid N}\underset{0\le b+a_1+c\le 1}{\sum_{c,a_1,a_3,b}}\sum_{0\le c_1\le \min(2a_1,2a_3)}\frac{ (-1)^{a_1}p^{c_1}}{p^{3b+3c+3a_1+2a_3}}=\frac{\zeta^{(N)}(2)}{\zeta^{(N)}(4)}.
\end{align} 
Similarly, the sum over $d_1$ can we evaluated as below.
\begin{equation}
\sum_{d_1|(N/L)}\sum_{r_7|(N/d_1L)}\sum_{r_6|d_1} \frac{\mu(d_1)}{d_1^3r_6r_7^3}=  \prod\nolimits_{p|(N/L)}\left(1-\frac{1}{p^4}\right)=\frac{1}{\zeta_{(N/L)}(4)}.   
\end{equation}
For the sum over $d_2, \ell_2$, we have the following.
\begin{align}
\sum_{\ell_2|L^\infty}\sum_{d_2|L}\sum_{a_5a_6^2=d_2\ell_2^2}\frac{\mu(d_2)}{a_5d_2\ell_2^{2}}&=\sum_{d_2|L}\sum_{\ell_2|L^\infty}\sum_{a_6^2|d_2\ell_2^2}\frac{\mu(d_2)a_6^2}{d_2^2\ell_2^{4}}=\prod\nolimits_{p|L}\underset{\ell_2\ge 0}{\sum_{0\le d_2\le 1}}\sum_{0\le a_6\le (d_2+2\ell_2)/2}\frac{(-1)^{d_2}}{p^{2d_2+4\ell_2-2a_6}} \\
&=\prod\nolimits_{p|L}\left(1-\frac{1}{p^2}\right)
\prod\nolimits_{p|L}\frac{1}{(p^2-1)}\sum_{\ell_2\ge 0}\left(\frac{p^2}{p^{2\ell_2}}-\frac{1}{p^{4\ell_2}} \right) =\zeta_{(L)}(4).
\end{align}
Substituting back in \eqref{M1Penult}, we get
\begin{align}
    \mc S^{(11)}&=2\cdot N\phi(N) \left(\prod\nolimits_{p|N}\frac{1}{(p^2-1)}\right)\sum_{L|N}\frac{\psi(L)}{\phi(L)}\zeta_{(N/L)}(2)\zeta_{(N/L)}(4)\cdot \frac{\zeta_{(L)}(4)}{\zeta_{(N/L)}(4)} +O(N^{-(2+\delta_1)/10+\epsilon}).
\end{align}
Recall that $\psi(L)=\frac{2^{\omega(L)}\prod\nolimits_{p|L}(p+1)^2(p^2+1)}{L^{4}}$. Thus
\begin{align}
\mc S^{(11)}
&= \frac{2\zeta_{(N)}(2)^2\phi(N)}{N}\sum_{L|N}\frac{2^{\omega(L)}}{\phi(L)L^2}\prod\nolimits_{p|L}(p+1)^2+O(N^{-(2+\delta_1)/10+\epsilon}). \label{actual-avg}
\end{align}
The quantity $\mc S^{(11)}$ is the global main term we were after. One can show that it matches with the conjectural main term as given in Conjecture \ref{norm-conj}. We check this for $N=p$, a prime. The same holds for all square-free levels, which we have not written down, but can be done at the cost of some more technical calculations.

{\allowdisplaybreaks
\begin{align}
    \mc S
    &=\frac{24\zeta_{(p)}(2)}{(2k-1)p}\sum_{f\in \mc{B}_{2k}^{new}(p)}\frac{1}{(1+w_f(p)p^{-1})}\left(1+\frac{\zeta_{(p)}(2)(p+1)^3}{\phi(p)p^3}(1+w_f(p))^2\right)\\
    &=\frac{24\zeta_{(p)}(2)^2}{(2k-1)p}\sum_{f\in \mc{B}_{2k}^{new}(p)}(1+\lambda_f(p)p^{-1/2})\left(1+\frac{2\zeta_{(p)}(2)(p+1)^3}{\phi(p)p^3}(1-\lambda_f(p)p^{1/2})\right).
\end{align}
Introducing $L(\mrm{sym}^2 f, 1)$ and using the Dirichlet polynomial approximation for $L(\mrm{sym}^2 f, 1)$ from \eqref{L1sym2DP}, we note that
\begin{equation}
    \begin{split}
\mc S\sim\frac{24\zeta_{(p)}(2)^2}{(2k-1)p}\underset{(b_1b_2, p)=1}{\sum_{b_1,b_2}}\sum_{b_3|p^\infty}\frac{\exp(-b_1b_3b_2^2/p^{\delta_1})}{b_1b_2^2b_3^2} &\sum_{f\in \mc{B}_{2k}^{new}(p)}\frac{\lambda_f(b_1^2)(1+\lambda_f(p)p^{-1/2})}{L(\mrm{sym}^2 f, 1)}\\
&\q\times\left(1+\frac{2\zeta_{(p)}(2)(p+1)^3}{\phi(p)p^3}(1-\lambda_f(p)p^{1/2})\right).
    \end{split}
\end{equation}
Using Petersson for newforms and retaining only the main term, that is $b_1=1$, we get that
\begin{align}
\mc S&\sim\frac{2\zeta_{(p)}(2)^2\phi(p)}{\zeta(2)p}\sum_{(b_2, p)=1}\sum_{b_3|p^\infty}\frac{\exp(-b_3b_2^2/p^{\delta_1})}{b_2^2b_3^2}\left(1+2\frac{\zeta_{(p)}(2)(p+1)^3}{\phi(p)p^3}-2\frac{\zeta_{(p)}(2)(p+1)^3}{\phi(p)p^4} \right).   
\end{align}
We may essentially assume that $\exp(-b_3b_2^2/p^{\delta_1}) \approx 1$, and get that the  sums over $b_2,b_3$ are $\sim \zeta(2)$. This, of course, can be made precise either by writing in terms of Mellin inverses, or in this case, we can make a simple truncation at $b_3b_2^2 \le p^{\delta_1+\epsilon}$ with negligible error, for a small $\epsilon$. Then the 
the sum with $b_3b_2^2 \le p^{\delta_1-\epsilon}$ is seen to be $\zeta(2)+O(p^{-\epsilon})$ and the remaining terms bounded as $O(p^{-\epsilon}\sum_{b_1,b_2} b^{-(1+\epsilon)}_3b_2^{-(1+\epsilon) } )O(p^{-\epsilon})$ by choosing $\epsilon$ small enough.

Thus we get that
\begin{align} \label{heuristic-avg}
    \mc S&\sim\frac{2\zeta_{(p)}(2)^2\phi(p)}{p}\left(1+2\frac{\zeta_{(p)}(2)(p+1)^3}{\phi(p)p^3}-2\frac{\zeta_{(p)}(2)(p+1)^3}{\phi(p)p^4} \right)=\frac{2\zeta_{(p)}(2)^2\phi(p)}{p}\left(1+2\frac{(p+1)^2}{p^2\phi(p)} \right).
\end{align}}
Clearly, \eqref{heuristic-avg} matches with the main term of \eqref{actual-avg} when $N=p$.\qed
}

\printbibliography
\end{document}